\documentclass[12pt]{amsart}

\usepackage{amssymb}
\usepackage{bm}
\usepackage{graphicx}
\usepackage{psfrag}
\usepackage{enumerate}
\usepackage{color}
\usepackage{url}

\newcommand{\excise}[1]{}

\newtheorem{thm}{Theorem}[section]
\newtheorem{lemma}[thm]{Lemma}

\newtheorem{cor}[thm]{Corollary}
\newtheorem{prop}[thm]{Proposition}

\newtheorem{question}[thm]{Question}
\newtheorem{prob}[thm]{Problem}

\theoremstyle{definition}
\newtheorem{example}[thm]{Example}
\newtheorem{remark}[thm]{Remark}
\newtheorem{defn}[thm]{Definition}
\newtheorem{conv}[thm]{Convention}

\numberwithin{equation}{section}

\newcommand{\ring}[1]{\ensuremath{\mathbb{#1}}}

\renewcommand\>{\rangle}
\newcommand\<{\langle}

\newcommand\CC{\ring{C}}

\newcommand\NN{\ring{N}}

\newcommand\QQ{\ring{Q}}
\newcommand\RR{\ring{R}}

\newcommand\ZZ{\ring{Z}}

\newcommand\kk{\Bbbk}
\newcommand\mm{{\mathfrak m}}

\newcommand\pp{{\mathfrak p}}

\newcommand\ww{{\mathbf w}}
\newcommand\xx{{\mathbf x}}

\newcommand\oJ{{\hspace{.45ex}\overline{\hspace{-.45ex}J}}}
\newcommand\oQ{\hspace{.15ex}\ol{\hspace{-.15ex}Q\hspace{-.25ex}}\hspace{.25ex}}
\newcommand\ttt{\mathbf{t}}

\newcommand\app{\mathord\approx}

\newcommand\nil{\infty}
\newcommand\ott{\ol\ttt{}}
\newcommand\til{\mathord\sim}

\newcommand\Iqp{I_q{}^{\hspace{-.7ex}P}}
\newcommand\Irn{I_{\rho,\nothing}}
\newcommand\Irp{I_{\rho,P}}
\newcommand\Isp{I_{\sigma,P}}
\newcommand\Iwp{I_w{}^{\hspace{-1.2ex}P}}

\newcommand\rqp{\rho_q^{\hspace{.2ex}P}}
\newcommand\Iaug{I_{\mathrm{aug}}}

\newcommand\MipI{M_\infty^P(I)}
\newcommand\MrpI{M_\rho^P(I)}
\newcommand\MspI{M_\sigma^P(I)}
\newcommand\MwpI{M_\ww^P(I)}
\newcommand\Qlqp{Q_{\preceq q}^P}
\newcommand\Qlwp{Q_{\preceq w}^P}
\newcommand\WwpI{W_\ww^P(I)}

\newcommand\into{\hookrightarrow}
\newcommand\onto{\twoheadrightarrow}

\newcommand\rpqP{\rho_{p+q}^{\hspace{.2ex}P}}

\newcommand\minus{\smallsetminus}
\newcommand\nothing{\varnothing}

\renewcommand\iff{\Leftrightarrow}
\renewcommand\implies{\Rightarrow}

\def\ol#1{{\overline {#1}}}

\DeclareMathOperator\ann{ann} 
\DeclareMathOperator\Ass{Ass} 
\DeclareMathOperator\sat{sat} 

\begin{document}

\mbox{}
\title{Decompositions of commutative monoid congruences and binomial ideals\qquad}
\author{Thomas Kahle}
\address{Fakult\"at f\"ur Mathematik, Otto-von-Guericke Universit\"at\\D-39106 Magdeburg, Germany}
\email{http://www.thomas-kahle.de}
\author{Ezra Miller}
\address{Mathematics Department\\Duke University\\Durham, NC 27708\\USA}
\email{http://math.duke.edu/\~{\hspace{-.3ex}}ezra}

\makeatletter
  \@namedef{subjclassname@2010}{\textup{2010} Mathematics Subject Classification}
\makeatother
\date{13 May 2014}

\begin{abstract}
Primary decomposition of commutative monoid congruences is insensitive
to certain features of primary decomposition in commutative rings.
These features are captured by the more refined theory of
\emph{mesoprimary decomposition} of congruences, introduced here
complete with witnesses and associated prime objects.  The
combinatorial theory of mesoprimary decomposition lifts to arbitrary
binomial ideals in monoid algebras.  The resulting \emph{binomial
mesoprimary decomposition} is a new type of intersection decomposition
for binomial ideals that enjoys computational efficiency and
independence from ground field hypotheses.  Binomial primary
decompositions are easily recovered from mesoprimary decomposition.
\end{abstract}
\maketitle

\setcounter{tocdepth}{1}
\tableofcontents

\section{Introduction}

\subsection*{Overview}

Primary decomposition of ideals and modules has been a mainstay of
commutative algebra since Emmy Noether's unification of scattered
results roughly a century ago~\cite{noether}.  A formally analogous
theory for congruences on commutative monoids made its first
appearance around fifty years ago \cite{drbohlav}, and subsequently
the topic of decompositions has similarly played a central role in
commutative semigroup theory \cite{grillet}.  Our first goal is to
demonstrate that the formal analogy in the setting of finitely
generated monoids and congruences---the \emph{combinatorial
setting}---fails to capture the essence of primary decomposition in
noetherian rings and modules.  We justify this claim, and rectify it,
by exhibiting a more sensitive theory of \emph{mesoprimary
decomposition} of congruences, complete with witnesses, associated
prime objects, and other facets of control afforded in parallel with
primary decomposition in rings.  We then proceed beyond formal analogy
by lifting our witnessed theory of mesoprimary decomposition to the
\emph{arithmetic setting} of binomial ideals in semigroup rings, at
the interface of commutative ring theory with finitely generated
monoids.

Mesoprimary decomposition of binomial ideals is not binomial primary
decomposition, but a new type of intersection decomposition for
binomial ideals, with numerous advantages over ordinary primary
decomposition, such as combinatorial clarity, independence from
properties of the ground field, and computational efficiency.
Nevertheless, binomial primary decomposition is easily recovered from
mesoprimary decomposition.  In essence, by lifting mesoprimary
decomposition of congruences, binomial mesoprimary decomposition
distills the coefficient-free combinatorics inherent in primary
decomposition of binomial ideals and isolates the precise manner in
which coefficients subsequently determine the primary components.  The
subtlety of coefficient arithmetic causes the lifting procedure to
fail verbatim translation, particularly where redundancy is involved.
Part of our study therefore contrasts the slightly different notions
of witness and associatedness in the combinatorial and arithmetic
settings.

\subsection*{General motivation}
\enlargethispage{4.5ex}

The need for natural decompositions in the monoid and binomial
contexts has become increasingly important in recent years, in view of
appearances and applications in numerous areas.  Some of these
directly involve commutative monoids, such as schemes
over~$\mathbb{F}_1$ \cite{connes, deitmar}, where monoids form the
foundation just as rings do for usual schemes.  Another instance is
the arrival of mis\`ere quotients in combinatorial game theory, where
monoids provide data structures for recording and computing winning
strategies \cite{Pla05, misereQuots} (see also \cite{abelSurvey} for
an algebraic introduction).  At the same time, binomial ideals
interact with other parts of mathematics and the sciences, motivating
research into applicable descriptions of their decompositions.  For
example, dynamics of mass-action kinetics, where steady states in
detailed-balanced cases are described by vanishing of binomial
trajectories, arise from stoichiometric exponential growth and decay
\cite{mass-actionReview}; binomial decompositions in mass-action
kinetics can identify which species persist or become extinct
\cite{shiu-stu}.  In algebraic statistics, decompositions of binomial
ideals give insight into how a set of conditional independence
statements among random variables can be realized \cite{algstat,HHH}.
More generally, the connectivity of lattice point walks in polyhedra
can be analyzed using decompositions of binomial ideals
\cite{diaconisLattice,krs}.  These applications rely on decompositions
of \emph{unital} ideals---generated by monomials and differences of
monomials---into unital ideals; these are mesoprimary decompositions.
The algebra, geometry, and combinatorics of binomial primary
decomposition interacts with systems of differential equations of
hypergeometric type \cite{GGZ87, GKZ89}, whose solutions are
eigenfunctions for binomial differential operators encoding the
infinitesimal action of an algebraic torus.  In fact, it was in the
hypergeometric framework that the combinatorics of binomial primary
decomposition had its origin \cite{rankJumps, primDecomp, dmm},
providing tight control over series solutions.  In the meantime,
mesoprimary decomposition serves as an improved method for presenting
and visualizing binomial primary decomposition in algorithmic output
\cite{BinomialsM2}.  Beyond that, the methods here have already found
a theoretical application to combinatorial game theory
\cite{latticeGames,affineStrat}.

\subsection*{Conventions}

Unless otherwise stated, $Q$ denotes a finitely generated
(equivalently, noetherian) commutative monoid, and $\kk$ denotes an
arbitrary field.

\subsection*{Gathering primary components rationally}
\enlargethispage{5ex}

Staring at output of binomial primary decomposition algorithms
intimates that certain primary components belong together.

\begin{example}\label{e:8to1}
During investigations of presentations of mis\`ere quotients of
combinatorial games (culminating in the definition of lattice games
\cite{gmw, latticeGames}), Macaulay2~\cite{M2} produced long lists of
primary binomial ideals.  In one instance, eight of the components
were
$$%
\begin{array}{c@{\ }c}
  \<e - 1, d - 1, b - 1, a - 1, c^3\>,
& \<e - 1, d - 1, b - 1, a + 1, c^3\>,
\\\<e - 1, d + 1, b - 1, a - 1, c^3\>,
& \<e - 1, d + 1, b - 1, a + 1, c^3\>,
\\\<e + 1, d - 1, b + 1, a - 1, c^3\>,
& \<e + 1, d - 1, b + 1, a + 1, c^3\>,
\\\<e + 1, d + 1, b + 1, a - 1, c^3\>,
& \<e + 1, d + 1, b + 1, a + 1, c^3\>.
\end{array}
$$
The urge to gather these eight into one piece
(a piece of eight\hspace{-1pt}?\hspace{-1pt}), namely their \mbox{intersection}
$$%
  \<b - e, e^2 - 1, d^2 - 1, a^2 - 1, c^3\>,
$$
is irresistible.  (Who would rather sift through the big list?)  And
it would have become more so had the exponents in the single gathered
component been larger integers, for then the coefficients in the long
list of primary ideals would not even have been rational numbers,
though the intersection would still have been rational.
\end{example}

An arbitrary binomial prime ideal~$\Irp$ in a finitely generated
monoid algebra~$\kk[Q]$ is determined by a monoid prime ideal $P
\subset Q$ and a character $\rho: K \to \kk^*$ defined on a subgroup
of the unit group $G_P \subseteq Q_P$ in the localization of $Q$
along~$P$ (Definition~\ref{d:localize} and
Theorem~\ref{t:binomPrime}).  A binomial ideal $I \subseteq \kk[Q]$
might possess many associated primes sharing the same~$P$ and~$K$,
differing only in the character~$\rho$.  \emph{Mesoprimary ideals}
(Definition~\ref{d:cellular}; see also Propositions~\ref{p:one}
and~\ref{p:filter}) are data structures for keeping track of primary
components for such groups of associated binomial primes.  The term
``group'' here is used in the ordinary nonmathematical sense, but it
is appropriate mathematically: the primary components of a mesoprimary
ideal over an algebraically closed field are indexed by the characters
of a finite abelian group, namely the quotient $\sat(K)/K$ of the
saturation of~$K$ in~$G_P$ (Propositions~\ref{p:assofmeso}
and~\ref{p:mesoprimcomp}).  Gathering primary components into
mesoprimary ideals saves space just as writing the presentation for a
finite abelian group instead of listing every one of its characters
does.

The situation is not typically as simple as in Example~\ref{e:8to1}.
Indeed, upon inspecting a binomial primary decomposition, it can be
difficult to determine which mesoprimary ideals ought to occur, and
which mesoprimary ideal each primary component ought to contribute to.
Furthermore, some primary components of a mesoprimary ideal can be
absent, even if the mesoprimary ideal clearly ought to appear.

\begin{example}\label{e:unital}
If $\mathrm{char}(\kk) \neq 2$, the ideal $I = \<y-x^2y, y^2-xy^2,
y^3\> \subseteq \kk[x,y]$ has primary decomposition $I = \<y\> \cap
\<1+x, y^2\> \cap \<1-x, y^3\>$.  The ideal $I$ is unital, being
generated by differences of monomials, so the component $\<1+x, y^2\>$
feels out of place.  Yet there are no obvious components to gather.
What's missing is a ``phantom'' component $\<1-x, y^2\>$, hidden by
$\<1-x, y^3\>$.  Gathering yields $\<1+x, y^2\> \cap \<1-\nolinebreak
x, y^2\> = \<1-x^2, y^2\>$.  If $\mathrm{char}(\kk) = 2$, then $I =
\<y\> \cap \<1-x^2, y^2-xy^2, y^3\>$ is a primary decomposition
of~$I$.  While this decomposition is forced to be unital, it feels not
fine enough.  Indeed, $1-x^2$ and $1-x$ look like they should
contribute two associated objects, and in all but a single
characteristic they do.  Independent of the characteristic the
mesoprimary decomposition splits the second component: $I = \<y\> \cap
\<1-x^2, y^2\> \cap \<1-x,y^3\>$.
\end{example}

A \emph{mesoprimary decomposition} of a binomial ideal~$I$ is an
expression of~$I$ as an intersection of \emph{mesoprimary components}
(Definition~\ref{d:Wwp}), each of which is a mesoprimary ideal.
Mesoprimary decompositions of binomial ideals always exist
(Definition~\ref{d:mesodecomp'} and Theorem~\ref{t:mesodecomp'}) in a
form that realizes our initial intent (Theorems~\ref{t:primDecomp}
and~\ref{t:galois}).  However, an arbitrary intersection of
mesoprimary ideals is not a mesoprimary decomposition, even if the
intersection is a binomial ideal; exigent additional conditions must
be met regarding the interaction of the combinatorics and the
arithmetic of the mesoprimary components, as compared with that of~$I$
(Remark~\ref{r:stronger}).  In summary, mesoprimary decomposition
gathers primary components so~that:
\begin{enumerate}
\item%
the decomposition into binomial ideals requires no hypotheses on the
ground~field;
\item%
specifying one mesoprimary component takes the place of individually
listing all primary components arising from saturated extensions of a
fixed character;~and
\item%
the combinatorics of the components and their associated prime objects
accurately and faithfully reflects the combinatorics of the decomposed
binomial ideal.
\end{enumerate}

\subsection*{Congruences: binomial combinatorics}
\enlargethispage{1.5ex}

The simple (and not new) idea of binomial combinatorics is that a
binomial ideal $I \subseteq \kk[Q]$ determines an equivalence
relation~$\til$ on~$Q$ that sets $u \sim v$ if $I$ contains a two-term
binomial $\ttt^u - \lambda\ttt^v$ (Definition~\ref{d:I}).  The
quotient $\oQ = Q/\til$ modulo this relation is a monoid.

\begin{example}\label{e:msri}
The following ideals induce the depicted congruences on~$\NN^2$ and
quotient monoids.  The congruence classes are the connected components
of the graphs drawn in the left-hand pictures.  Each element labeled
$0$ is the identity of the quotient monoid.  Each element labeled
$\infty$ in the right-hand picture is \emph{nil}
(Definition~\ref{d:nil} and Remark~\ref{r:zero}) in the quotient
monoid; its congruence class comprises all monomials in the given
binomial ideal.  In items~2 and~4, the groups labeling the rows
indicate how the group in the bottom row acts on the higher rows.  In
all four items, every element outside of the bottom row of the
quotient monoid is \emph{nilpotent}~(Definition~\ref{d:nil}).\vspace{-.1ex}
\begin{enumerate}
\item%
For the ideal $\<y\> \subset \kk[x,y]$, the quotient monoid is $\NN
\cup \infty$:

\begin{tabular}[b]{@{}c@{}}
\\[-3.5ex]
\psfrag{0}{\footnotesize\raisebox{1.5ex}{$0$}}
\psfrag{i}{\footnotesize\raisebox{-.5ex}{\hspace{-.5ex}$\infty$}}
\psfrag{x}{\footnotesize $x$}
\psfrag{y}{\footnotesize $y$}
\begin{tabular}{@{}c@{}}
  \includegraphics[scale=.7]{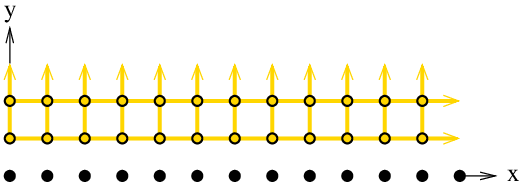}
\end{tabular}
\begin{tabular}{@{}c@{}}
  \\\\\\\quad$\onto$\quad\mbox{}
\end{tabular}
\begin{tabular}{@{}c@{}}
\psfrag{x}{}
\psfrag{y}{}
  \\\\\\[1.06ex]
  \includegraphics[scale=.7]{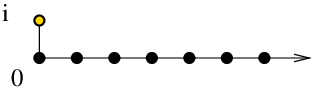}
\end{tabular}
\\\\[-3ex]
\end{tabular}

\item%
For the ideal $\<1-x^2,y^2\> \subset \kk[x,y]$, the quotient monoid is
a copy of the group $\ZZ/2\ZZ$ (the bottom row), a free module over
$\ZZ/2\ZZ$ (the middle row), and a nil:

\begin{tabular}[b]{@{}c@{}}
\\[-6.5ex]
\psfrag{0}{\footnotesize\raisebox{1.5ex}{\hspace{-4.25ex}$\ZZ/2\ZZ$}}
\psfrag{i}{\footnotesize\raisebox{-.5ex}{\hspace{-.5ex}$\infty$}}
\psfrag{x}{\footnotesize $x$}
\psfrag{y}{\footnotesize $y$}
\begin{tabular}{@{}c@{}}
  \includegraphics[scale=.7]{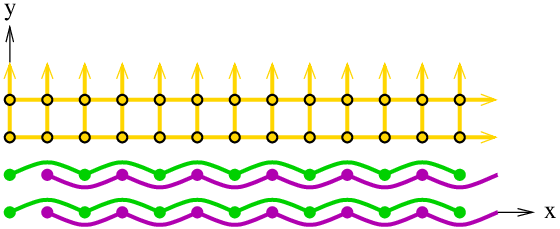}
  \\[-4ex]
\end{tabular}
\begin{tabular}{@{}c@{}}
  \\\\\\\\\quad$\onto$\qquad\mbox{}
\end{tabular}
\begin{tabular}{@{}c@{}}
  \\\\\\\\[1.6ex]\hspace{-1ex}
  \includegraphics[scale=.7]{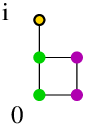}
\end{tabular}
\\\\[-3ex]
\end{tabular}

\item%
For the ideal $\<1-x, y^3\> \subset \kk[x,y]$, the quotient monoid is
the quotient $\NN/(3 + \NN)$ of the natural numbers modulo the
\emph{Rees congruence} of the ideal $3 + \NN$, which makes all
elements of the ideal equivalent and leaves the other elements
of~$\NN$~alone:

\begin{tabular}[b]{@{}c@{}}
\\[-2ex]
\psfrag{0}{\footnotesize\raisebox{1.5ex}{$0$}}
\psfrag{i}{\footnotesize\raisebox{-.5ex}{\hspace{-.5ex}$\infty$}}
\psfrag{x}{\footnotesize $x$}
\psfrag{y}{\footnotesize $y$}
\begin{tabular}{@{}c@{}}
  \includegraphics[scale=.7]{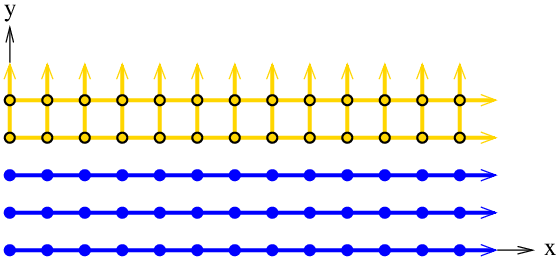}
\end{tabular}
\begin{tabular}{@{}c@{}}
  \\\\\\\quad$\onto$\quad\mbox{}
\end{tabular}
\begin{tabular}{@{}c@{}}
  \\\\\\[1ex]\hspace{-1ex}
  \includegraphics[scale=.7]{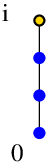}
\end{tabular}
\\\\[-4ex]
\end{tabular}

\item%
For the ideal $\<y-x^2y, y^2-xy^2, y^3\> \subset \kk[x,y]$, the
quotient monoid is a disjoint union of the group~$\ZZ$ and three
$\ZZ$-modules:

\begin{tabular}[b]{@{}c@{}}
\psfrag{0}{}
\psfrag{i}{\footnotesize\hspace{-.5ex}$\infty$}
\psfrag{Z/Z}{\footnotesize\hspace{-2.5ex}$\ZZ/\ZZ$}
\psfrag{Z/2}{\footnotesize\hspace{-3.5ex}$\ZZ/2\ZZ$}
\psfrag{Z}{\footnotesize $\,\ZZ$}
\psfrag{x}{\footnotesize $x$}
\psfrag{y}{\footnotesize $y$}
\begin{tabular}{@{}c@{}}
  \includegraphics[scale=.7]{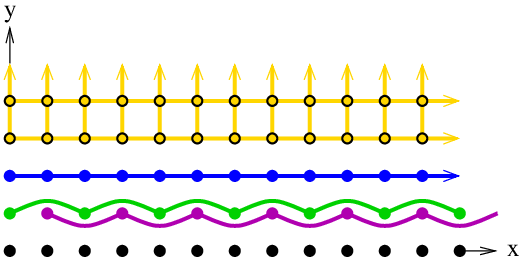}
\end{tabular}
\begin{tabular}{@{}c@{}}
  \\\\\\\quad$\onto$\ \ \mbox{}
\end{tabular}\qquad
\begin{tabular}{@{}c@{}}
  \\\\\\[1.8ex]
  \includegraphics[scale=.7]{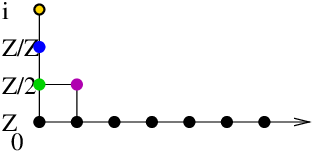}
\end{tabular}
\\\\[-4.5ex]
\end{tabular}
\end{enumerate}
\end{example}

We examined (the literature on) monoid congruences on the premise that
an appropriate decomposition theory for them should lift, either
directly or analogously, to the desired mesoprimary theory for
binomial ideals.  However, although we found rich decomposition
theories for commutative semigroups \cite{grillet}, the expected
analogue of binomial primary decomposition was absent.

The most promising development we encountered along these lines is
Grillet's discovery of conditions guaranteeing that a commutative
semigroup can be realized as a subsemigroup of the multiplicative
semigroup of a primary ring---that is, a ring with just one associated
prime \cite{grilletPrimary}.  That work covers ground
anticipating---in a more general setting---the characterization of
primary binomial ideals over algebraically closed fields of
characteristic zero \cite{primDecomp}.

The closest monoid relative in the literature to primary decomposition
in rings seems to be primary decomposition of congruences
\cite{drbohlav} (see \cite{gilmer} for a treatment in the context of
semigroup rings).  However, one of our motivating discoveries is that
primary decomposition of congruences, being much closer to a shadow of
cellular binomial decomposition (see Theorem~\ref{t:til}), falls
short of serving as a rubric for either primary or mesoprimary
decomposition of binomial ideals.  Indeed, congruences that are
\emph{prime}, meaning that quotients modulo them are cancellative
except perhaps for a nil (Definition~\ref{d:prim*}.4), fail to be
irreducible (Example~\ref{e:irreducible}).  Furthermore, congruences
that are \emph{primary}, meaning that every element in the quotient is
either nilpotent or cancellative (Definition~\ref{d:prim*}.1), admit
further decompositions into pieces that are visibly more
``homogeneous'', in a manner more analogous to primary decomposition
in the presence of embedded primes than to irreducible decomposition
of primary ideals.

\begin{example}\label{e:meso}
All of the congruences depicted in Example~\ref{e:msri} are primary,
but the first three are visibly more homogeneous: in each one, the
non-nil rows all look the same.  In fact, the fourth congruence is the
\emph{common refinement} (Section~\ref{s:primary}) of the first three.
This is equivalent, given that all of the ideals (and their
intersection) are binomial, to saying that the fourth binomial ideal
equals the intersection of the first three, since the ideals in
question are all unital and contain monomials; see
Remark~\ref{r:unital} and Theorem~\ref{t:chain}.  This intersection is
the mesoprimary decomposition from Example~\ref{e:unital}.
\end{example}

Primary binomial ideals in characteristic zero induce \emph{primitive}
congruences (Definition~\ref{d:prim*} and Theorem~\ref{t:til}), but
congruences usually do not admit expressions as intersections (common
refinements) of primitive congruences.  The reason stems from the same
phenomenon that requires one to assume, for binomial primary
decomposition, that the base field is algebraically closed:
decompositions of ideals generated by binomials---even unital
ones---usually require nontrivial roots of unity.  Viewed another way,
the arithmetic part of binomial primary decomposition has a
combinatorial ramification: intersecting multiple primary ideals
inducing the same primitive congruence results in a single mesoprimary
ideal whose associated prime congruence has finite index in the
primitive one (Proposition~\ref{p:mesoprimcomp}).  In essence, primary
congruences on~$Q$ are too coarse to reflect binomial primary
decomposition in~$\kk[Q]$ accurately, and primitive congruences on~$Q$
are too fine, requiring additional arithmetic data from~$\kk$ to
resolve otherwise indistinguishable associated primes in~$\kk[Q]$.

An additional layer of complication arises from the fact that primary
binomial ideals in positive characteristic need not induce nicely
filtered congruences (Example~\ref{e:primary=/=>mesoprimary}).  The
reason for this failure is not under our control: the ideal
$\<(x-1)^p, y(x-1), y^2\>$ happens to be primary, the ideal $\<x^p-1,
y(x-1), y^2\>$ happens to be binomial, and---accidentially, one may
conclude---they coincide in characteristic~$p$.  This highlights that
even the ``binomiality'' of a ring-theoretic construction can depend
on the characteristic, and consequently no study of binomial ideals
can skirt the resulting distinctions.

The true monoid congruence analogue of primary decomposition in
rings is a suitable compromise, developed (in
Sections~\ref{s:taxonomy}--\ref{s:cong}) as \emph{mesoprimary
decomposition for congruences} (Definition~\ref{d:mesodecomp} and
Theorem~\ref{t:mesodecomp}).  The type of homogeneity mentioned before
Example~\ref{e:meso}, discovered by Grillet \cite{grilletPrimary}
(Remark~\ref{r:primary}.4), characterizes mesoprimary congruences
(Corollary~\ref{c:one} and Remark~\ref{r:semifree}).  These are also
distinguished (Theorem~\ref{t:oneassprim}) as those with just one
\emph{associated prime congruence} (Definitions~\ref{d:prim*}.4
and~\ref{d:mesoass}), a notion new to monoid theory.  For comparison,
a congruence is primary precisely when it has just one
\emph{associated prime ideal} (Definition~\ref{d:witness} and
Corollary~\ref{c:primary}).

The development of binomial mesoprimary decomposition in the latter
half of the paper (Sections~\ref{s:aug}--\ref{s:false}) mirrors the
first half directly.  Arithmetic existence statements build on
combinatorial ones by exhibiting lifts of statements or requirements
concerning elements equivalent under congruences to statements or
requirements concerning binomials with nonzero coefficients.

It is worth warning the reader at this juncture of the inevitable
clash of terminology in translating between combinatorics and
arithmetic; see the table in Section~\ref{s:binomial}, which in
particular explains the source of our term \emph{mesoprimary} to mean
``between the two occurrences of `primary'\,".
To aid readers coming from commutative ring theory, the basic notions
from semigroup theory are reviewed from scratch
(Sections~\ref{s:taxonomy} and~\ref{s:primary}).  For readers
interested primarily in monoids, we complete the entire combinatorial
theory in Section~\ref{s:cong}, before starting the arithmetic theory
in Section~\ref{s:aug}.

\subsection*{Witnessed associated objects}

In ordinary primary decomposition, a witness is an element whose
annihilator is (an associated) prime.  Our \emph{witnesses} also have
\emph{associated} prime objects (Definitions~\ref{d:witness},
\ref{d:mesoass}, and~\ref{d:witness'}).  Continuing the parallel, our
notions of associatedness are defined by local combinatorial or
algebraic conditions but equivalently characterized by the consistent
appearance of prime objects in every primary decomposition
(Theorems~\ref{t:primary} and~\ref{t:assPrim}).  The local conditions
defining witnesses incorporate the combinatorial quiddity of having
prime \mbox{annihilator in ordinary ring theory}.

The proof of concept for mesoprimary decomposition as a mode to
connect the combinatorial and arithmetic settings lies in a
fundamental discovery: there is a combinatorially defined set of
witnesses that captures decompositions of both a binomial ideal and
its induced congruence.  To yield finite decompositions, however, not
all witnesses are to be believed.  The \emph{key witnesses} for
congruences (Definition~\ref{d:witness}) and \emph{essential
witnesses} for binomial ideals (Definition~\ref{d:witness'}) yield
finitely many components whose intersections suffice.  These key and
essential decompositions can generally fail to be minimal in ways that
even retain symmetry.  In the cellular binomial ideal case, we
demonstrate a systematic reduction to \emph{character witnesses}
(Defintion~\ref{d:false'}) that should have an extension to general
binomial ideals.  The dichotomy between key and essential witnesses
demands care, as do other subtle distinctions between the
combinatorial and arithmetic aspects of the theory, since they
necessitate occasional slight weakenings, or failures of the
combinatorics to lift; see Remarks~\ref{r:mesorefine}
and~\ref{r:construct}, for~instance.

\subsection*{Acknowledgements}

The authors are very grateful to Chris O'Neill and Howard M Thompson
for their detailed readings of previous drafts; their comments led to
substantial mathematical corrections and expositional improvements.
In particular, O'Neill detected an oversight in the definition of
coprincipal congruence that led to the excision of claims about
binomial irreducible decomposition; see \cite{soccular} for amended
statements and corrected proofs.  Zekiye \c Sahin and Laura Matusevich
provided crucial mathematical corrections as well.  TK was supported
by an EPDI fellowship and gratefully acknowledges the hospitality of
Institut Mittag-Leffler, where substantial parts of the research for
this paper were carried out.  EM had support from NSF grants
DMS-0449102 = DMS-1014112 and DMS-1001437.

\section{Taxonomy of congruences on monoids}\label{s:taxonomy}

Fix a \emph{commutative semigroup}~$Q$: a set with an associative,
commutative binary operation (usually denoted by~$+$ here).  Assume
that $Q$ has an identity, usually denoted by~$0$ here, so $Q$ is a
\emph{monoid}.  An \emph{ideal} $T \subseteq Q$ is a subset such that
$T+Q \subseteq T$, and $T$ is \emph{prime} if $t+s \in T$ implies
$t\in T$ or $s\in T$.  The ideal generated by elements $q_1,\dots q_s$
is written $\<q_1,\dots,q_s\>$.  A~\emph{congruence} $\sim$ on~$Q$ is
an equivalence relation that is additively closed: $a \sim b \implies
a + c \sim b + c$ for all $a,b,c \in Q$.  The quotient $Q/\til$ by any
congruence is a monoid.  The minimal relation satisfying this
definition is equality itself, called the \emph{identity congruence}.
The congruence that equates all pairs of elements in~$Q$, and has
trivial quotient, is the \emph{universal congruence}.  For any ideal
$T\subseteq Q$, under the \emph{Rees congruence} $\til_T$ all elements
of $T$ form one class, while all elements outside of $T$ are
singletons.

\begin{defn}\label{d:module}
A \emph{module} over a commutative monoid $Q$ is a nonempty set $T$
with an \emph{action} of~$Q$, which means a map $Q \times T \to T$,
written $(q,t) \mapsto q + t$, that satisfies
\begin{itemize}
\item%
$0 + t = t$ for all $t \in T$, and
\item%
$(q + q') + t = q + (q' + t)$,
\end{itemize}
the latter meaning that the action respects addition.  A
\emph{congruence} on a module is an equivalence relation that is
preserved by the action.  A \emph{module homomorphism} over a given
monoid is a set map that respects the actions.  For any element $q\in
Q$, the \emph{addition morphism} $\phi_q : Q \to \<q\>$ is the module
morphism defined by $p\mapsto p+q$.  The \emph{kernel} $\ker (\phi)$
of a module homomorphism $\phi : T_1 \to T_2$ is the congruence on
$T_1$ under which $t \sim s \iff \phi(t) = \phi (s)$.
\end{defn}

\begin{remark}\label{r:grilletActs}
For general semigroups Grillet defines an \emph{act} as a set with an
action of a semigroup that satisfies only the second bullet in
Definition~\ref{d:module}, even if the semigroup was a monoid to start
with~\cite{grilletActions}.  To every semigroup $S$ a formal identity
element~$e$ can be adjoined (even if $S$ is already a monoid) to form
the monoid $S \cup \{e\}$.  Upon this operation an $S$-act turns into
an $(S \cup \{e\})$-module as it automatically satisfies the first
item in Definition~\ref{d:module}.
\end{remark}

\begin{remark}
A subsemigroup of a monoid may have an identity and in that case it
may or may not be the identity of the monoid.  To the contrary, a
submonoid is required to have the same identity as its ambient monoid.
In this sense a subsemigroup of a monoid can be a monoid without being
a submonoid.
\end{remark}

\begin{defn}\label{d:subgroup}
A \emph{subgroup of a monoid} is a subsemigroup that is a group.
\end{defn}

\begin{defn}\label{d:Green}
\emph{Green's preorder} on a monoid is the divisibility preorder $p
\preceq q \iff \<p\> \supseteq \<q\>$.  \emph{Green's relation} on a
monoid is $p \sim q \iff \<p\> = \<q\>$.
\end{defn}

\begin{lemma}\label{l:modGreen}
The quotient of a commutative monoid modulo Green's relation is
partially ordered by divisibility.
\end{lemma}
\begin{proof}
\cite[Proposition~I.4.1]{grillet}.
\end{proof}

\begin{remark}
Green's relation measures the extent to which group-like behavior
occurs in a monoid.  Idempotents and non-trivial units are
obstructions to partially ordering a monoid by divisibility.  In
particular, a monoid with trivial unit group is partially ordered if
Green's relation is trivial.  Note that our divisibility preorder is
the opposite direction compared to Grillet's, to be compatible with
divisibility of monomials.
\end{remark}

The following observation, which relies crucially on the noetherian
hypothesis, is applied in the proof of
Proposition~\ref{p:coprincipal}.

\begin{lemma}\label{l:bijective}
Fix a noetherian commutative monoid~$Q$.  If $p \in Q$ and the Green's
class of~$w$ satisfies $[w] = [p + w]$, then the map $[w] \to [p + w]$
of Green's classes induced by adding~$p$ is bijective.
\end{lemma}
\begin{proof}
Suppose that $v \in [w] = [p + w]$.  For surjectivity, first note that
$v \in p + \<w\>$, because $v \in \<v\> = \<p + w\> = p + \<w\>$.
Consequently $v \in p + [w]$ because $[v] = [w]$ is the (unique)
minimal element in the poset of Green's classes with representatives
in~$\<w\>$ (that is, $[v]$ can't lie in $p + [u]$ if $[u] \succ [w]$).

Since the sets in question can be infinite, injectivity requires
additional reasoning.  Suppose that $v \in [w]$ satisfies $p + w = p +
v$.  By surjectivity, for $k \in \NN$ choose $w_k, w_k' \in [w]$ so
that $k \cdot p + w_k = w$ and $k \cdot p + w_k' = v$.  If $\til_k$ is
the kernel congruence of addition by $k \cdot p$, then $\til_k$
refines $\til_\ell$ whenever $k \leq \ell$.  The noetherian property
implies that the chain of kernel congruences stabilizes: $\til_k =
\til_{k+1}$ for $k \gg 0$.  But $w_k \sim_{k+1} w_k'$ for all~$k$
because $p + w = p + v$, whence $w_k \sim_k w_k'$ for $k \gg 0$ by
stability.  For $k \gg 0$, then, $w = k \cdot p + w_k \sim k \cdot p +
w_k' = v$.
\end{proof}

\begin{defn}\label{d:nil}
A non-identity element $\nil$ in a monoid~$Q$ is \emph{nil} if $q +
\nil = \nil$ for all $q \in Q$.  An element $q \in Q$ is
\begin{itemize}
\item%
\emph{nilpotent} if one of its multiples $nq$ is nil for some
nonnegative integer $n \in \NN$.
\item%
\emph{cancellative} if addition by it is injective:
$q + a = q + b \implies a = b$ in~$Q$.
\item%
\emph{partly cancellative} if $q + a = q + b \neq \nil \implies a = b$
for all cancellative $a,b \in Q$.
\end{itemize}
A set $S$ of elements in a monoid is \emph{torsion-free} if $na = nb
\implies a = b$ for all $n \in \NN$, whenever $a,b\in S$.  An
\emph{affine semigroup} is a monoid isomorphic to a finitely generated
submonoid of a free abelian group.  A \emph{nilmonoid} is a monoid
whose nonidentity elements are all nilpotent.
\end{defn}

\begin{remark}\label{r:zero}
In the literature a nil is often called a zero instead; but when
we work with monoid algebras, we need to distinguish the nil monomial
$\ttt^\nil$ from the zero element~$0$ of the algebra (see
Section~\ref{s:aug} for ramifications of this distinction), and we
need to identify the identity monomial~$\ttt^0$ with the unit
element~$1$ of the algebra.
\end{remark}

\begin{remark}\label{r:differ}
The condition $a + c = b + c'$ for cancellative $c,c'$ means that $a$
and~$b$ are off by a unit in the localization~$Q'$ of~$Q$ obtained by
inverting all of its cancellative elements.  Note that the natural map
$Q \to Q'$ is~injective.
\end{remark}

\begin{defn}\label{d:prim*}
Fix a commutative monoid~$Q$, a congruence~$\til$, and use a bar to
denote passage to the quotient $\oQ = Q/\til$.  The congruence $\til$
is
\begin{enumerate}
\item%
\emph{primary} if
every element of~$\oQ$ is either nilpotent or cancellative.
\item%
\emph{mesoprimary} if it is primary and every element of~$\oQ$ is
partly cancellative.
\item%
\emph{primitive} if it is mesoprimary and the cancellative subset
of~$\oQ$ is torsion-free.
\item%
\emph{prime} if every element of~$\oQ$ is either nil or cancellative.
\item%
\emph{toric} if the non-nil elements of~$\oQ$ form an affine
semigroup.
\end{enumerate}
\end{defn}

\begin{remark}\label{r:primary}
The notions just defined are nearly or exactly the same as concepts
that have appeared in the literature on monoids.
\begin{enumerate}
\item%
Our definition of prime and primary congruences agrees with those in
the literature \cite[\S5]{gilmer}.  In the case of prime congruences,
where the non-nil elements of $\oQ$ form a cancellative monoid, this
is easy.  In the case of primary congruences, for $q \in Q$ the
condition Gilmer expresses as $q + a \sim q + b$ for all $a,b \in Q$
is equivalent to the class $\ol q$ being a nil in~$\oQ = Q/\til$, so
$q$ lies in the nil class; and the condition that Gilmer expresses by
saying that $q$ lies in the radical of the nil class is equivalent to
$\ol q$ being nilpotent in~$\oQ$.  

\item%
Our definition of affine semigroup differs slightly from \cite[\S
II.7]{grillet}: Grillet requires the unit group to be trivial, whereas
we do not.  Equivalently, our affine semigroups are the finitely
generated, cancellative, torsion-free commutative monoids, while
Grillet additionally requires affine semigroups to be reduced (that
is, to have trivial unit group).

\item%
A congruence on~$Q$ is primary if and only if $\oQ$ is a
\emph{subelementary} monoid, by definition \cite[\S VI.2.2]{grillet}.

\item%
A congruence on~$Q$ is mesoprimary if and only if the subelementary
monoid~$\oQ'$, obtained from the monoid $\oQ$ in the previous item by
inverting its cancellative elements, is \emph{homogeneous}
\cite[\S VI.5.3]{grillet}; this is Corollary~\ref{c:one}, below.
\end{enumerate}
\end{remark}

\begin{lemma}\label{l:implies}
For monoid congruences,
\begin{itemize}
\item%
toric $\implies$ prime $\implies$ mesoprimary $\implies$ primary; and
\item%
toric $\implies$ primitive $\implies$ mesoprimary $\implies$ primary.
\end{itemize}
\end{lemma}
\begin{proof}
The only implication that is not immediate from the definitions is
that prime implies mesoprimary.  For this assume $\til$ is a prime
congruence and that $\ol q + \ol a = \ol q + \ol b$ in~$\oQ$ with
neither side being nil.  Then $\ol q$ is not nil, whence $\ol a = \ol
b$ by cancellativity.
\end{proof}

\begin{defn}\label{d:I}
The \emph{semigroup algebra} $\kk[Q] = \bigoplus_{q \in Q} \kk \cdot
\ttt^q$ is the direct sum with multiplication $\ttt^p \ttt^q =
\ttt^{p+q}$.  Any congruence $\til$ on~$Q$ induces a grading of
$\kk[Q]$ by $\oQ = Q/\til$ in which the \emph{monomial}\/ $\ttt^q$ has
degree~$\ol q \in \oQ$ whenever $q \mapsto \ol q$ under the quotient
map $Q \to \oQ$.  A \emph{binomial ideal}\/ $I \subseteq \kk[Q]$ is an
ideal generated by \emph{binomials} $\ttt^p - \lambda\ttt^q$, where
$\lambda \in \kk$ is a scalar, possibly equal to $0 \in \kk$.  A
binomial ideal is \emph{unital} if all coefficients $\lambda$ are
equal to either~$0$ or~$1$.  The ideal $I$ \emph{induces} the
congruence $\til_I$ in which $p \sim_I q$ whenever $\ttt^p -
\lambda\ttt^q \in I$ for some unit~$\lambda \in \kk^*$.
\end{defn}

\begin{remark}\label{r:unital}
Giving a congruence on~$Q$ is the same as giving a unital ideal
in~$\kk[Q]$ that is generated by unital binomials $\ttt^p - \ttt^q$
(and no monomials).  In particular, every congruence is induced by
some binomial ideal.  That said, other binomial ideals can induce the
same congruence as the canonical unital ideal, by rescaling the
variables or via Theorem~\ref{t:chain}, for instance.
\end{remark}

\begin{example}[Some congruences from unital ideals]\label{e:pure-diff}
\mbox{}
\begin{enumerate}
\item%
The prime ideal $\<x-y\> \subset \kk[x,y]$ induces a toric congruence
such that $\ol{\NN^2} \cong \NN$.

\item%
The ideal $\<x^2-y^2\> \subset \kk[x,y]$ induces a prime congruence
with $\ol{\NN^2}$ isomorphic to the submonoid $Q \subseteq G = \ZZ
\oplus \ZZ/2\ZZ$ generated by $(1,0)$ and~$(1,1)$.  The monoid is not
torsion-free since $x^2 = y^2$ but $x \neq y$ in~$\kk[Q]$.  Therefore
the congruence on~$\NN^2$ is not toric, since $Q$ generates~$G$ as a
group.

\item%
The ideal $\<x^2-x\> \subset \kk[x]$ induces the same toric congruence
on~$\NN$ as the prime ideal~$\<x\>$ does, but $\<x^2-x\>$ is not
primary (in fact, not even cellular; see Definition~\ref{d:cellular}).
Nevertheless $\til_{\<x^2-1\>} = \til_{\<x\>}$ is irreducible
according to Definition~\ref{d:irreducible}.

\item%
The $\<x,y\>$-primary ideal $\<x^2,x-y\>$ induces the primitive
congruence on~$\NN^2$ with $\ol{\NN^2} \cong \{0,x,\infty\} =: Q$.
The monoid algebra $\kk[Q]$ has a presentation $\kk[x,y]/J$ where $J =
\<x-y, x-x^{2}\> = \<x-1,y-1\> \cap \<x,y\>$ induces the same
congruence.

\item%
The binomial ideal $\<y-x^2y,y^2-xy^2,y^3\>$ induces a primary
congruence whose classes are depicted as connected components of the
graph in the following~figure.
\begin{center}
\psfrag{x}{\footnotesize $\ x$}
\psfrag{y}{\footnotesize $\ y$}
\includegraphics{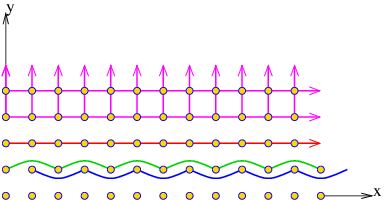}
\end{center}
This congruence exhibits the distinction between primary and
mesoprimary congruences: for a primary congruence, no injectivity is
required of addition by a nilpotent element.  In the picture, this
means that translating two dots in different classes upward by one
unit can force them into the same non-nil class.  To make the
congruence mesoprimary, homogenize the bottom three rows by replacing
any two of them with the third; after that, upward translation on two
dots keeps them in separate classes unless both land in the nil class.
This replacement procedure also exhibits the distinction between
mesoprimary and primitive congruences: it results in a primitive
congruence only if the bottom row or the third row is preserved;
preserving the second row yields torsion in the cancellative part
of~$\oQ$.
\end{enumerate}
\end{example}

The following example demonstrates the partly cancellative property.

\begin{example}\label{e:mesoprimary}
Partly cancellative elements can still merge congruence classes.  For
instance, consider the congruence on~$\NN^2$ induced by $I = \<x^2-xy,
xy-y^2, x^3, y^3\> \subseteq \kk[x,y]$.  In the following figure
\begin{center}
\psfrag{x}{\footnotesize $\ x$}
\psfrag{y}{\footnotesize $\ y$}
\includegraphics[scale=.7]{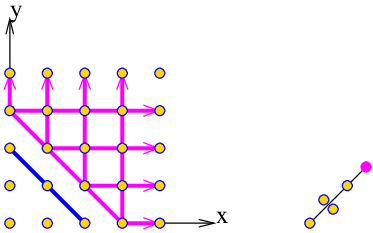}
\end{center}
the congruence on~$\NN^2$ appears at left, and the
quotient~$\ol\NN{}^2$ appears at right.  The~quotient is the monoid
$\NN$ with two copies of~$1$ modulo the Rees congruence of $\<3\>$
(declare all elements in $\<3\>$ congruent).  The two copies of~$1$
become identified upon addition by either: $1 + 1 = 1 + 1' = 1' + 1' =
2$.  Nonetheless, both $1$ and~$1'$ are partly~cancellative and the
congruence is mesoprimary.
\end{example}

The next result will be applied in the proofs of
Theorems~\ref{t:coprincipal} and~\ref{t:til}.  The conclusion says
that $Q/F$ is a nilmonoid whose Green's preorder is an order (i.e.~is
antisymmetric).  Equivalently, it says that $Q/F$ is \emph{naturally
partially ordered}, or a \emph{holoid} \cite[\S V.2.2]{grillet}.

\begin{lemma}\label{l:poset}
Fix a monoid $Q$ whose identity congruence is primary, so the
non-nilpotent elements of $Q$ constitute a cancellative submonoid $F
\subseteq Q$.  The quotient monoid~$Q/F$ defined by the congruence
$$%
  p \sim q\ \iff\ p + f = q + g \text{ for some } f,g \in F
$$
is a nilmonoid partially ordered by divisibility.  If $Q$ is finitely
generated, $Q/F$ is~finite.
\end{lemma}
\begin{proof}
This is more or less \cite[Proposition~VI.3.3]{grillet}, but the proof
is simple.  Every nonidentity element of $Q/F$ is nilpotent by
definition, so when $Q$ is finitely generated, $Q/F$ is finite.  The
rest follows because every nilmonoid is partially ordered by
divisibility; this is easy, and can be found in
\cite[Proposition~IV.3.1]{grillet}.
\end{proof}

\begin{remark}
It is a crucial assumption for Lemma~\ref{l:poset} that every element
is nilpotent or cancellative, excluding idempotents.  If every
cancellative element is a unit, e.g.~after localizing at the nilpotent
ideal (see Section~\ref{s:witness}), then $Q/F$ equals $Q$ modulo
Green's relation.
\end{remark}

Concluding this section we comment on the notion of irreducibility for
congruences which is, despite the close connection between binomial
ideals and their congruences, quite different from irreducibility for
ideals.

\begin{defn}\label{d:irreducible}
A congruence is \emph{irreducible} if it cannot be expressed as the
common refinement of two congruences neither of which equals the given
one.
\end{defn}

The theories of irreducible and primary decomposition for congruences
in commutative monoids are not as nice as for (binomial) ideals in
rings.  The following example might come as a nasty surprise (it did
to us).  Quotients by irreducible congruences are characterized in
\cite[Theorem~VI.5.3]{grillet}.

\begin{example}\label{e:irreducible}
The identity congruence on $\NN^2$ is reducible: it is the common
refinement of the congruences induced by $\<x-1\>$ and $\<y-1\>$.
Ring-theoretically, this is due to the fact that $\<x-1\> \cap
\<y-1\>$ does not contain binomials.
\end{example}

Example~\ref{e:irreducible} demonstrates the sad reality that prime
congruences need not be irreducible.  In a wider sense, unrestricted
primary or irreducible decomposition of congruences decomposes them
into components that are too fine to provide nuanced information about
their combinatorics.  The theory of mesoprimary decomposition, with
its well founded notions of associatedness for prime ideals and prime
congruences, is our remedy.

\section{Primary decomposition and localization in monoids}\label{s:primary}

We review the notion of primary decomposition for congruences on
finitely generated commutative monoids, which traces back to
Drbohlav~\cite{drbohlav}.  This decomposition is only a coarse
approximation of mesoprimary decomposition, a central goal of this
paper.  In general, a \emph{decomposition} of a congruence is an
expression of it as a common refinement of congruences.  The notion of
refinement here is standard: formally, an equivalence relation on $Q$
is a reflexive, symmetric, transitive subset of $Q \times Q$; one
relation~$\til$ \emph{refines} another relation~$\app$ if $\app$
contains~$\til$ (we also say $\app$ \emph{coarsens}~$\til$); and the
\emph{common refinement} of a family of equivalence relations is their
intersection~in~$Q \times Q$.

\begin{remark}
Every congruence in this setting admits a \emph{primary
decomposition}: an expression as the common refinement of finitely
many primary congruences \cite[Theorem~5.7]{gilmer}.  Similarly to the
case of rings, this follows from the existence of irreducible
decomposition using a noetherian induction argument.  Any
decomposition theory that is finer than primary decomposition---that
is, any theory that further decomposes each primary component---yields
a greater number of congruences each of which is coarser than some
primary component.
\end{remark}

\begin{remark}
The preimage under any monoid homomorphism of a prime ideal is prime.
Since $\NN^n$ has only finitely many prime ideals and a finitely
generated commutative monoid $Q$ has a presentation $\NN^n \onto Q$,
it follows that $Q$ has only finitely many prime ideals.  Precisely
one of these is the \emph{maximal ideal} of~$Q$.
\end{remark}

\begin{conv}\label{conv}
To avoid tedious case distinctions in the following, we consider the
empty set as an ideal of any monoid, and in fact we declare it to be a
prime ideal (its complement is, after all, a submonoid).  The empty
set considered as an ideal will be denoted by $\nothing \subset Q$;
this symbol is never used for any other purpose in this~paper.
\end{conv}

\begin{defn}\label{d:nilpideal}
The \emph{nilpotent ideal} of a congruence $\til$ on~$Q$ is the ideal
of~$Q$ consisting of all elements with nilpotent image in $Q/\til$.
If $P$ is the nilpotent ideal of a primary congruence $\til$,
then~$\til$ is \emph{$P$-primary}.
\end{defn}

\begin{lemma}\label{l:nothing}
If $\til$ is a primary congruence, then the nilpotent ideal is prime.
If $Q/\til$ is cancellative, then $\til$ is $\nothing$-primary.\qed
\end{lemma}

\begin{remark}\label{r:generators}
If $q_1,\dots, q_n$ generate $Q$, then a primary congruence defines a
partition of~$\{q_1,\ldots,q_n\}$ into generators with cancellative
and nilpotent images, respectively.  In this case the nilpotent ideal
is generated by the generators~$q_i$ with nilpotent images.
\end{remark}

\begin{prop}\label{p:pprefine}
The common refinement of finitely many $P$-primary congruences is
$P$-primary.
\end{prop}
\begin{proof}
It suffices by induction to show this for two $P$-primary
congruences $\til_1$ and $\til_2$.  Reducing modulo their
intersection, we can assume that the intersection is the identity
congruence on~$Q$.  Denote by $Q_1$ and~$Q_2$ the quotients modulo
$\til_1$ and~$\til_2$, respectively.  By assumption $P \subset Q$ is
the nilpotent ideal of both $\til_1$ and $\til_2$.  We claim that if
$p\in P$ then $p$ is nilpotent already in~$Q$.  Indeed, a
sufficiently high multiple of $p$ is congruent to nil under
both~$\til_1$ and~$\til_2$, and since their intersection is trivial
this can only happen if that multiple is nil.  On the other hand, if
$p\notin P$, then it must be cancellative: if there exist $a,b \in
Q$ with $a+p = b+p$, then $a \sim_1 b$ and $a \sim_2 b$ both
hold---whence $a = b$, in fact---since $p$ is cancellative modulo
$\til_1$ and~$\til_2$.
\end{proof}

\begin{remark}\label{r:pprefine}
Albeit in different language, \cite[Theorem~5.6.2]{gilmer} contains a
variant of the statement of Proposition~\ref{p:pprefine}.
\end{remark}

Passing from the theory surrounding $P$-primary congruences to that
for general congruences is best accomplished by localizing.

\begin{defn}\label{d:localize}
The \emph{localization} $T_P$ of a $Q$-module $T$ at a prime ideal $P
\subset Q$ is the set of formal differences $t - q$ for $t \in T$ and
$q \notin P$, with $t - q$ and $t' - q'$ identified when $w + q' + t =
w + q + t'$ for some $w \in Q \minus P$.  Conventions for this are as
follows.
\begin{itemize}
\item%
The localization $Q_P$ of~$Q$ itself is naturally a monoid, and $T_P$
is a $Q_P$-module.
\item%
The image of $P$ in~$Q_P$ is the maximal ideal~$P_P$ of~$Q_P$.
\item%
Any given congruence~$\til$ on~$Q$ induces a congruence on~$Q_P$, also
denoted~$\til$.
\item%
If $\oQ = Q/\til$ then we write $\oQ_P = Q_P/\til$.
\item%
The \emph{unit group at $P$} is the subgroup $G_P = Q_P \minus
P_P$.
\end{itemize}
\end{defn}

\begin{example}\label{e:noNil}
Localizing $Q$ at the empty prime ideal yields the universal
group~$Q_\nothing$.  When $Q$ has a nil, $Q_\nothing$ is trivial.  In
fact, the universal group $Q_\nothing$ is trivial precisely when $Q$
has a nil.  (Proof: If $Q_\nothing$ is trivial, then $q$ becomes equal
to~$0$ after inverting every element of~$Q$.  Thus there is an element
$x_q \in Q$ such that $x_q + q = x_q$.  As $Q$ is generated by a
finite set $S \subseteq Q$, the sum of the elements $x_s$ for $s \in
S$ exists, and it is nil in~$Q$.)
\end{example}

By definition, the group of units of~$Q_P$ acts on itself and also on
the set $\oQ_P$ of equivalence classes modulo any congruence on~$Q_P$.
Here and in what follows, we often think of the quotient~$\oQ$
explicitly as a set of congruence classes in~$Q$.  Thus $\oQ_P$ is a
set of congruence classes in $Q_P$.  We record this fact for future
reference.

\begin{lemma}\label{l:action}
Let $P \subset Q$ be a prime ideal.  Given any congruence on~$Q$, the
unit group of\/~$Q_P$ acts on the quotient~$\oQ_P$ modulo the induced
congruence on~$Q_P$.\qed
\end{lemma}

In analogy with what happens over rings, primary decomposition of
congruences behaves well under localization.

\begin{thm}\label{t:commutes}
Primary decomposition of congruences commutes with localization: if
$\til = \til_1 \cap \cdots \cap \til_r$ is a primary decomposition of
the congruence $\til$ on~$Q$, and $P \subset Q$ is a prime ideal, then
each of the congruences induced by~$\til_1,\ldots,\til_r$ on~$Q_P$ is
primary or universal, and their common refinement is the congruence
induced by~$\til$ on~$Q_P$.
\end{thm}
\begin{proof}
If some element of~$Q$ lies outside of~$P$ but becomes nilpotent
in~$Q/\til_j$, then $\til_j$ induces the universal congruence
on~$Q_P$, so assume no such element exists.  Suppose that $q - u \in
Q_P$.  Our assumption means that $u$ has cancellative image
in~$Q/\til_j$.  It follows that $q - u \in Q_P$ becomes cancellative
in~$Q_P/\til_j$ as long as $q - u$ does not become nilpotent
in~$Q_P/\til_j$.  Therefore $\til_j$ induces a primary congruence
on~$Q_P$.  The rest of the proof is covered by the following lemma.
\end{proof}

\begin{lemma}\label{l:commutes}
Localization commutes with finite common refinement of congruences: if
$\til = \til_1 \cap \cdots \cap \til_r$ as congruences on~$Q$, and $P
\subset Q$ is a prime ideal, then the induced congruences on the
localization~$Q_P$ still satisfy $\til = \til_1 \cap \cdots \cap
\til_r$.
\end{lemma}
\begin{proof}
For the duration of this proof, a dot denotes passage to~$Q_P$, so
$\dot\til$ is the congruence on~$Q_P$ induced by~$\til$ on~$Q$.  If $v
\mathbin{\dot\sim_j} w$ in~$Q_P$ for all~$j$, then for each~$j$ there
is an element $u_j \in Q \minus P$ with $u_j + v \sim_j u_j + w$.
Summing these elements $u_j$ yields an element $u = u_1 + \cdots +
u_r$ such that $u + v \sim_j u + w$ for all~$j$, whence $u + v \sim u
+ w$ by definition of~$\til$ as the common refinement.  Therefore $v
\mathbin{\dot\sim} w$.  This logic easily reverses to show that $v
\mathbin{\dot\sim} w \implies v \mathbin{\dot\sim_j} w$ for all~$j$.
We conclude that $\dot\til = \dot\til_1 \cap \cdots \cap \dot\til_r$,
as~desired.
\end{proof}

\section{Witnesses and associated prime ideals of congruences}\label{s:witness}

Our aim in this section is to show that primary decompositions of
congruences in finitely generated commutative monoids have
well-defined associated prime ideals.  These, and their witnesses,
reflect the combinatorial features of a given congruence more
accurately than does primary decomposition alone.

\begin{defn}\label{d:annihilator}
For any ideal $T \subseteq Q$, the \emph{annihilator modulo~$T$} is
the common refinement $\ann(T) = \bigcap_{t \in T} \ker(\phi_t)$ of
the kernels of the addition morphisms $\phi_t$ for~$t \in T$.
\end{defn}

\begin{remark}
If $q_1 + v = q_2$ then $\ker(\phi_{q_1})$ refines $\ker(\phi_{q_2})$.
Therefore, in the definition of~$\ann(T)$, it suffices to intersect
only over generators of $T$.  Equivalently, if $T$ is generated by
$t_1,\ldots,t_r$, then $\ann(T) = \ker(\phi_{t_1}\! \oplus \cdots
\oplus \phi_{t_r}: Q \to T^{\oplus r})$.  If $T = \nothing$ is the
empty ideal, then $\ann(T)$ is the universal congruence (that has just
one class).
\end{remark}

\begin{example}
To explain the ``annihilator'' terminology, let $Q$ be a monoid with
nil~$\infty$ and write $\kk[Q]^- := \kk[Q]/\<\ttt^\infty\>$.  If $T
\subseteq Q$ is a monoid ideal, then $\ann(T)$ is the congruence
induced by the binomials (and the monomials) in the ideal
$\big(0:\kk\{T\}\big) = \{f \in \kk[Q] \mid f\,\kk\{T\} = 0 \text{ in
} \kk[Q]^-\}$.
\end{example}

\begin{defn}\label{d:cover}
Fix a prime ideal $P \subset Q$ with $P_P \subset Q_P$ minimally
generated by $p_1, \ldots, p_r$.  The \emph{$P$-covers} of $q \in Q$
are the elements $q+p_i \in Q_P$ for $i=1, \ldots, r$.  The
\emph{cover morphisms at~$P$} are the morphisms $\phi_i : Q_P \to
\<p_i\>_P$ defined via $q \mapsto q + p_i$; if $P$ is the maximal
ideal, then the~$\phi_i$ are called simply \emph{the cover morphisms}
of~$Q$.
\end{defn}

\begin{remark}\label{r:cover}
The set of $P$-cover morphisms depends on the choice of generators
$p_1, \ldots, p_r$ and may be infinite if, for example, $Q_P$ has a
lot of units.  However, modulo Green's relation on~$Q_P$ there is a
unique finite minimal generating set of any ideal, and every minimal
generating set for~$P_P$ maps bijectively to it.
\end{remark}

\begin{lemma}\label{l:cover}
For a fixed prime~$P$, the set of kernels of $P$-cover morphisms is
finite.
\end{lemma}
\begin{proof}
Two cover morphisms $\phi_p$ and $\phi_{p'}$ for elements $p,p'$ that
are Green's equivalent in~$Q_P$ have the same kernel, because if $p
\in \<p'\>$ then there exists an element $u$ such that $p = p'+u$, and
thus the kernel of $\phi_{p'}$ refines the kernel of $\phi_p$ and vice
versa.
\end{proof}

Next comes the first main new definition of the paper (note that the
concept of mesoprimary congruence in Definition~\ref{d:prim*} is
equivalent to a notion already available in the literature; cf.\
Remark~\ref{r:primary}.4), whose details can best be seen in action in
the proofs of Proposition~\ref{p:coprincipal} and
Theorem~\ref{t:coprincipal}.

\begin{defn}\label{d:witness}
Let $\til$ be a congruence on $Q$ and $P \subset Q$ a prime ideal.
Consider the localized quotient $\oQ_P$.  For each $q \in Q$ let $\ol
q$ be its image in $\oQ_P$.  An element $\ol q$ is \emph{exclusively
maximal} in a subset $S \subseteq \oQ_P$ if $\ol q$ is the unique
maximal element of~$S$ under Green's preorder.  An element $w \in Q$
is~a
\begin{enumerate}
\item%
\emph{witness for $P$} if the class of $\ol w$ is non-singleton under
the kernel of each cover morphism (i.e.~the class $\ol p + \ol w$ is
non-singleton for all $p \in P$) and in each of its non-singleton
kernel classes, $\ol w$ is not exclusively maximal;
\item%
\emph{key witness for $P$} if the class of $\ol w$ is non-singleton
under the intersection of the kernels of all cover morphisms (i.e.~if
the class of $\ol w$ is non-singleton under~$\ann(\ol P_P)$) and $\ol
w$ is not exclusively maximal in the non-singleton class.
\end{enumerate}
The ideal $P$ is an \emph{associated prime ideal of~$\til$} if the
annihilator modulo $\ol P_P \subset \oQ_P$ is not the identity
congruence.
\end{defn}

\begin{conv}\label{c:witnessSpeak}
A (key) witness is a (key) witness for some prime ideal~$P$.  When we
speak of the set of (key) witnesses for a given congruence we mean the
set of pairs $(w,P)$ where $w\in Q$ is a (key) witness for a prime
ideal~$P \subset Q$.  If the con\-gruence~$\til$ is not clear from
context, a (key) witness may be called a (key)~$\til$\mbox{-witness}.
\end{conv}

\begin{lemma}\label{l:associated}
A prime ideal $P \subset Q$ is associated to a congruence~$\til$
on~$Q$ if and only if $Q$ has a key witness for~$P$.
\end{lemma}
\begin{proof}
Once the annihilator $\ann(\ol P_P)$ does not equal the identity
congruence, it has a class of size $2$ or more; at least one element
therein avoids being exclusively~maximal.$\!\!$
\end{proof}

\begin{defn}\label{d:aide}
Fix the notation of Definition~\ref{d:witness}.
\enlargethispage{.5ex}
\begin{enumerate}
\item%
An \emph{aide}%
\footnote{The English word ``aide'' is fortuitously a transliteration
of the Hebrew word for ``witness''.  In talmudic courts, a pair of
witnesses was required for any conviction.}
for a witness $w$ and a generator $p \in P$ is an element $w' \in Q$
whose image~$\ol w' \in \oQ_P$ is (i)~distinct from~$\ol w$, but
(ii)~congruent to~$\ol w$ in the kernel of the cover
morphism~$\phi_p$, and (iii)~maximal (under Green's preorder) in
the~set~$\{\ol w, \ol w'\}$.
\item%
A \emph{key aide} for a key witness~$w$ is an element $w' \in Q$ whose
image~$\ol w' \in \oQ_P$ is (i)~distinct from~$\ol w$, but
(ii)~congruent to~$\ol w$ in the intersection of the kernels of all
cover morphisms, and (iii)~maximal (under Green's preorder) in the set
$\{\ol w, \ol w'\}$.
\end{enumerate}
\end{defn}

\begin{lemma}\label{l:aide}
Every witness for~$P$ and generator $p \in P$ has an aide.  Every key
witness has a key aide.
\end{lemma}
\begin{proof}
In each case, there is a non-singleton class containing $\ol w \in
\oQ_P$, so there exists an element $\ol w' \neq \ol w$ in this class.
The point is to choose $\ol w'$ so that it does not precede~$\ol w$
under Green's preorder and so that $\ol w'$ lies in the image of the
composite morphism $Q \to \oQ \to \oQ_P$.  The existence of $\ol w'$
not preceding~$\ol w$ is a consequence of $\ol w$ not being
exclusively maximal.  Now use that every element of~$\oQ_P$ is off
from the image of~$\oQ$ by an element of~$P$, and that $Q \to \oQ$ is
surjective.
\end{proof}

\begin{remark}\label{r:witness}
Every key witness is a witness, because any key aide is an aide for all
generators of~$P$.
\end{remark}

\begin{remark}\label{r:aide}
An aide $w'$ for a witness~$w$ and $p \in P$ can be a witness but need
not~be:
\begin{itemize}
\item%
adding $\ol p$ could join $\ol w$ to~$\ol w'$ while some other element
of~$P$ fails to join~$\ol w$~to~$\ol w'$;
\item%
$\ol w'$ can be exclusively maximal in its class under the kernel of
the cover morphism.
\end{itemize}
Similarly, a key aide can be a witness (and hence a key witness) but
need not be; howev\-er, in the key case only the second circumstance
(i.e., exclusive maximality)~\mbox{can occur}.
\end{remark}

In the set of (key) witnesses for a congruence, a single $w \in Q$ can
occur multiple times for different~$P$.  For instance, this happens
when $\nothing$ is associated.

\begin{example}\label{e:groupWitn}
The condition for an element to be a witness for the empty prime
ideal~$\nothing$ is vacuous: there are no cover morphisms.
Furthermore, the congruence $\ann(\nothing)$ in the definition of key
witness is an empty intersection of congruences, so it is the
universal congruence on~$\oQ_\nothing$.  Thus the empty ideal is
associated to a congruence if and only if the universal group
$\oQ_\nothing$ of the quotient modulo that congruence is nontrivial,
and that occurs precisely when $\oQ$ has no nil (see
Example~\ref{e:noNil}).  Every $q \in Q$ is a (key) witness in this
case but at the same time $\oQ_\nothing$ has only one class under
Green's~relation.
\end{example}

The following series of examples demonstrates various features of
associatedness of prime ideals and their witnesses.

\begin{example}\label{e:assprim}
As usual it will be convenient to describe congruences on $\NN^n$ by
unital binomial ideals in polynomial rings.  We use $e_x,
e_y,\dots$ to denote the generators of $\NN^n$ corresponding to
variables $x,y,\dots$ in the polynomial ring $\kk[\NN^n]$, but we
denote the addition morphisms by $\phi_x, \phi_y,\ldots$ instead of
$\phi_{e_x}, \phi_{e_y},\ldots$, for simplicity.
\begin{enumerate}
\item\label{e:witt1}%
Let $\til$ be the congruence on $\NN^2$ induced by the binomial ideal
$\<x^2-xy, xy-y^2\> \subset \kk[x,y]$.  The set of associated prime
ideals in $\NN^2$ consists of the empty ideal $\nothing$ and the
maximal ideal $P = \<e_x,e_y\>$.  Localization at the maximal ideal
does nothing and there are only two cover morphisms, given by adding
$e_x$ and $e_y$, respectively.  To establish that $P$ is associated,
note that $e_x$ and $e_y$ themselves are key witnesses for $P$,
congruent under $\ann(P)$, and serve as aides for one another.  Indeed,
$\ann(P)$, the intersection of the two kernels, contains the pair
$(e_x,e_y)$ since $e_x + e_x \sim e_y + e_x$ and also $e_x + e_y \sim
e_y + e_y$.  The identity $0\in \NN^2$ is not a witness for~$P$.
Neither $\<e_x\>$ nor $\<e_y\>$ is associated since adjoining inverses
to either turns the quotient $\NN^2/\til$ into a cancellative monoid.
In this case all kernels of addition morphisms are trivial.  Finally,
localizing at the empty prime ideal amounts to considering the induced
congruence on~$\ZZ^2$, which is induced by the binomial ideal $\<x-y\>
\subset \kk[x^\pm, y^\pm]$.  Since the quotient is nontrivial,
$\nothing$ is associated too.  Every element of $\NN^2$ is a witness
for~$\nothing$, but taken together they form only one Green's class
in~$\ZZ^2$.
\item%
Let $\til$ be the congruence on $\NN^3$ induced by $\<x^2 -xy, y^2-xy,
x(z-1)\> \subset \kk[x,y,z]$.  The associated prime ideals are
$\<e_x,e_y\>$ and $\nothing$.  The argument for $\nothing$ is the same
as in item~\ref{e:witt1}.  The localization of $\til$ at $\<e_x,
e_y\>$ is induced by the same ideal, considered in $\kk[x,y,z^\pm]$.
This says that $e_z$ is cancellative; i.e.~that the addition morphism
$\phi_z : q \mapsto q + e_z$ is injective.  The set of key witnesses
is invariant under the $\phi_z$-action.  It consists of $e_y + ke_z$
and $e_x + ke_z$ for $k\in\NN$.  The translates of $e_y$ all become
equivalent when adding $e_x$ or $e_y$.  Any translates of $e_x$ are
witnesses since they are each joined to a translate of $e_y$.  No
$e_z$-translate of $0$ is a witness, though.  Again, all witnesses
are~key.
\item\label{e:witt2nonKey}%
Let $\til$ be the congruence on $\NN^4$ induced by $\<x^2 -xy, y^2-xy,
x(z-1), y(w-1)\> \subset \kk[x,y,z,w]$.  The associated prime ideals
are again $\nothing$ and $P = \<e_x, e_y\>$.  The set of witnesses
for~$P$ is determined as follows.  The element $0\in\NN^4$ is a
witness that is not key.  The kernel congruences of $\phi_x$ and
$\phi_y$ are generated by $\{(0,e_z), (e_x, e_y)\}$ and $\{(0,e_w),
(e_x, e_y)\}$ in~$\NN^4 \times \NN^4$, respectively.  This shows the
witness property and also, because their common refinement leaves it
singleton similarly to~Example~\ref{e:irreducible}, that $0$ is not
key.  In contrast, $e_x$ and $e_y$ are key witnesses because $\phi_x
(e_x) = \phi_x (e_y)$ and likewise for $\phi_y$.  
A mesoprimary decomposition (Theorem~\ref{t:mesodecomp'}) of the
binomial ideal defining~$\til$ has components corresponding to all
three witnesses, while a mesoprimary decomposition of the congruence
$\til$ itself needs components only for the two key witnesses
(Theorem~\ref{t:coprincipal}).  Why the extra binomial component?  The
common refinement of the congruences induced by $\<z-1 , x^2, y\>$ and
$\<w-1, x, y^2\>$ leaves the class of $0$ singleton, but the
intersection of the ideals is merely free of binomials, rather than
being altogether zero.
\end{enumerate}
\end{example}

This next example demonstrates how the monoid prime ideal $P$ matters
in the definition of a (key) witness for~$P$, and how the same element
can be a witness for~different~$P$.

\begin{example}\label{e:samewitness}
Fix the congruence $\til$ induced on~$\NN^4$ by the unital binomial
ideal $\<x(z-1),x(w-1), y(z-1), y^2\> \subset \kk[x,y,z,w]$.  The
associated prime ideals of $\til$ are $\<e_x,e_y\>$ and $\<e_y\>$.
Consider the addition morphisms $\phi_x$ and $\phi_y$.  The key
witnesses for $\<e_y\>$ are $e_y + ke_x $ and all their translates in
the $e_z$ and~$e_w$ directions.  No element in the ideal $\<e_x\>$ can
be a witness for a monoid prime containing $e_x$ because $\phi_x$ acts
injectively on that ideal.  Indeed, the witnesses for $\<e_x, e_y\>$
are $0 \in \NN^4$ together with all its translates in the $e_z$
direction, and $e_y$ together with its translates in the $e_z$
and~$e_w$~directions.
\end{example}

The final example on witnesses demonstrates the prohibition on
exclusive maximality, which in particular bars $\nil$ and idempotents
from being witnesses.  See Remark~\ref{r:exclusively-maximal} for a
deeper explanation of the ban on exclusive maximality.

\begin{example}\label{e:exclusively-maximal}
Let $P = \<e_x,e_y\>$ be the maximal ideal of~$\NN^2$.
\begin{enumerate} 
\item%
Under the Rees congruence induced by the monomial ideal $\<x^2,y^2\>$,
the element $e_x + e_y$ is joined to nil under both cover morphisms.
Only $e_x + e_y$ a $P$-witness, and in fact a key witness.  In
contrast, $\nil$ is a key aide but not a witness and hence certainly
not a key witness.
\item%
Under the congruence induced by the unital binomial ideal $\<y,
x^2-x\>$, both cover morphisms join the identity $0$ to~$e_x$.
However, only the identity is a witness, because $e_x$ lies in the
ideal that $0$ generates.
\end{enumerate}
\end{example}

\begin{lemma}\label{l:ann}
If $P$ is maximal among the prime ideals associated to the components
in a primary decomposition, then $\ann(P)$ refines all $P'$-primary
components with~$P' \subsetneq P$.
\end{lemma}
\begin{proof}
Fix a $P'$-primary component $\app$ with $P' \subsetneq P$, and choose
$p \in P \minus P'$, so that $\ol p \in Q/\app$ is cancellative.  By
definition, if $a,b \in Q$ are congruent modulo $\ann(P)$ then $a + p$
and $b + p$ are congruent modulo the original congruence, so $a + p
\approx b + p$, and therefore $a \approx b$ by the cancellative
property of~$\ol p$.  Thus $\ann(P)$ refines~$\app$.
\end{proof}

\begin{lemma}\label{l:qquniv}
For all primes $P \not \supseteq P'$, the congruence on~$Q_P$
induced by any $P'$-primary congruence on~$Q$ is universal
on~$Q_P$.
\end{lemma}
\begin{proof}
Localization adjoins an inverse for a nilpotent element.
\end{proof}

Despite the oddities in Example~\ref{e:irreducible}, primary
decomposition of congruences is combinatorially well behaved: the
associated prime ideals of a congruence reflect which components are
necessary in every primary decomposition.

\begin{thm}\label{t:primary}
A prime $P \subset Q$ is associated to a congruence~$\til$ on~$Q$ if
and only if every primary decomposition of~$\til$ has a $P$-primary
component.  Moreover, if $P$ is not associated to~$\til$, then every
$P$-primary component in every primary decomposition of~$\til$ is
\emph{redundant}: omitting it leaves another primary decomposition
of~$\til$.
\end{thm}
\begin{proof}
Suppose that a primary decomposition with no $P$-primary component is
given.  Working modulo $\til$, assume that the congruence to be
decomposed is the identity congruence on~$Q$.  After localizing
along~$P$, the induced congruences on~$Q_P$ form a primary
decomposition of the identity congruence there by
Theorem~\ref{t:commutes}, with all $P'$-primary components for $P'
\not\subseteq P$ being universal and thus redundant by
Lemma~\ref{l:qquniv}.  That is to say, we can assume that $P$ is the
maximal monoid prime ideal of~$Q$.  Since the primary decomposition
has no $P$-primary component, Lemma~\ref{l:ann} implies that $\ann(P)$
refines all primary components, and thus it refines their
intersection.  Thus $\ann(P)$ is trivial and $P$ is not associated.

To prove the rest of the statement, it suffices to show that $P$ is an
associated prime of~$\til$ if some primary decomposition of~$\til$ has
a $P$-primary component $\til_P$ that is \emph{irredundant} in the
sense that omitting $\til_P$ yields a coarser congruence than~$\til$.
Write $\app$ for the (not necessarily primary) common refinement of
all other congruences in the decomposition.  Thus $\til_P
\hspace{.1ex}\cap\hspace{.1ex} \app$ is a nontrivial decomposition of
the identity congruence.  Choose~$a \neq\nolinebreak b \in Q$ with
$a\approx b$ but $a \not \sim_P b$.  Let $T =\{t\in Q \mid t+a \sim_P
t+b\}$.  Since $\til_P$ is $P$-primary, the radical of~$T$ is~$P$.
Modulo Green's relation on $Q_P$, find a maximal element $\hat{t}$ not
in the image of~$T$.  If $t \in Q$ maps to~$\hat{t}$ then the images
of $t+a$ and $t+b$ in $Q_P$ are joined under each cover morphism.
Therefore their class is non-singleton under $\ann(P)$, so one of them
is a key witness for~$P$.
\end{proof}

Theorem~\ref{t:primary} implies a natural characterization of primary
congruences.

\begin{cor}\label{c:primary}
A congruence is primary if and only if it has exactly one associated
prime ideal.
\end{cor}

\begin{remark}
Via the Rees congruence construction, primary decomposition of
congruences is a refinement of primary decomposition of ideals in
monoids.  There is an extensive literature on the second type of
decomposition surveyed in~\cite{andersonSurvey}.  Our definitions are
aligned with those in the literature: the Rees congruence of a monoid
ideal is primary if and only if that monoid ideal is primary.  In this
case its unique associated monoid prime ideal is the unique associated
monoid prime ideal of the congruence.
\end{remark}

\section{Associated prime congruences}\label{s:ass}

Each primary congruence on a finitely generated commutative monoid $Q$
has a unique associated prime ideal.  One of the most basic insights
in this paper is that a single primary congruence can have several
associated prime congruences.  The first definition says how a
congruence looks near a given $q \in Q$.

\begin{defn}\label{d:prime-component}
Fix a prime ideal $P \subseteq Q$, a congruence~$\til$ on~$Q$, and an
element $q \in Q$.  The \emph{$P$-prime congruence} of~$\til$
\emph{at~$q$} is the kernel of the morphism $Q \to \big(\<\ol
q\>/\<\ol q + P\>\big){}_P$ induced by the quotient $Q \to Q/\til =
\oQ$, addition $\phi_\ol q: \oQ \to \<\ol q\>$, and localization
at~$P$.
\end{defn}

\begin{defn}\label{d:mesoass}
A prime congruence $\app$ on $Q$ is \emph{associated} to an arbitrary
congruence~$\til$ if $\app$ equals the $P$-prime congruence of~$\til$
at a key witness for~$P$.
\end{defn}

\begin{remark}
The definition implies that the associated prime $P$ of $\app$ is
associated to~$\til$ too.  If $P$ is clear from the context, such as
after $\app$ is fixed, then we also speak of a key witness for $P$
simply as a \emph{key witness}.
\end{remark}

\begin{lemma}\label{l:samePprime}
If $p, q\in Q$ are equivalent under Green's relation, that is, if
$\<p\> = \<q\>$, then their $P$-prime congruences agree for each~$P$.
\end{lemma}
\begin{proof}
The same argument as for Lemma~\ref{l:cover} applies.
\end{proof}

\begin{example}\label{e:samewitness2}
In the situation of Example~\ref{e:samewitness}, the associated prime
congruences are induced by the ideals $\<x,y\>$, $\<x,y, z-1\>$, and
$\<y, z-1, w-1\>$.  The first two correspond to witnesses for $\<e_x,
e_y\>$, while the third corresponds to all of the witnesses
for~$\<e_y\>$.
\end{example}

The following and Lemma~\ref{l:poset} are the central finiteness
results, reflected in all of the following development, particularly
Theorem~\ref{t:coprincipal}.

\begin{thm}\label{t:finiteAss}
Fix a congruence~$\til$ on a finitely generated commutative
monoid~$Q$.  For each of the finitely many primes~$P$ of~$Q$, the key
$\til$-witnesses for~$P$ generate only finitely many Green's classes
in the localization $Q_P$ along~$P$.  Consequently, each congruence
on~$Q$ has only finitely many associated prime congruences.
\end{thm}
\begin{proof}
Since the definition of key witness for~$P$ is already local, it
suffices to treat the case where $P$ is the maximal ideal of~$Q$.
Form a relation on~$Q$ by joining every key witness~$w$ to a key
aide~$a$.  This relation is a congruence by definition of key witness
and key aide.  The claim about Green's classes holds because $Q$ is
noetherian.  To prove the consequence for associated prime
congruences, use Lemma~\ref{l:samePprime}.
\end{proof}

\begin{example}
The congruence in Example~\ref{e:mesoprimary} is primary with
respect to the maximal ideal.  The (key) witnesses are $e_x, e_y$, and
also $2e_x, e_x+e_y$, and~$2e_y$, since their class gets joined to nil
under $\phi_x$ and $\phi_y$.  Although the witnesses look
combinatorially different, the only associated prime congruence is the
identity congruence on the monoid $\{0,\infty\}$.  This is forced, as
the identity is the only cancellative element in~$\oQ$.
\end{example}

If on $Q$ the identity congruence is primary, then the assignment of
witnesses to their $P$-prime congruences is order preserving.  It
would be interesting to understand which posets of witnesses and
associated prime congruences can occur (Problem~\ref{p:assmeso}).

\section{Characterization of mesoprimary congruences}\label{s:mesoprimary}

In parallel with the theory of ordinary primary ideals in commutative
rings, the mesoprimary condition admits a characterization in terms of
associated prime congruences.  Definition~\ref{d:prim*} was made with
this proposition in mind.

\begin{thm}\label{t:oneassprim}
A congruence is mesoprimary if and only if it has exactly one
associated prime congruence.
\end{thm}
\begin{proof}
Fix a $P$-primary congruence~$\til$ on~$Q$.  If $\til$ is mesoprimary
and $w$ is not nil, then the $P$-prime congruence of~$\til$ at~$w$
coincides with the $P$-prime congruence of~$\til$ at the identity
because $\ol w$ is partly cancellative.  The uniqueness of the
associated prime congruence follows from the special case where $w$ is
a key witness.

On the other hand, assume $\til$ has a unique associated prime
congruence.  Then $\til$ is primary by Corollary~\ref{c:primary}.
Replacing $Q$ with $\ol Q$, assume $\til$ is the identity congruence
on~$Q$.  Suppose that $a$ and $b$ are distinct elements whose images
in $Q_P$ satisfy $a + u = b$ for some unit $u \in Q_P$.
Using the partial order from Lemma~\ref{l:poset}, let
$w \in Q$ be any element such that $w + a \neq w + b$ and the image
of~$w$ modulo the cancellative elements $F \subseteq Q$ is maximal
with this property.  Let $w'$ be any maximal non-nil element whose
image in the poset $Q/F$ is comparable to~$w$ but not below.  The
choices of~$w$ and $w'$ make them both key witnesses: $w'$ has
$\infty$ as an aide, and $w$ is verified directly to be a key witness
since $p + w = p + (w + a - b)$ in~$Q_P$ for all $p \in P$.  Replacing
$w'$ with $w' + c$ for some cancellative element~$c$ if necessary,
assume that $w' = w + q$ for some $q \in Q$.  Uniqueness of the
associated prime congruence, combined with the relation $\phi_{w'} =
\phi_q \circ \phi_w$ among addition morphisms, implies that $w' + a
\neq w' + b$.  By maximality of $w'$ in $Q/F$, the relation $v+a =
v+b$ can only hold for $v$ such that $v+a = \nil$.  Thus $\til$ is
partly cancellative.
\end{proof}

\begin{remark}
A primary congruence has only one associated monoid prime ideal by
Corollary~\ref{c:primary}.  Theorem~\ref{t:oneassprim} makes precise
the notion that further decomposition along the associated prime
congruences is natural, as is visible already in Example~\ref{e:msri}.
\end{remark}

Quotients by mesoprimary congruences can be described fairly
explicitly in terms related to the action in Lemma~\ref{l:action}.
Making this description into a precise alternative characterization of
mesoprimary congruences requires some specialized notions involving
monoid actions.

\begin{defn}\label{d:semifree}
The action of a monoid $F$ on an $F$-module~$T$ is \emph{semifree} if
\begin{itemize}
\item%
$\hspace{.25ex}t \mapsto f + t$ is an injection $\hspace{.2ex}T
\into T$ for all $f \in F$, and
\item%
$f \mapsto f + t$ is an injection $F \into T$ for all
$\hspace{.25ex}t \in T$.
\end{itemize}
\end{defn}

\begin{remark}
The letter ``$F$'' stands for ``face'': in practice, the monoid $F$ is
often a face of an affine semigroup, and thinking of it that way is
good for intuition.
\end{remark}

\begin{lemma}\label{l:semifree}
An action of a cancellative monoid $F$ on an $F$-module~$T$ is
semifree if and only if the localization map $T \into T_\nothing$ is
injective and the universal group~$F_\nothing$ acts freely
on~$T_\nothing$.
\end{lemma}
\begin{proof}
The cancellative condition means that the natural map $F \into
F_\nothing$ is injective.  Using this fact, the ``if'' direction is
elementary, and omitted.  In the other direction, the semifree case,
the first injectivity condition guarantees that $t - f = t' - f' \iff
f' + t = f + t'$.  In particular, $t - 0 = t' - 0 \iff t = t'$, so the
natural map $T \into T_\nothing$ is injective.  The second injectivity
condition guarantees that the action of~$F_\nothing$ is free: $(f -
f') + (t - w) = t - w \iff (f + t) - (f' + w) = t - w \iff (w + f) + t
= (f' + w) + t$, and by the second injectivity condition this occurs
if and only if $f + w = f' + w$, which is equivalent to $f = f'$
because $F$ is cancellative.
\end{proof}

In contrast to group actions, monoid actions need not define
equivalence relations, because the relation $t \sim f + t$ can fail to
be symmetric.  The relation is already reflexive and transitive,
however, precisely by the two axioms for monoid actions.

\begin{defn}\label{d:orbit}
An \emph{orbit} of a monoid action of~$F$ on~$T$ is an equivalence
class under the symmetrization of the relation $\{(s,t) \mid f + s =
t$ for some $f \in F\} \subseteq T \times T$.
\end{defn}

Combinatorially, from an $F$-module~$T$, one can construct a directed
graph with vertex set~$T$ and an edge from $s$ to~$t$ if $t = f + s$
for some $f \in F$.  Then an orbit is a connected component of the
underlying undirected graph.

\begin{cor}\label{c:one}
A congruence~$\til$ on a finitely generated commutative monoid~$Q$ is
mesoprimary if and only if the set $F$ of non-nilpotent elements
in~$\oQ = Q/\til$ is a cancellative monoid that acts semifreely on
$\oQ \minus \{\nil\}$ with finitely many orbits.
\end{cor}
\begin{proof}
Whether we assume the mesoprimary condition on~$\til$ or the condition
on the non-nilpotent elements in~$\oQ$, we can in each case deduce
that $\til$ is $P$-primary for some prime $P \subset Q$.  The image
of~$Q\minus P$ in~$\oQ$ is the cancellative submonoid~$F$ by
definition, which has finitely many orbits by Lemma~\ref{l:poset}.
The only feature of the corollary's statement that distinguishes
mesoprimary congruences from general primary ones is semifreeness,
which we claim is equivalent to uniqueness of the associated prime
congruence in Theorem~\ref{t:oneassprim}.  Indeed, $F$ acts semifreely
if and only if the $P$-prime congruences at all non-nil elements
of~$\oQ$ coincide.  Those coincidences certainly imply that the
$P$-prime congruences at all witnesses coincide, in which case $\til$
is mesoprimary.  On the other hand, if $\til$ is mesoprimary, then the
$P$-prime congruences at all key witnesses coincide.  They all
coincide with the $P$-prime congruence at the identity, or else there
would be two key witnesses, one sharing its $P$-prime congruence with
the identity and the other not.  Since the image in $Q/F$ of every
non-nil element of~$Q$ lies between the identity and a key witness,
the $P$-prime congruence of every non-nil element is forced to agree
with the one shared by the identity and the key witnesses.
\end{proof}

\begin{remark}\label{r:semifree}
As the proof of Corollary~\ref{c:one} shows, one interpretation of the
structure theorem in the statement is that a $P$-primary congruence
has the same $P$-prime congruence at every non-nil element as soon as
it has the same $P$-prime congruence at every key witness, and that is
what it means to be mesoprimary.
\end{remark}

\begin{prop}\label{p:mesorefine}
Given a finite set of congruences on~$Q$, all of which are
meso\-primary with the same associated prime congruence, their common
refinement is also mesoprimary with the same associated prime
congruence.
\end{prop}
\begin{proof}
Let $\til$ be the common refinement of finitely many $P$-mesoprimary
congruences.  Then $\til$ is $P$-primary by
Proposition~\ref{p:pprefine}.  Applying Theorem~\ref{t:oneassprim}, it
suffices to show that the $P$-prime congruence of~$\til$ at any
element $q \in Q$ that lies outside the nil class of~$\til$ is the
same as the $P$-prime congruence of~$\til$ at the identity.

Lemma~\ref{l:commutes} implies that we may assume $P$ is the maximal
ideal of~$Q$ with unit group $G = Q \minus P$.  Under each of the
given mesoprimary congruences, Corollary~\ref{c:one} (in the guise of
Remark~\ref{r:semifree}) implies that the class of~$q$ is either nil
or its intersection with the orbit $G + q$ equals $K + q$, where $K
\subseteq G$ is the subgroup that stabilizes (fixes as a set, but not
necessarily pointwise) the class of the identity under each of the
mesoprimary congruences.  Since the nil class contains $K + q$ once it
contains~$q$, the class of~$q$ under~$\til$ is either nil or its
intersection with the orbit $G + q$ equals $K + q$.  Having excluded
nil by our choice of~$q$, the intersection must be $K + q$.  Thus the
$P$-prime congruence at~$q$ under~$\til$ coincides with the $P$-prime
congruence at~$q$ under (every) one of the mesoprimary
congruences~modulo which $q$~is~not nil.

In particular, letting $q$ be the identity shows that $K$ is the
intersection of the identity class of~$\til$ with~$G$.  Consequently,
the $P$-prime congruence of~$\til$ at~$q$ coincides with the $P$-prime
congruence of $\til$ at the identity, as desired.
\end{proof}

\section{Coprincipal congruences}\label{s:coprincipal}

In commutative rings, irreducible decomposition underlies primary
decomposition.  Analogously, coprincipal decomposition underlies
mesoprimary decomposition of commutative monoid congruences (but see
the remarks and examples after~Theorem~\ref{t:coprincipal}).

\begin{defn}\label{d:peaks}
A \emph{peak} of a monoid $Q$ is a non-nil element $q \in Q$ such that
$q + a = \nil$ for all nonunit $a \in Q$.  The \emph{cogenerators} of
a $P$-primary congruence on~$Q$ are the elements of~$Q$ whose images
in $\oQ_P$ are peaks.
\end{defn}

\begin{defn}\label{d:coprincipal}
A congruence~$\til$ on~$Q$ is \emph{coprincipal} if it is
$P$-mesoprimary for some monoid prime $P$ and additionally the
quotient of $\oQ_P$ modulo its Green's relation has precisely one
peak.
\end{defn}

\begin{example}\label{e:peaks}
The congruence in Example~\ref{e:mesoprimary} is coprincipal.  It is
$P$-mesoprimary for $P = Q\minus \{0\}$ and its unique peak is the
class of~$2$.  
\end{example}

\begin{defn}\label{d:orderideal}
Fix a congruence on~$Q$ with quotient $\oQ$.  The \emph{order ideal
$\Qlqp$ cogenerated by $q \in Q$ at a prime ideal $P \subset Q$}
consists of those $a \in Q$ whose image precedes that of~$q$ in the
partially ordered quotient of~$\oQ_P$ modulo its Green's
relation~(Lemma~\ref{l:modGreen}).
\end{defn}

\begin{example}\label{e:GreenOrderIdeal}
Let $\til$ be the congruence on $\NN$ induced by the binomial ideal
$\<x^3-x^6\> \subset \kk[x]$.  Set $P = \<e\>$, where $e = e_x$ is the
generator of~$\NN$.
\begin{enumerate}
\item%
The order ideal $\NN_{\preceq e}^P$ consists of~$e$ itself and $0 \in
\NN$.
\item%
Including $2e$ yields the order ideal $\NN_{\preceq 2e}^P =
\{0,e,2e\}$.
\item%
The order ideals $\NN_{\preceq q}^P$ for $q = m e$ with $m \geq 3$ all
coincide with~$\NN$ itself.  Thus, in general, order ideals $\Qlqp
\subseteq Q$ need not be finite, although their images in~$\oQ_P$
modulo Green's relation always are.
\item%
The order ideals $\NN_{\preceq q}^\nothing$ for $q \in \NN$ all
coincide with~$\NN$ itself.
\end{enumerate}
\end{example}

\begin{example}\label{e:anotherOrderIdeal}
Let $\til$ be the identity congruence on $Q = \NN^3$, and set $P =
\<e,f\>$, where $e,f$ are two of the three generators of~$\NN^3$, the
third being~$g$.  The order ideal $Q_{\preceq e + f + 2g}^P$ consists
of the lattice points on the nonnegative $g$-axis together with their
translates by $e$, $f$, and $e + f$.  The answer would have been the
same had $e + f + 2g$ been replaced by $e + f$, or $e + f + g$, or $e
+ f + mg$ for any $m \in \NN$.
\end{example}

\begin{defn}\label{d:cogenerated}
Fix a congruence~$\til$.  The congruence \emph{cogenerated by~$q$
along~$P$} is the coarsening~$\til_q^P$ of~$\til$ obtained by first
joining any pair of elements in $Q \minus \Qlqp$ and also joining any
pair $(a,b) \in Q$ such~that
\begin{enumerate}[\quad(i)]
\item%
the images $\ol a$ and $\ol b$ in~$\oQ_P$ differ by a unit in~$\oQ_P$
and
\item%
$\ol c + \ol a = \ol c + \ol b = \ol q \in \oQ_P$ for some $c \in
Q_P$.
\end{enumerate}
\end{defn}

\begin{example}\label{e:need-not-be-mesoprimary}
The congruence $\til_q^P$ in Definition~\ref{d:cogenerated} need not
be primary, and hence it need not be coprincipal.  Essentially, the
prime $P$ has to be small enough to foster the mesoprimary condition.
In Example~\ref{e:exclusively-maximal}.2, the congruence cogenerated
by $q = e_x$ along $P' = \{e_x,\nil\}$ is not primary.  However, along
$P = \{\nil\}$ localization inverts more, causing $e_x$ to be joined
with~$0$, resulting in a primary---and hence
\mbox{coprincipal---congruence}.
\end{example}

\begin{prop}\label{p:coprincipal}
Fix a congruence~$\til$ and a witness $w$ for a prime~$P$.  All
elements of~$P$ are nilpotent modulo the congruence~$\til_w^P$, whose
nil class is $Q \minus \Qlwp$.
\end{prop}
\begin{proof}
Given an aide $w'$ for $w$ and a generator $p$ of~$P$, one of two
things must happen, and in both cases $p + w$ is nil
modulo~$\til_w^P$.  Write $[q]$ for the Green's class of $\ol q \in
\oQ_P$.
\begin{enumerate}
\item%
$[w] \neq [w']$.  In this case, either $[w] < [w']$ or $[w]$ and
$[w']$ are incomparable, but these both imply that $w'$ maps to nil
modulo the coprincipal congruence~$\til_w^P$, so $p + w = p + w'$ is
nil modulo~$\til_w^P$.
\item%
$[w] = [w']$; that is, their images lie in the same Green's class.  In
this case, $[p + w] > [w]$ by Lemma~\ref{l:bijective}, since addition
by $p$ joins $[w]$ to~$[w']$.
\end{enumerate}
Since $p$ is an arbitrary generator of~$P$, it follows that $P + w$ is
nil modulo~$\til_w^P$.  This implies that every element of~$P$ is
nilpotent modulo~$\til_w^P$, as follows.  There are only finitely many
Green's classes beneath~$[w]$, so the Green's classes of multiples of
any given nonunit element $a \in Q_P/\til_w^P$ are not all distinct:
there must be repeats.  Suppose $[\alpha \cdot a] = [\beta \cdot a]$
for some positive integers $\alpha < \beta$.  Every non-nil element of
$Q_P/\til_w^P$ precedes $\ol w$ in Green's preorder.  Therefore, if
neither $\alpha \cdot a$ nor $\beta \cdot a$ is nil, then there is
some $c \in Q$ such that $[\alpha \cdot a] + c = [w]$, whence
$$%
  [w]
  =
  [\beta \cdot a] + c
  =
  (\beta-\alpha) \cdot a + [\alpha \cdot a] + c
  =
  (\beta-\alpha) \cdot a + [w] \subseteq P + [w]
$$
is nil modulo~$\til_w^P$, contradicting the choice of~$w$.

The statement about the nil class holds because $Q \minus \Qlwp$ is an
ideal of~$Q$ (so its image is nil) that does not contain $w$ itself
(so the image of~$w$ is not made nil by the first relations in
Definition~\ref{d:cogenerated}) or any element in~$\Qlwp$ (so none of
the relations defined by (i) and~(ii) in
Definition~\ref{d:cogenerated} make $w$ or any other element
of~$\Qlwp$ nil).
\end{proof}

\begin{remark}\label{r:exclusively-maximal}
Proposition~\ref{p:coprincipal} can fail if $w$ is merely an
aide---even a key aide.  The not-exclusively-maximal property of a
witness guarantees existence of an aide that can be set congruent to
nil modulo the coprincipal congruence without forcing $w$ to be nil as
well.  In Example~\ref{e:exclusively-maximal}.2, for instance, there
is no way to define a coprincipal congruence cogenerated by~$e_x$ in
such a way that $e_x$ is nilpotent without it being~nil.
\end{remark}

\begin{thm}\label{t:P-coprincipal}
Given a congruence~$\til$, the congruence $\til_w^P$ cogenerated by
any witness~$w$ for~$P$ is coprincipal, with associated prime
ideal~$P$.
\end{thm}
\begin{proof}
Every nonunit in the localization $Q_P/\til_w^P$ of the quotient
monoid $Q/\til_w^P$ along~$P$ is nilpotent by
Proposition~\ref{p:coprincipal}.  The statement about the nil class in
that same proposition implies that the Green's class of~$w$ is the
unique peak.  The localization morphism $Q/\til_w^P \to Q_P/\til_w^P$
is injective by condition~(i) in Definition~\ref{d:cogenerated}.
Condition~(ii) there forces the $P$-prime congruence at the identity
to equal the $P$-prime congruence at~$w$, which consequently forces
the $P$-prime congruences at all non-nil elements to coincide, since
they lie between the $P$-prime congruences at the identity and at~$w$.
Therefore the action of the unit group of $Q_P/\til_w^P$ on its
non-nilpotent elements is free.  The proof is complete by
Corollary~\ref{c:one}, using the characterization of semifreeness in
Lemma~\ref{l:semifree}.
\end{proof}

\begin{defn}\label{d:component}
If $w$ is a witness for an associated $P$-prime congruence of~$\til$,
then the congruence~$\til_w^P$
is the \emph{coprincipal component} of~$\til$ cogenerated by~$w$
along~$P$.  If the prime ideal $P$ is clear from context, e.g.~if $w$
is already specified to be a witness for~$P$, then we simply speak of
the coprincipal component cogenerated~by~$w$.
\end{defn}

\begin{example}\label{e:coprincipalNotPrimary}
Consider the congruence on $\NN^2$ induced by $I=\<x^3-x^2,
y^3-y^2\>$.  The quotient $Q = \NN^2/\til_I$ has nine elements, with
the class of $2e_x + 2e_y$ being nil.  The quotient also has two
idempotents, namely the classes of $2e_x$ and $2e_y$.  Neither of the
congruences cogenerated by $q = e_x + 2e_y$ and $q = e_y + 2e_x$ along
$P = \<e_x,e_y\>$ is primary; however, these elements are not
$P$-witnesses.  In fact, there are no $P$-witnesses: the maximal ideal
is not associated.  In contrast, the coprincipal components for the
witnesses $(2e_2, \<e_1\>)$ and $(2e_1, \<e_2\>)$ are mesoprimary, as
per Theorem~\ref{t:P-coprincipal}.
\end{example}

\begin{example}\label{e:GreenOrderIdeal2}
In the setting of Example~\ref{e:GreenOrderIdeal}, the coprincipal
component of~$\til$ cogenerated by any~$q\in\NN$ along~$\nothing$ is
induced by the binomial ideal $\<1-x^3\>$.  The component cogenerated
by the key witness~$2e$ along $\<e\>$ is induced by the binomial
ideal~$\<x^3\>$.
\end{example}

\begin{prop}\label{p:orderideal}
Given any witness~$w$ for an associated $P$-prime congruence
of~$\til$, the coprincipal component of~$\til$ cogenerated by~$w$
along~$P$ is refined by~$\til$.
\end{prop}
\begin{proof}
Starting from $\til$ the coprincipal component is formed by
identifying additional pairs of elements.
\end{proof}

\begin{prop}\label{p:mesoCopr}
Any mesoprimary congruence~$\til$ equals the common refinement of the
coprincipal components of~$\til$ cogenerated by the cogenerators
of~$\til$.
\end{prop}
\begin{proof}
Fix a $P$-mesoprimary congruence~$\til$.  By
Proposition~\ref{p:orderideal} each coprincipal component at a
cogenerator coarsens~$\til$.  On the other hand, suppose that $q
\not\sim q'$.  Let $\ol q$ and $\ol q'$ denote their images in the
localized quotient $\ol Q_P$.  By mesoprimaryness, $\ol q \neq \ol
q'$.  Modulo Green's relation on $\ol Q_P$, every element precedes a
peak.  If exactly one of $\ol q$ and~$\ol q'$ precedes some peak~$\ol
w$, then modulo~$\til_w^P$ exactly one of $q$ and~$q'$ maps to nil, so
they are incongruent.  If no such peak exists, then $\ol q$ and~$\ol
q'$ both precede some peak~$\ol w$.  For $\ol q$ and~$\ol q'$ to be
joined by~$\til_w^P$ they must differ by a unit and satisfy $\ol q +
\ol c = \ol q' + \ol c = \ol w$ for some $c \in Q_P$, all by
Definition~\ref{d:cogenerated}.  However, since $\til$ is mesoprimary,
$\ol q + \ol c = \ol q' + \ol c$ implies that both sides are nil.
Consequently, $q \not \sim_w^P q'$.
\end{proof}

\section{Mesoprimary decompositions of congruences}\label{s:cong}

\begin{defn}\label{d:mesodecomp}
Fix a congruence $\til$ on a finitely generated commutative
monoid~$Q$.
\begin{enumerate}
\item%
An expression of~$\til$ as the common refinement of finitely many
mesoprimary congruences is a \emph{mesoprimary decomposition} if, for
each mesoprimary congruence~$\app$ that appears in the decomposition
with associated prime ideal $P \subset Q$, the $P$-prime congruences
of~$\til$ and~$\app$ at every cogenerator of~$\app$ coincide.
\item%
Each mesoprimary congruence that appears is a \emph{mesoprimary
component}~of~$\til$.
\item%
If every cogenerator of every $P$-mesoprimary component~$\app$ is a
key $\til$-witness for~$P$, then the decomposition is a \emph{key}
mesoprimary decomposition.
\end{enumerate}
\end{defn}

\begin{example}\label{e:irredNotComb}
According to Definition~\ref{d:mesodecomp} the decomposition in
Example~\ref{e:irreducible} is not a mesoprimary decomposition because
the intersectands are not components of the identity congruence: the
combinatorics at the witnesses for the mesoprimary congruences in the
decomposition do not agree with the combinatorics of the identity
congruence.  More precisely, the $\nothing$-prime congruence at each
element of~$\NN^2$ is the identity congruence, not the congruence
induced by $\<x-1\>$ or $\<y-1\>$.
\end{example}

\begin{thm}\label{t:mesodecomp}
Every congruence on a finitely generated commutative monoid admits a
key mesoprimary decomposition.
\end{thm}
\begin{proof}
Two examples are the decompositions in Theorem~\ref{t:coprincipal} and
Corollary~\ref{c:mesodecomp}, by Remark~\ref{r:allwitnesses} and
finiteness of the set of Green's classes of witnesses in
Theorem~\ref{t:finiteAss}.
\end{proof}

In the remainder of this section, Convention~\ref{c:witnessSpeak}
leads to some simplification of terminology.  The first statement to
benefit is our first main decomposition theorem (the other being
Corollary~\ref{c:mesodecomp}), which generalizes to arbitrary monoid
congruences the notion of irreducible decomposition for monoid ideals;
see Examples~\ref{e:irredundant} and~\ref{e:redundant}.

\begin{thm}\label{t:coprincipal}
Every congruence on a finitely generated commutative monoid is the
common refinement of the coprincipal congruences cogenerated by its
key witnesses.
\end{thm}
\begin{proof}
Fix a congruence $\til$ on~$Q$.  Proposition~\ref{p:orderideal}
implies that the intersection of all of the coprincipal congruences
for witnesses is refined by~$\til$.  On the other hand, suppose that
$q \not\sim q'$ for two elements $q,q' \in Q$.  The proof is done once
we find a prime $P \subset Q$ and a key witness $w \in Q$ whose
coprincipal congruence $\til_w^P$ on~$Q$ fails to join $q$ to~$q'$.

Let $T = \{t \in Q \mid t + q \sim t + q'\}$ be the ideal of elements
joining $q$ to~$q'$.  Fix a prime ideal~$P$ minimal among primes
of~$Q$ containing~$T$.  The images $\hat q$ and $\hat q'$ of $q$
and~$q'$ in the localization $Q_P$ remain incongruent because $P$
contains~$T$.  In contrast, every element in the localized image $T_P$
joins $\hat q$ to~$\hat q'$; that is, $\hat t + \hat q \sim \hat t +
\hat q'$ for all $\hat t \in T_P$.  Since the maximal ideal $P_P$
of~$Q_P$ is minimal over~$T_P$ by minimality of~$P$ over~$T$, there is
a maximal Green's class among those represented by the elements
$\{\hat t \in Q_P \mid \hat t + \hat q \not\sim \hat t + \hat q'\}$.
If the image of~$t$ lies in such a maximal Green's class, then in~$Q$
at least one of the elements $w = t + q$ and $w' = t + q'$---namely
one whose image in~$\oQ_P$ is not strictly greater than the other
under Green's preorder---is a key witness by definition.  Assuming, by
symmetry, that $w$ is a key witness, the localization of the
congruence~$\til_w^P$ satisfies $\hat q \not\sim_w^P \hat q'$, so $q
\not\sim_w^P q'$ before localization.
\end{proof}

\begin{remark}\label{r:allwitnesses}
In Theorem~\ref{t:coprincipal} it makes no difference whether one uses
all the key witnesses or just one per Green's class.  This follows
instantly from the definition of a coprincipal component; indeed, for
a given Green's class of key witnesses, the coprincipal components are
all equal---not just equivalent, but literally the same congruence.
\end{remark}

\begin{example}\label{e:irredundant}
For a monomial ideal in an affine semigroup ring, the coprincipal
decomposition of the Rees congruence afforded by
Theorem~\ref{t:coprincipal} arises equivalently from the Rees
congruences of the components in the unique irredundant irreducible
decomposition into monomial ideals \cite[Theorem~2.4]{irredres}; see
also \cite[Corollary~11.5 and Proposition~11.41]{cca}.
\end{example}

\begin{example}\label{e:redundant}
Unlike the case in Example~\ref{e:irredundant}, the decomposition in
Theorem~\ref{t:coprincipal} can be redundant in general.  This happens
for the congruence in Example~\ref{e:assprim}.\ref{e:witt1}.  The
decomposition produced by Theorem~\ref{t:coprincipal} has three
mesoprimary components: $\til_w^P$ for $P = \<e_x,e_y\>$ and $w \in
\{(e_x,0), (0,e_y)\}$ arise from joining $e_y$ and $e_x$,
respectively, to nil.  A third component $\til^\nothing$ arises for
$P=\nothing$ (with any element as a witness) and is induced by
$\<x-y\>$.  The decomposition into three congruences is redundant: the
given congruence is already the common refinement of~$\til^\nothing$
and either of~$\til^{P}_{w}$, the point being that once
$\til^\nothing$ is given, one only needs to separate $(1,0)$
from~$(0,1)$.  That said, the points $(1,0)$ and $(0,1)$ represent
distinct Green's classes of key witnesses for the associated prime
congruence induced by the binomial ideal $\<x,y\>$.  There is simply
no way of constructing an irredundant coprincipal decomposition
without breaking the symmetry: no systematic method of eliminating one
of the redundant components in this example would have a way to choose
between them.
\end{example}

\begin{remark}\label{r:coprDecCopr}
A coprincipal congruence can have more than one Green's class of key
witnesses, such as Example~\ref{e:mesoprimary}.  In any such case the
mesoprimary decomposition from Theorem~\ref{t:coprincipal} produces
more than one coprincipal component.  By Proposition~\ref{p:mesoCopr},
however, it is guaranteed that the original congruence appears as the
component for the Green's class of the unique peak, and thus all other
components are redundant.  This phenomenon prevents arbitrary
coprincipal congruences from accurately reflecting the combinatorics
of irreducible decomposition of binomial ideals.  One irreducible
decomposition of the coprincipal ideal $I = \<x^2-xy, xy-y^2, x^3\>$
from Example~\ref{e:mesoprimary} is $I = \<x-y, x^3\> \cap \<x^2, y\>
\cap \<x, y^2\>$, as can be seen by
applying~\cite[Proposition~3.1.7]{Vas}.
\end{remark}

\begin{remark}\label{r:irred}
Any irreducible congruence is mesoprimary: if a congruence is not
mesoprimary then it has at least two associated prime congruences by
Theorem~\ref{t:oneassprim}, and then it is reducible by mesoprimary
decomposition.  However, irreducible decompositions of congruences do
not, in general, reflect the combinatorics of congruences in a manner
that is witnessed combinatorially by the congruence itself.
\end{remark}

\begin{lemma}\label{l:peak}
Every cogenerator of the common refinement of a finite set of
$P$-meso\-primary congruences is a cogenerator of one of the given
mesoprimary \mbox{congruences}.
\end{lemma}
\begin{proof}
If $w$ is a cogenerator of the common refinement~$\til$, then $w$ is
not nil modulo~$\til$, so $w$ is not nil modulo (at least) one of the
given mesoprimary congruences.  On the other hand, $p + w$ is nil
modulo $\til$ for all $p \in P$, whence $P + w$ is nil modulo each one
of the given mesoprimary congruences.  Therefore $w$ is a cogenerator
of each of the given mesoprimary congruences modulo which it is not
nil.
\end{proof}

Combining Theorem~\ref{t:coprincipal} with
Proposition~\ref{p:mesorefine} and Lemma~\ref{l:peak} yields the
following, the culmination of our study of commutative monoid
congruence decompositions.

\begin{cor}\label{c:mesodecomp}
Every congruence on a finitely generated commutative monoid admits a
key mesoprimary decomposition with one component per associated prime
congruence.
\end{cor}

\begin{example}\label{e:non-ass}
In general the set of key witnesses is properly contained in the set
of witnesses.  Example~\ref{e:assprim}.\ref{e:witt2nonKey} shows one
way this can happen.  Exploiting the weirdness of irreducible
decomposition of the identity congruence is not necessary: consider
the primary congruence induced by the (cellular) binomial ideal
\begin{equation*}
  I = \<a^2-1, b^2-1, x(b-1), y(a-1), z(a-b), x^2, y^2, z^2\>.
\end{equation*}
The geometry of the quotient is shown here, where $\ZZ^\delta_2$ is
the diagonal copy of $\ZZ_2$ in $\ZZ_2\times \ZZ_2$, i.e.~generated by
$(1,1)$:
$$%
\psfrag{a}{\footnotesize $\,\ZZ_2^\delta$}
\psfrag{b}{\footnotesize \hspace{-4ex}$\ZZ_2 \times 0$}
\psfrag{c}{\footnotesize $0 \times \ZZ_2$}
\psfrag{d}{\footnotesize $0$}
\psfrag{x}{\footnotesize $x$}
\psfrag{y}{\footnotesize $y$}
\psfrag{z}{\footnotesize $z$}
\includegraphics{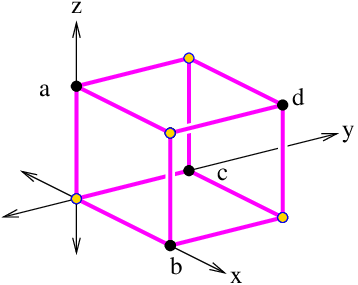}
$$
The solid dots indicate key witnesses and are labeled with quotients
of~$\oQ$ modulo the corresponding stabilizers, under the action from
Lemma~\ref{l:action}.  The origin is not a key witness because the
common refinement of the three kernels of addition morphisms is
trivial.  According to Theorem~\ref{t:coprincipal}, a coprincipal
mesoprimary decomposition of $\til_I$ is induced by the following
decomposition of~$I$ into unital binomial~ideals:
\begin{align*}
     I = \<a-1,b-1,z^2,y^2,x^2\> &\cap \<a^2-1,b-1,z,y,x^2\>\\
\mbox{}\cap\<a-1,b^2-1,z,x,y^2\> &\cap \<ab-1,a-b,y,x,z^2\>.
\end{align*}
The heart of the remainder of this paper---the ring-theoretic
part---is to make the corresponding decomposition of arbitrary
(non-unital) binomial ideals precise.  For reference, the primary
decomposition of~$I$ is
\begin{align*}
    I = \<a-1,b-1,z^2,y^2,x^2\> &\cap \<a+1,b-1,z,y,x^2\>\\
\mbox{}\cap \<a-1,b+1,z,x,y^2\> &\cap \<a+1,b+1,y,x,z^2\>.
\end{align*}
\end{example}


\section{Augmentation ideals, kernels, and nils}\label{s:aug}

One of our goals is to compare the combinatorics of congruences on a
commutative monoid~$Q$ in purely monoid-theoretic settings with their
ring-theoretic counterparts.  It is therefore important to note that
various binomial ideals $I \subset \kk[Q]$ can induce the same
congruence on~$Q$.  One way for this to happen is an arithmetic way,
via binomials involving the same monomials but different sets of
coefficients; this occurs for binomial primes~$\Irp$ whose
characters share their domain of definition (see
Section~\ref{s:components}).

\begin{example}\label{e:z+1}
Let $\mathrm{char}(\kk) \neq 2$.  In the polynomial ring $\kk[x,y,z]$,
both of the ideals $I = \<x(z-1),y(z-1),z^2-1,x^2,xy,y^2\>$ and $I' =
\<x(z-1),y(z+1),z^2-1,x^2,y^2\>$ induce the same congruence; note that
$I'$ contains $\<xy\>$, so the only difference between these two
ideals is the character on $\ZZ = \{0\} \times \{0\} \times \ZZ
\subseteq \ZZ \times \ZZ \times \ZZ$ induced by the monomials $y, zy,
z^2y, \ldots$ due to the generator $y(z+1)$ instead of~$y(z-1)$.
\end{example}

\noindent
Another way, demonstrated in parts~3 and~4 of
Example~\ref{e:pure-diff}, is combinatorial: when $Q$ has a
nil~$\nil$, the binomial ideal $\<\ttt^\nil\>$ induces the same
(trivial) congruence on~$Q$ as the zero ideal $\<0\> \subseteq
\kk[Q]$.  Nils are the only way for this to occur.

\begin{lemma}\label{l:nilmonomial}
Fix a binomial ideal $I \subseteq \kk[Q]$ whose congruence $\til_I$ is
trivial (every class is a singleton).  Then $I = 0$ or $I =
\<\ttt^\nil\>$ for a nil $\nil \in Q$.
\end{lemma}
\begin{proof}
If $I \neq 0$ then $I$ must be a monomial ideal with a unique
monomial, or else the congruence $\til_I$ has a class of size at
least~$2$.  Hence the result follows because a monoid can have at most
one nil.
\end{proof}

\begin{defn}\label{d:truncate}
If $\nil \in Q$ is a nil, then the \emph{truncated algebra} is
$\kk[Q]^- := \kk[Q]/\<\ttt^\nil\>$.  By convention, if $Q$ has no nil,
then we set $\kk[Q]^- := \kk[Q]$.
\end{defn}

\begin{remark}\label{r:nil}
Truncated algebras arise naturally from monoid algebras because of
differences in the way quotients of monoids and monoid algebras by
ideals are formed.  If $F\subseteq Q$ is a monoid ideal and $\til_F$
its Rees congruence, the quotient $\kk[Q] \to \kk[Q]/M_F$ modulo the
monomial ideal $M_F = \<\ttt^f \mid f \in F\>$ equals
$\kk[Q/\til_F]^-$ rather than $\kk[Q/\til_F]$ itself.  We shall see
that if $Q$ has a nil, then $\kk[Q]$ and $\kk[Q]^-$ reflect certain
aspects of the algebra of~$Q$ to varying degrees of accuracy.
\end{remark}

More generally, if the congruence induced by a (not necessarily
unital) binomial ideal~$I$ results in a quotient $Q/\til_I$
that has a nil, then throwing in monomials from the nil class results
in an ideal that determines the same congruence.

\begin{prop}\label{p:nil}
Fix a binomial ideal $I \subseteq \kk[Q]$.  The only binomial ideals
containing~$I$ that determine the same congruence~$\til_I$ are $I$
itself and, if $\oQ = Q/\til_I$ has a nil~$\nil$, the ideal \mbox{$I +
\<\ttt^q \mid \ol q = \nil\>$}, where the bar denotes passage from $q
\in Q$ to its image $\ol q \in \oQ$.
\end{prop}
\begin{proof}
Under the grading of the quotient algebra $\kk[Q]/I$ by $\oQ =
Q/\til_I$ the dimension of the graded piece $(\kk[Q]/I)_\ol q$ as a
vector space over~$\kk$ is either $0$ or~$1$, depending on whether $I$
contains a monomial in the corresponding class.  Since the (exponents
on) monomials in~$I$ form a single class, the dimension can only
be~$0$ for at most one~$\ol q$, and $\ol q$ must be a nil in~$\oQ$.
Now note that $\kk[Q]/I$ is close enough to the monoid algebra
$\kk[\oQ]$ for the argument from Lemma~\ref{l:nilmonomial} to work,
and lift the result from $\kk[Q]/I$ to~$\kk[Q]$.
\end{proof}

The two binomial ideals in Proposition~\ref{p:nil} are unequal
precisely when $I$ contains no monomials, and in this case it is
trivial to form the second ideal by inserting monomials.  In special
circumstances, it is possible to reverse this procedure.  To this end,
we wish to examine the transition from $\kk[Q]$ to the truncated
algebra~$\kk[Q]^-$ (when $Q$ has a nil) in terms of primary
decomposition of binomial ideals.  This naturally leads to the
following concept refining that of a nil.

\begin{defn}\label{d:ker}
A \emph{kernel} of a commutative monoid~$Q$ is a nonempty ideal
contained in all nonempty ideals of~$Q$.  (Such an ideal might not
exist.)
\end{defn}

\begin{example}\label{e:nil}
A nil is the same thing as a kernel of cardinality~$1$.
\end{example}

The existence of a nil in~$Q$, or a finite kernel more generally, is
reflected by a certain kind of maximal ideal of~$\kk[Q]$ being an
associated prime of~$\kk[Q]$.

\begin{defn}\label{d:aug}
Fix a commutative monoid~$Q$, and write $\kk^* = \kk \minus \{0\}$.
The \emph{unital augmentation ideal} in the monoid algebra~$\kk[Q]$ is
the ideal
$$%
  \Iaug^1 := \<\ttt^q - 1 \mid q \in Q\>
$$
generated by all monomial differences.  More generally, an
\emph{augmentation ideal} for a given binomial ideal~$I \subseteq
\kk[Q]$ is a proper ideal of the form
$$%
  \Iaug := \<\ttt^q - \lambda_q \mid q \in Q, \lambda_q \in \kk^*\>
  \subseteq \kk[Q],
$$
such that $I \cap \Iaug$ is a binomial ideal.
\end{defn}

\begin{example}
The ideal $I = \<x^2\> \subset \kk[x,y]$ induces a primary congruence
(a Rees congruence) identifying all monomials in $I$.  A compatible
augmentation ideal is $\Iaug = \<x-1,y-1\>$, which satisfies $I \cap
\Iaug = \<x^2-x^3, yx^2-x^2\>$.  This intersection induces the same
congruence~$\til$ as~$I$ does.  Note that $\kk[x,y]/(I\cap \Iaug)
\cong \kk[\NN^2/\til]$ is isomorphic to the semigroup algebra
of~$\NN^2/\til$ while $\kk[\NN^2]/I \cong \kk[\NN^2/\til]^-$ is the
truncated algebra.
\end{example}

\begin{lemma}\label{l:aug}
Given an augmentation ideal $\Iaug$ as in Definition~\ref{d:aug}, the
association $q \mapsto \lambda_q$ constitutes a monoid homomorphism
$\phi: Q \to \kk^*$.
\end{lemma}
\begin{proof}
The maximal ideals of~$\kk[Q]$ with residue field~$\kk$ are in
bijection with the monoid homomorphisms $Q \to \kk$;
Definition~\ref{d:aug} guarantees that the image lies in~$\kk^*$.
\end{proof}

\begin{prop}
Fix a monoid algebra $\kk[Q]$ over a field\/~$\kk$ with $Q$ finitely
generated.  An augmentation ideal is associated to~$\kk[Q]$ if and
only if $Q$ has a finite kernel, and in that case the unital
augmentation ideal is associated to~$\kk[Q]$.
\end{prop}
\begin{proof}
If $Q$ has a finite kernel~$K$, then $\Iaug^1$ is the annihilator of
the sum $f = \sum_{k \in K}\ttt^k$.  Indeed, $q + K \subseteq K$ is an
ideal of~$Q \implies q + K = K$ for all $q \in Q$ $\implies \ttt^q f =
f$ for all $q \in Q \implies (\ttt^q - 1) f = 0$ for all $q \in Q
\implies \Iaug^1 \subseteq \ann(f)$; but $\Iaug^1$ is a maximal ideal.

Now suppose that an augmentation ideal $\Iaug$ is associated
to~$\kk[Q]$.  The homomorphism $q \mapsto \lambda_q$ in
Lemma~\ref{l:aug} induces an automorphism of~$\kk[Q]$ that rescales
the monomials by $\ttt^q \mapsto \lambda_q\ttt^q$.  This automorphism
takes $\Iaug$ to~$\Iaug^1$.  Therefore, we may as well assume $\Iaug =
\Iaug^1$ is the unital augmentation ideal.  Let $K \subseteq Q$ be a
nonempty subset such that $f = \sum_{k \in K} \mu_k \ttt^k$ is
annihilated by~$\Iaug^1$, where $\mu_k \in \kk^*$ for all~$k \in K$.
It suffices to show that $K$ is a kernel of~$Q$.  But $\ttt^q f = f$
for all $q \in Q$ implies that $q + K = K$ for all $q \in Q$, which
implies both that $K$ is an ideal of~$Q$ (since $q + K \subseteq K$
for all~$q$) and also that $K$ is contained in every ideal of~$Q$
(since $K + q \supseteq K$).
\end{proof}

\begin{thm}\label{t:chain}
If $I_\ell \supset \cdots \supset I_0$ is a chain of distinct binomial
ideals in~$\kk[Q]$ inducing the same congruence on~$Q$, then $\ell
\leq 1$.  Moreover, if $\ell = 1$ then $I_1$ contains monomials and
$I_0$ does not: $I_0 = I_1 \cap \Iaug$ for an augmentation ideal
$\Iaug$ compatible~with~$I_1$.
\end{thm}
\begin{proof}
The first sentence follows from Proposition~\ref{p:nil}, as does the
statement about monomials when $\ell = 1$.  It remains to show that
$I_0 = I_1 \cap \Iaug$ if $\ell = 1$.  Set $I = I_0$.  Under the
grading of the quotient algebra $\kk[Q]/I$ by $\oQ = Q/\til_I$ the
dimension of the graded piece $(\kk[Q]/I)_\ol q$ as a vector space
over~$\kk$ is~$1$ for all $\ol q \in \oQ$.  Let $\ol\nil \in \oQ$ be
the nil, which exists because it is the class of all exponents on
monomials in~$I_1$.  Fix a nonzero element $\ttt^\ol\nil \in \kk[Q]/I$
of degree~$\ol\nil$.  Then $\ttt^q\ttt^\ol\nil = \lambda_q
\ttt^\ol\nil$ for each $q \in Q$.  Set $\Iaug = \<\ttt^q - \lambda_q
\mid q \in Q\>$.  Then $\Iaug \supseteq I$ by construction, but $\Iaug
\not \supseteq I_1$, since $I_1$ contains monomials and $\Iaug$ does
not.  Therefore $I_1 \supsetneq I_1 \cap \Iaug \supseteq I$, whence
$I_1 \cap \Iaug = I$, because $I_1/I = \<\ttt^\ol\nil\> \subseteq
\kk[Q]/I$ has dimension~$1$ as a vector space over~$\kk$ by
Proposition~\ref{p:nil}.
\end{proof}

\begin{example}
The ideal $I = \<x^2 - xy, xy - 2y^2\> \subseteq \kk[x,y]$ contains
monomials even when $\mathrm{char}(\kk) \neq 2$, because $I$ contains
both of $x^2y - xy^2$ and $x^2y - 2xy^2$, so $x^2y$ and~$xy^2$ lie
in~$I$.  However, Theorem~\ref{t:chain} implies that there is no
augmentation ideal compatible with~$I$.  Indeed, every binomial
ideal~$I'$ contained in~$I$ and inducing the same congruence
necessarily contains a binomial of the form $x^2 - \lambda xy$ and one
of the form $xy - \mu y^2$, so $I'$ contains both $x^2 - xy$ and $xy -
2y^2$ (and therefore $I' = I$) since~$xy \notin I$.
\end{example}

\section{Taxonomy of binomial ideals in monoid algebras}\label{s:binomial}

The concepts of primary, mesoprimary, primitive, prime, and toric
congruence from Definition~\ref{d:prim*} have precise analogues for
binomial ideals in monoid algebras.  As a small measure to aid the
reader with conflicting usages of the terms ``primary'' and ``prime'',
long since immovably set in the literature, the items in the following
definition are listed in the order corresponding exactly to
Definition~\ref{d:prim*}, as Theorem~\ref{t:til} makes precise; for
quick reference, consult the following table.

\medskip
\begin{center}
\begin{tabular}{c|c}
\ldots congruence on~$Q$ &
\ldots binomial ideal in $\kk[Q]$\\\hline
primary & cellular\\
mesoprimary & mesoprimary\\
primitive & primary\\
prime & mesoprime\\
toric & prime\\[-2ex]\multicolumn{2}{c}{}
\end{tabular}
\end{center}
The table explains our choice of terminology: ``mesoprimary'' sits
between the two occurrences of ``primary'', being stronger than one
and weaker than the other.

Our choice to work over fields that need not be algebraically closed
forces us to consider slight generalizations of group algebras.

\begin{defn}\label{d:twisted}
A \emph{twisted group algebra} over a field~$\kk$ is a $\kk$-algebra
that is graded by a group~$G$ and, after tensoring with the algebraic
closure $\ol\kk$, is isomorphic to the group algebra~$\ol\kk[G]$ via a
$G$-graded isomorphism.  A \emph{monomial homomorphism} from a monoid
algebra to a twisted group algebra takes each monomial to a
homogeneous element (possibly~$0$).
\end{defn}

\begin{example}\label{e:twisted}
The ring $R = \QQ[x]/\<x^3-2\>$ is not isomorphic to the group
algebra~$\QQ[G]$ for $G = \ZZ/3\ZZ$ over~$\QQ$, because no element
of~$R$ is a cube root of~$2$.  On the other hand, the element $y =
x\sqrt[3]2 \in R_\CC := R \otimes_\QQ \CC$ generates $R_\CC$, yielding
the presentation $R_\CC = \CC[y]/\<y^3-1\> \cong \CC[G]$.  Therefore
$R$ is a nontrivial twisted group algebra for the group $G = \ZZ/3\ZZ$
over the rational numbers~$\QQ$.
\end{example}

Generalizing the manipulations in Example~\ref{e:twisted} yields the
following.

\begin{prop}\label{p:twisted}
The twisted group algebras~$R$ over~$\kk$ (for a finitely generated
group $G$) are precisely the quotients of Laurent polynomial rings
over~$\kk$ by binomial~ideals.
\end{prop}
\begin{proof}
Let $R$ be a twisted group algebra.  Every $G$-graded piece of~$R$ has
dimension $\dim_\kk(R_g) = 1$ for all~$g \in G$, because this is true
after tensoring with~$\ol\kk$ by definition.  Thus $R$ admits a
binomial presentation $R \cong \kk[\NN^n]/I$
\cite[Proposition~1.11]{ES96}.  Every monomial~$\xx^u \in \kk[\NN^n]$
becomes invertible in~$R$ because every such monomial becomes
invertible in $R_\ol\kk := R \otimes_\kk \ol\kk$.  Therefore $R \cong
\kk[\ZZ^n]/I$ is a binomial quotient of a Laurent polynomial ring.  On
the other hand, the characterization of Laurent binomial ideals~$I$
\cite[Theorem~2.1]{ES96} (or see Lemma~\ref{l:groupalg}, below)
implies that there is a unique sublattice $L \subseteq \ZZ^n$ and
character $\sigma: L \to \kk$ such that $I = \<\xx^q - \sigma(q) \mid
q \in L\>$.  Over~$\kk$, not much more can be said, in general; but
over~$\ol\kk$, the fact that $\ol\kk{}^*$ is an injective abelian
group implies that $\sigma$ extends to a character $\rho: \ZZ^n \to
\ol\kk{}^*$.  If $y_i$ is the image in $R_\ol\kk$ of $\rho(x_i)x_i \in
\ol\kk[\ZZ^n]$, then naturally $R_\ol\kk = \ol\kk[y_1,\ldots,y_n] =
\ol\kk[G]$ for $G = \ZZ^n/L$.
\end{proof}

\begin{defn}\label{d:cellular}
A binomial ideal $I \subset \kk[Q]$ in the monoid algebra for a
monoid~$Q$ is
\begin{enumerate}
\item%
\emph{cellular} if every monomial $\ttt^q \in \kk[Q]/I$ is either
nilpotent or a nonzerodivisor.

\item%
\emph{mesoprimary} if it is maximal among the proper binomial ideals
inducing a given mesoprimary congruence (as per
Theorem~\ref{t:chain}).

\item%
\emph{primary} if the quotient $\kk[Q]/I$ has precisely one associated
prime ideal.

\item%
\emph{mesoprime} if $I$ is the kernel of a monomial homomorphism from
$\kk[Q]$ to a twisted group algebra over~$\kk$.

\item%
\emph{prime} if $\kk[Q]/I$ is an integral domain: $fg = 0$ in $\kk[Q]/I$
implies $f = 0$ or $g = 0$.
\end{enumerate}
\end{defn}

\begin{remark}
The maximality for a mesoprimary ideal $I \subseteq \kk[Q]$ amounts to
stipulating that the nil class of~$\til_I$ consists of elements $q \in
Q$ with $\ttt^q \in I$, the alternative being that none of these
monomials lie in~$I$ but differences of scalar multiples thereof~do.
\end{remark}

\begin{thm}\label{t:til}
Let $\alpha \in \{1,2,4\}$.  A binomial ideal~$I$ satisfies
part~$\alpha$ of Definition~\ref{d:cellular} if and only if its
induced congruence satisfies part~$\alpha$ of Definition~\ref{d:prim*}
and $I$ is maximal among proper ideals inducing that congruence.  For
$\alpha = 5$ the same holds if\/ $\kk$ is algebraically closed.  When
$\alpha = 3$ the implication Definition~\ref{d:prim*}.3 $\implies$
Definition~\ref{d:cellular}.3 holds in general, and the converse holds
if\/ $\kk$ is algebraically closed~of~characteristic~$0$.
\end{thm}
\begin{proof}
Fix a binomial ideal~$I$ and use notation as in
Definition~\ref{d:prim*} for $\til = \til_I$.  First we assume that
$I$ satisfies Definition~\ref{d:cellular}.$\alpha$ and show that $I$
satisfies Definition~\ref{d:prim*}.$\alpha$.
\begin{enumerate}
\item%
If a monomial $\ttt^q \in \kk[Q]/I$ is nilpotent or a nonzerodivisor
then the image $\ol q \in \oQ$ of~$q$ is nilpotent or cancellative,
respectively.

\item%
By definition.

\item%
Pick a presentation $\NN^n \onto Q$.  The kernel of the induced
surjection $\kk[\NN^n] \onto \kk[Q]$ is a binomial ideal
\cite[\S7]{gilmer}, so the preimage of~$I$ in $\kk[\NN^n]$ is a
primary binomial ideal $I' \subseteq \kk[\NN^n]$ such that
$\NN^n/\til_{I'} = \oQ$.  Replacing $I$ by~$I'$ if necessary, we
therefore may as well assume $Q = \NN^n$, since the definitions of
primitive congruence and primary ideal depend only on the quotients
$\ol\NN{}^n = \oQ$ and $\kk[\NN^n]/I' = \kk[Q]/I$.

Each binomial prime in $\kk[\NN^n] = \kk[x_1,\ldots,x_n]$ can be
expressed as a sum $\pp_b + \mm_J \subseteq \kk[\NN^n]$ of its
``binomial portion'' $\pp_b$, which is a prime binomial ideal
containing no monomials, and a monomial prime $\mm_J := \<x_i \mid i
\notin J\>$, which is generated by the variables whose indices are
\emph{not} contained in $J \subseteq \{1,\ldots,n\}$
\cite[Corollary~2.6]{ES96}; this deduction relies on the algebraically
closed hypothesis.  Rescaling the variables of $\kk[\NN^n]$ if
necessary, we can assume that the unique associated prime $\pp = \pp_b
+ \mm_J$ of~$\kk[\NN^n]/I$ is \emph{unital}---that is, $\pp_b$ is a
unital ideal.  Given that $\kk$ is algebraically closed of
characteristic~$0$, the $\pp$-primary condition on~$I$ implies that it
contains~$\pp_b$ \cite[Theorem~7.1$'$]{ES96}.  Therefore, replacing
$\kk[\NN^n]$ by $\kk[\NN^n]/\pp_b$ and $I$ by~$I/\pp_b$, we assume
that $Q$ is an affine semigroup and $\pp$ is generated by monomials.
The desired result now follows from \cite[Theorem~2.15 and
Proposition~2.13]{primDecomp} or \cite[Theorem~2.23]{abelSurvey}, the
latter being an equivalent statement that directly implies the
characterization of mesoprimary congruences in Corollary~\ref{c:one}.

\item%
If $\ol q$ is not nil then $\ttt^q \in \kk[Q]$ lies outside of~$I$, so
$\ttt^q$ maps to a nonzero monomial in the twisted group algebra,
whence $\ol q$ is cancellative because $G$ is cancellative.

\item%
When $I$ is a monomial prime in an affine semigroup ring, the result
is obvious.  But prime $\implies$ primary, so the reduction to that
case in part~3 applies.  Moreover, since $I = \pp$ contains~$\pp_b$
already, the characteristic~$0$ hypothesis is superfluous.
\end{enumerate}

For this half of the theorem, it remains to explain, for $\alpha \neq
2$, why $I$ is maximal among ideals inducing~$\til$.  For that, it
suffices by Theorem~\ref{t:chain} to show that $I$ contains a
mono\-mial if $\oQ$ has a nil~$\nil$.  For part~1 (the cellular case),
if $\ol q = \nil$, then by definition of~nil there is for each~$r \in
\NN$ a binomial $\ttt^q - \lambda_r\ttt^{rq} \in I$ for some
\mbox{$\lambda_r \in \kk^*$}, so $\ttt^q(1 - \lambda_r\ttt^{(r-1)q})
\in I$, whence $\ttt^q$ is a zerodivisor modulo~$I$ and thus nilpotent
modulo~$I$---say $\ttt^{rq} \in I$; then $\ttt^q - \lambda_r\ttt^{rq}
\in I \implies \ttt^q \in I$.  For part~3 (the primary case),
Theorem~\ref{t:chain} implies that~$I$ has at least two associated
primes---one or more arising from an augmentation ideal---if
maximality fails.  For part~4 (the mesoprime case), any monomial
$\ttt^q$ with $\ol q = \nil$ must lie in~$I$ because a group has no
nil.  For part~5 (the prime case), the maximality is a special case of
part~1, because prime $\implies$ cellular for binomial ideals.

Next, assuming that $I$ is maximal among the binomial ideals inducing
a congruence~$\til$ on~$Q$ satisfying
Definition~\ref{d:prim*}.$\alpha$, we prove that $I$ satisfies
Definition~\ref{d:cellular}.$\alpha$.  As a matter of notation, write
$\ott^q$ for the image of $\ttt^q$ in~$\kk[Q]/I$.  In all cases, if $q
\in Q$ is an element whose image $\ol q \in \oQ$ is nil, then $\ott^q
= 0$ by Theorem~\ref{t:chain}, using the maximality property of~$I$.
Consequently, if $q \in Q$ is nilpotent, then $\ott^q$ is nilpotent
in~$\kk[Q]/I$.
\begin{enumerate}
\item%
By the previous paragraph, if $q \in Q$, then either the monomial
$\ott^q$ is nilpotent or $\ol q$ is cancellative.  In the latter case,
multiplication by $\ott^q$ is injective on~$\kk[Q]/I$ because
$\kk[Q]/I$ is $\oQ$-graded and addition by~$\ol q$ is injective
on~$\oQ$.

\item%
By definition.

\item%
The quotient $\oQ$ satisfies the condition of Corollary~\ref{c:one}
in which the cancellative monoid $F \subseteq \oQ$ is an affine
semigroup.  Each orbit is a finite union of translates $\ol q + F$
because $\oQ$ itself is generated by~$F$ and finitely many
nilpotent elements.  The proof now proceeds as
\cite[Proposition~2.13]{primDecomp} does: owing to the partial order
on the set of orbits afforded by Lemma~\ref{l:poset}, the
$\oQ/F$-grading on $\kk[Q]/I$ induces a filtration by
$\kk[Q]$-submodules with associated graded module
$$%
  \mathrm{gr}(\kk[Q]/I) \cong \bigoplus_{F\text{-orbits}~T} \kk\{T\},
$$
where $\kk\{T\}$ is the vector space over~$\kk$ with basis~$T$.  The
isomorphism above is as $\kk[F]$-modules, or equivalently, as
$\kk[Q]$-modules annihilated by the kernel~$\pp_F$ of the
surjection $\kk[Q] \onto \kk[F]$, with the $\kk[F]$-module
structure on $\kk\{T\}$ induced by the $F$-action on~$T$.  Since
$\kk\{T\}$ is torsion-free as a $\kk[F]$-module, the direct sum
over~$T$ has only one associated prime, namely~$\pp_F$, whence
$\kk[Q]/I$ does,~too.

\item%
Set $\oQ' = \oQ \minus \{\nil\}$ if $\oQ$ has a nil, and $\oQ' = \oQ$
otherwise.  By maximality of~$I$, the quotient $\kk[Q]/I$ is
$\oQ'$-graded.  By part~1, every nonzero monomial $\ott^q \in
\kk[Q]/I$ is a nonzerodivisor.  Therefore $\kk[Q]/I$ injects into its
localization~$R$ obtained by inverting the nonzero monomials.  Any
presentation $\ZZ^n \onto G$ for the universal group~$G$ of~$Q$
results in a presentation $\kk[\ZZ^n] \onto \kk[G] \onto \kk[G]/I = R$
whose kernel is a binomial ideal.  Thus $R$ is a twisted group algebra
over~$\kk$ by Proposition~\ref{p:twisted}.

\item%
The argument for part~4 works in this case, too, but now $\oQ'$ is an
affine semigroup, so that $\ol\kk\otimes_\kk R$, and hence
also~$\kk[Q]/I$, is an integral domain.\qedhere
\end{enumerate}
\end{proof}

\begin{cor}\label{c:implies'}
For binomial ideals in~$\kk[Q]$, over an arbitrary field except
where~noted,
\begin{itemize}
\item%
prime $\implies$ mesoprime $\implies$ mesoprimary $\implies$ cellular;
and
\item%
prime $\implies$ primary $\implies$ mesoprimary $\implies$ cellular
(we only claim the second implication when $\kk$ is algebraically
closed of characteristic~$0$).
\end{itemize}
\end{cor}
\begin{proof}
Use Theorem~\ref{t:til}: if $I$ is maximal among binomial ideals
inducing a congruence from Definition~\ref{d:prim*}, then it is
maximal among binomial ideals inducing any of the weaker congruences
from Lemma~\ref{l:implies}.  This proves every implication except for
prime $\implies$ mesoprime, which a~priori requires $\kk$ to be
algebraically closed, if Theorem~\ref{t:til} is being applied.  But in
fact the implication holds in general, even though the quotient by a
prime binomial ideal~$I$ need not be an affine semigroup ring if $\kk$
is not algebraically closed.  This is a consequence of the stronger
statement in Theorem~\ref{t:binomPrime}, below.
\end{proof}

\begin{example}\label{e:primary=/=>mesoprimary}
In general a primary ideal need not be mesoprimary.  For instance,
$\<1-x^p, y-xy, y^2\>$ is primary in characteristic~$p$, but the
congruence it induces has two associated prime congruences regardless
of the characteristic.
\end{example}

\begin{remark}\label{r:primitive}
The given proof of the implication Definition~\ref{d:cellular}.3
$\implies$ Definition~\ref{d:prim*}.3 fails in characteristic~$p$,
even if the field~$\kk$ is algebraically closed, because primary
binomial ideals in characteristic~$p$ do not necessarily contain the
binomial part of their associated prime \cite[Theorem~7.1$'$]{ES96}.
\end{remark}

Theorem~\ref{t:til} implies the following result, which reflects the
table preceding Definition~\ref{d:cellular} homogeneously across all
of its rows, and shows that all of the richness in
Definition~\ref{d:cellular} is already exhibited by \emph{unital
ideals}: those generated by monomials and unital binomials.

\begin{cor}\label{c:til}
A congruence satisfies a part of Definition~\ref{d:prim*} if and only if
the kernel of the surjection $\kk[Q] \onto \kk[\oQ]^-\!$ satisfies the
corresponding part of Definition~\ref{d:cellular}.~~$\Box$
\end{cor}

\section{Monomial localization, characters, and mesoprimes}\label{s:char}

For arithmetic reasons, intersections of binomial ideals need not
reflect their combinatorics completely accurately.  The simplest
example is $\<x^2-1\> = \<x-1\> \cap \<x+1\>$, whose congruence fails
to equal the common refinement of the congruences induced~by $\<x-1\>$
and $\<x+1\>$.  Precise statements about relations between
combinatorics and arithmetic use characters on subgroups of the unit
groups of localizations of~$Q$.

Localizations of monoids at their prime ideals corresponds to
inverting monomials in their monoid algebras.

\begin{defn}\label{d:mmP}
For a prime ideal~$P \subset Q$, the corresponding monomial ideal
in~$\kk[Q]$ is $\mm_P = \<\ttt^p \mid p \in P\>$.
\end{defn}

\begin{remark}\label{r:maximal}
When $P$ is maximal, $\mm_P$ is the maximal proper $Q$-graded ideal in
the monoid algebra~$\kk[Q]$, but it need not be maximal in the set of
all proper ideals of~$\kk[Q]$.
\end{remark}

\begin{defn}\label{d:monomial-localization}
The \emph{monomial localization $\kk[Q]_P$ of\/ $\kk[Q]$ along~$P$} is
the monoid algebra of the localization~$Q_P$, arising by adjoining
inverses to all monomials outside of~$\mm_P$.  The \emph{monomial
localization} of any $\kk[Q]$-module $M$ \emph{along~$P$} is
\mbox{$M_P = M \otimes_{\kk[Q]} \kk[Q]_P$}.
\end{defn}

Localization behaves well upon passing between algebra and
combinatorics; it forms part of the justification for characterizing
algebraic notions, such as the concept of $I$-witness in the next
section, in combinatorial~terms.

\begin{lemma}\label{l:cong}
If $I \subseteq \kk[Q]$ is a binomial ideal inducing the
congruence~$\til$ on~$Q$ with quotient~$\oQ$, then for any monoid
prime $P \subset Q$, the quotient of~$Q_P$ modulo the congruence
induced by $I_P$ is the monoid localization~$\oQ_P$ from
Definition~\ref{d:localize}.
\end{lemma}
\begin{proof}
Immediate from the definitions.
\end{proof}

\begin{defn}\label{d:character}
For any group~$L$, a \emph{character} is a homomorphism $\rho: L \to
\kk^*$.  A character $\rho': L' \to \kk^*$ \emph{extends}~$\rho$ if
$L \subseteq L'$ is a subgroup and $\rho'(\ell) = \rho(\ell)$ for
$\ell \in L$.  The extension is \emph{finite} if $L'/L$ is finite.
\end{defn}

\begin{conv}
The domain $L$ is part of the data of a character $\rho: L \to \kk^*$;
that is, we simply speak of the character~$\rho$, and write $L_\rho$
if it is necessary to specify~$L$.
\end{conv}

\begin{defn}\label{d:P-mesoprime}
Fix a subgroup $K \subseteq G_P$ of the unit group $G_P$
at~$P$.  For any character $\rho: K \to \kk^*$, the
\emph{$P$-mesoprime of~$\rho$} is the preimage $\Irp$ in~$\kk[Q]$ of
the ideal
\begin{equation*}
(\Irp)_P
  := \<\ttt^u - \rho(u-v)\ttt^v \mid u-v \in K\> + \mm_P \subseteq \kk[Q]_P.
\end{equation*}
Viewing $P$ as implicit in the definition of~$\rho$, the symbol
$I_\rho$ refers to the preimage in~$\kk[Q]$ of the ideal $\<\ttt^u -
\rho(u-v)\ttt^v \mid u-v \in K\> \subseteq \kk[Q]_P$.
\end{defn}

\begin{defn}\label{d:saturate}
A subgroup $L \subseteq G$ of an abelian group is \emph{saturated}
in~$G$ if there is no subgroup of $G$ in which $L$ is properly
contained with finite index.  The \emph{saturation} $\sat(L)$ of~$L$
is the intersection of all saturated subgroups of $G$ that
contain~$L$.  For any prime number $p \in \NN$, the largest subgroup
of~$\sat(L)$ whose quotient modulo $L$~has~order
\begin{itemize}
\item%
a power of $p$ is $\sat_p(L)$.
\item%
coprime to~$p$ is $\sat'_p(L)$.
\end{itemize}
For $p=0$ set $\sat_p(L) = L$ and $\sat'_p(L) = \sat(L)$.
\end{defn}

The following implies, in particular, that the set of saturations of a
character is finite.  The statement is actually a slight
generalization of \cite[Corollary~2.2]{ES96}, in that the domain $L$
of~$\rho$ is allowed to be a subgroup of an arbitrary finitely
generated abelian unit group~$G_P$, and $\Irp$ is not an arbitrary
ideal in a finitely generated group algebra, but rather an ideal
containing the maximal monomial ideal in an arbitrary finitely
generated monoid algebra.  However, the generalization follows from
the original by working modulo the maximal monomial ideal and lifting
to any presentation of~$G_P$, taking note that all of the characters
in question are trivial on the kernel of the presentation.

\begin{prop}[{\cite[Corollary~2.2]{ES96}}]\label{p:assofmeso}
Fix an algebraically closed field $\kk$ of characteristic $p \geq 0$.
Let $\rho: L \to \kk^*$ be a character on a subgroup $L \subseteq
G_P$, and write $g$ for the order of $\sat'_p(L)/L$.  There are $g$
distinct characters $\rho_1,\dots,\rho_g$ on $\sat'_p(L)$ that
extend~$\rho$.  For each $\rho_j$ there is a unique character
$\rho_j'$ on $\sat(L)$ extending~$\rho_j$.  There is a unique
character $\rho'$ that extends~$\rho$ and is defined on~$\sat_p(L)$.
Moreover,
\begin{enumerate}
\item%
$\sqrt{\Irp} = I_{\rho', P}$,
\item%
$\Ass(S/\Irp) = \{I_{\rho_j',P} \mid j = 1,\ldots,g\}$, and
\item%
$\Irp = \bigcap_{j=1}^g I_{\rho_j,P}$.
\end{enumerate}
\end{prop}

The following lemma is a variant of \cite[Lemma~2.9]{dmm}
and~\cite[Theorem 2.1]{ES96}.

\begin{lemma}\label{l:groupalg}
If\/ $\kk[\Phi]$ is the group algebra of a finitely generated abelian
group $\Phi$, then for any proper binomial ideal $I \subset \kk[\Phi]$
there is a subgroup $L \subseteq \Phi$ and a character $\rho: L \to
\kk^*$ such that $I = I_\rho$.
\end{lemma}
\begin{proof}
The binomial ideal is of the form $\<1-\lambda_u \ttt^u \mid u \in
\mathcal{U}\>$ for some finite $\mathcal{U} \subseteq \Phi$.  First off,
$\mathcal{U}$ is a subgroup of~$\Phi$ since
$1 - \lambda\mu\ttt^{u+v} = \mu\ttt^v(1 - \lambda\ttt^u) + (1 -
\mu\ttt^v)$ for all $\lambda,\mu \in \kk$, including $\lambda =
\lambda_u$ and $\mu = \lambda_v$.  The set $\mathcal{U}$ is closed
under inverses because $(1 - \lambda\ttt^u)/\lambda\ttt^u = -(1
-\ttt^{-u}/\lambda)$ when $\lambda \neq 0$, and $I \neq \kk[\Phi]
\implies \lambda_u \neq 0$.  The very same arguments show that the map
$\rho: \mathcal{U} \to \kk^*$ defined by $u \mapsto \lambda_u$ is a
homomorphism.
\end{proof}

\begin{defn}\label{d:qchar}
Fix a binomial ideal $I \subseteq \kk[Q]$.
\begin{enumerate}
\item%
The \emph{stabilizer} of an element $q \in Q$ along a prime ideal $P
\subset Q$ is the subgroup $K_q^P \subseteq G_P$ (sometimes denoted
by~$K_q$ if $P$ is clear from context) fixing the class of $q \in Q_P$
under the action from Lemma~\ref{l:action} for the
congruence~$\til_I$.
\item%
For $\ttt^q \not\in I_P$, the \emph{character (of~$I_P$) at $q$} is
the homomorphism \mbox{$\rho = \rqp: K_q^P \to \kk^*$} such that the
$\kk[G_P]$-module map $\kk[G_P] \to \kk[Q_P]/I_P$ taking $1 \mapsto
\ttt^q$ has kernel~$I_\rho$.
\item%
The ideal $\Iqp := \Irp \subseteq \kk[Q]$ is the \emph{$P$-mesoprime
of~$I$ at~$q$}.
\end{enumerate}
\end{defn}

\begin{remark}\label{r:qchar}
The homomorphism $\kk[G_P] \to \kk[Q_P]/I_P$ in Definition~\ref{d:qchar}.2
has kernel of the form~$I_\rho$ by Lemma~\ref{l:groupalg}.  Indeed,
the kernel is a~priori the binomial ideal $(I_P:\ttt^q) \cap
\kk[G_P]$, which is not the unit ideal in $\kk[G_P]$ because $\ttt^q$
lies outside of~$I_P$.
\end{remark}

Saturations of subgroups (Definition~\ref{d:saturate}) are more or
less combinatorial in nature.  Saturations of characters, on the other
hand, are more subtle, because arithmetic properties of the target
field~$\kk$ can enter.

\begin{defn}\label{d:charSaturate}
Fix a subgroup $L$ of an abelian group~$G$.  A character $\rho: L \to
\kk^*$~is
\begin{itemize}
\item%
\emph{saturated} if the subgroup~$L$ is saturated, and
\item%
\emph{arithmetically saturated} if $\rho$ has no finite proper extensions.
\end{itemize}
A \emph{saturation} of~$\rho$ is an extension of~$\rho$ to~$\sat(L)$.
\end{defn}

The importance of saturated characters has been demonstrated in
Proposition~\ref{p:assofmeso}, which required the algebraically closed
hypothesis.  Without it, the arithmetically saturated condition holds
sway, and equivalence of primality with saturation can break.

\begin{thm}\label{t:binomPrime}
If a binomial ideal in~$\kk[Q]$ over an arbitrary field\/~$\kk$ is
prime then it is a mesoprime $\Irp$ for an arithmetically saturated
character~$\rho$.  The converse holds if $\kk$ is algebraically
closed, and it can fail if not.
\end{thm}
\begin{proof}
Suppose that $\kk[Q]/I$ is a domain.  The ideal of monoid elements $p
\in Q$ such that $\ttt^p \in I$ is a monoid prime~$P$.  Replacing $Q$
with the monoid $Q \minus P$ and $I$ with its image in $\kk[Q \minus
P] = \kk[Q]/\<\ttt^p \mid p \in P\>$, it suffices to prove that $I =
I_\rho$ for an arithmetically saturated character when $Q$ is
cancellative and $I$ contains no monomials.  Since $\kk[Q]$ injects
into its localization $\kk[Q]_\nothing = \kk[\Phi]$ for the universal
group $\Phi = Q_\nothing$, and $I$ contains no monomials,
Lemma~\ref{l:groupalg} implies the existence of a subgroup $L
\subseteq \Phi$ and a character $\rho: L \to \kk^*$ such that $I =
I_\rho$.  It remains to show that $I_\rho$ is not prime if $\rho$ is
not arithmetically saturated.  Suppose $\sigma: K \to \kk^*$ properly
extends~$\rho$ to a subgroup $K \subseteq \sat(L)$.  Then $I_\sigma
\supsetneq I_\rho$.  By restricting $\sigma$ to a subgroup of~$K$ that
still properly contains~$L$, we can assume that $|K/L| > 1$ and one of
the following~occurs:
\begin{itemize}
\item%
$\kk$ has positive characteristic~$p$ and $|K/L|$ is a power of~$p$;
\item%
$\kk$ has positive characteristic~$p$ and $|K/L|$ is relatively prime
to~$p$; or
\item%
$\kk$ has characteristic~$0$.
\end{itemize}
Proposition~\ref{p:assofmeso} implies that in the first case, the
extension $\ol I_\sigma$ of~$I_\sigma$ to~$\ol\kk$ has the same
radical as the extension~$\ol I_\rho$, in which case $I_\rho$ itself
is not a radical ideal.  In the remaining two cases,
Proposition~\ref{p:assofmeso} implies that $\ol I_\rho = \ol I_\sigma
\cap \ol J$, with no associated prime of either intersectand
containing an associated prime of the other.  It follows that $I_\rho
= I_\sigma \cap J$, where $I_\sigma$ and $J := (I_\rho \mid I_\sigma)$
both properly contain~$I_\rho$, so $I_\rho$ is not~prime.

The $\kk = \ol\kk$ converse is implicit in
Proposition~\ref{p:assofmeso}, and anyway follows easily from
\cite[Theorem~2.1]{ES96}.  Example~\ref{e:arithsat} demonstrates
failure of the general converse.
\end{proof}

\begin{example}\label{e:arithsat}
The ideal $I_\rho \subset \QQ[x]$ for the character $\rho: 4\ZZ \to
\QQ^*$ defined by $\rho(4) = -4$ is $\<x^4 + 4\>$.  This ideal is not
prime because it factors as $\<x^4 + 4\> = \<x^2 - 2x + 2\> \cap \<x^2
+ 2x + 2\>$.  Nonetheless, $\rho$ is arithmetically saturated because
$x^4 + 4$ has no binomial factors of degree~$2$.
\end{example}

\begin{example}\label{e:twisted'}
The ideal $\<x^3 - 2\>$ in Example~\ref{e:twisted} is prime (by
Eisenstein's criterion, for example).  Therefore the character $\rho:
3\ZZ \to \QQ^*$ sending $3 \mapsto 2$ is arithmetically saturated,
viewing $3\ZZ$ as a subgroup of~$\ZZ$: any proper extension of~$\rho$
to a character $\ZZ \to \QQ^*$ would require a cube root~of~$2$.
\end{example}

\section{Coprincipal and mesoprimary components of binomial ideals}\label{s:components}

\begin{defn}\label{d:witness'}
Fix a binomial ideal $I \subseteq \kk[Q]$ inducing a congruence~$\til$
on~$Q$.
\begin{enumerate}
\item%
An element $w \in Q$ is an \emph{$I$-witness} for a monoid prime~$P$
if it is a $\til$-witness~for~$P$ or if $P = \nothing$ is the empty
monoid ideal and $I$ contains no monomials.
\item%
An element $w \in Q$ is an \emph{essential} $I$-witness for a monoid
prime~$P$ if $w$ is a key $\til_I$-witness or some polynomial
annihilated by~$\mm_P$ in $\kk[Q_P]/I_P$ (Definitions~\ref{d:mmP}
and~\ref{d:monomial-localization}) has $\ttt^w$ minimal (under Green's
preorder) among its nonzero monomials.
\item%
If $\Irp$ is the $P$-mesoprime of~$I$ (Definition~\ref{d:qchar}) at
some $I$-witness $w$ for~$P$, then $w$ is an \emph{$I$-witness
for~$\Irp$}.
\item%
$\Irp$ is an \emph{associated mesoprime} of~$I$ if there is an
essential $I$-witness for~$\Irp$.
\end{enumerate}
\end{defn}

\begin{lemma}\label{l:ess}
Every essential $I$-witness for~$P$ is an $I$-witness for~$P$.
\end{lemma}
\begin{proof}
Assume that $f \in \kk[Q]$ such that $\mm_P f \subseteq I_P$.  Let $m
= \lambda\ttt^w$ be a term of~$f$ (that is, a nonzero constant times a
monomial) minimal under Green's preorder on~$Q_P$ restricted to the
terms of~$f$.  Fix a nonunit monoid element~$p \in Q_P$.  Since
$\ttt^p f \in I_P$, the term $\ttt^p m$ must equal, modulo~$I_P$, some
sum of terms whose monomials $\ttt^{p+a}$ have $\ttt^a$ appearing with
nonzero coefficient in~$f$.  It follows that $\ttt^p m$ shares its
$\oQ_P$-graded degree with at least one of these monomials
$\ttt^{p+a}$, where $\oQ = Q/\til_I$.  Thus $w$ is a witness by
minimality of~$m$: at least one of the elements~$a$ is an aide for
$w$~and~$p$.
\end{proof}

\begin{example}\label{e:ess-wit}
If $I = \<y-x^2y,y^2-xy^2,y^3\>$ is the binomial ideal from
Example~\ref{e:pure-diff}.5, then $\Irp = \<x^2-\lambda,y\>$ for $P =
\<e_y\>$, $\rho: \<(2,0)\> \to \kk^*$ defined by $\rho(2,0) = \lambda$
induces the associated prime congruence of~$\til_I$ for any
$\lambda\in\kk^*$.  The monomial $x^ay \in \kk[x,y]$ is a witness for
any $a \in \NN$, and it lies in one of two possible essential witness
classes, depending on the parity of~$a$; see the figure in
Example~\ref{e:pure-diff}.  However, only $\lambda = 1$ gives the
associated mesoprime itself, as opposed to merely inducing its
congruence.
\end{example}

\begin{lemma}\label{l:finite-ess}
Every binomial ideal in~$\kk[Q]$ has only finitely many essential
witnesses.
\end{lemma}
\begin{proof}
Theorem~\ref{t:finiteAss} takes care of key witnesses, so it is enough
to treat witnesses arising from annihilation by~$\mm_P$.  As~$Q$ has
finitely many prime ideals, it suffices to bound the number of
essential witnesses for a fixed prime ideal~$P$.  By definition,
$\mm_P$ annihilates the $\kk[Q_P]$-submodule of $\kk[Q_P]/I_P$
consisting of polynomials giving rise to essential $I$-witnesses.
Hence the $\kk[Q_P]$-submodule in question is finitely generated over
$\kk[G_P] = \kk[Q_P]/\mm_P$, so only finitely $G_P$-orbits of
(exponents on) monomials are~involved.
\end{proof}

\begin{remark}\label{r:allUnital}
All associated mesoprimes of a unital binomial ideal (generated by
differences of monomials with unit coefficients) are unital.
\end{remark}

\begin{remark}\label{r:nothing}
When $I$ contains no monomials, every monomial is an essential
$I$-wit\-ness for the empty monoid ideal~$\nothing \subset Q$.  The
condition that $\Irn$ be an associated mesoprime of~$I$ for some
(unique) character~$\rho$ is similar to the condition that $\nothing$
be assoc\-iated to the congruence~$\til$ induced by~$I$, but it is not
equivalent.  These conditions differ only when $I$ is minimal and not
maximal among binomial ideals inducing~$\til$
(cf.~Proposition~\ref{t:chain})---that is, when $\til$ has a nil class
but $I$ nonetheless contains no monomials---in which case $I$ has an
\mbox{associated~mesoprime~$\Irn$ but $\nothing$ is not associated
to~$\til$}.
\end{remark}

\begin{lemma}\label{l:mesoprime}
If $w$ is an $I$-witness for~$\Irp$, then the localization along~$P$
of the $P$-mesoprime $\Iwp$ of~$I$ at~$w$ satisfies $(\Iwp)_P =
(\Irp)_P = (I_P:\ttt^w) + \mm_P$.
\end{lemma}
\begin{proof}
The first equality is by Definition~\ref{d:qchar}.  For the second,
use Theorem~\ref{t:P-coprincipal}, which implies that $I$ and $I_P +
\ttt^w\mm_P$ have the same $P$-mesoprime at~$w$.  It follows that the
natural isomorphism $\kk[G_P] \to \kk[Q_P]/\mm_P$ induced by the
inclusion $\kk[G_P] \to \kk[Q_P]$ descends to an isomorphism
$\kk[G_P]/(I_P:\ttt^w) \cap \kk[G_P] \to \kk[Q_P]/\big((I_P:\ttt^w) +
\mm_P\big)$.  Now apply Remark~\ref{r:qchar}.
\end{proof}

\begin{remark}\label{r:assLattice}
If $Q=\NN^n$ and $I$ is unital, then all information about associated
mesoprimes is contained in the set of \emph{associated lattices} $L
\subset \ZZ^J$, each of which comes with an \emph{associated subset}
$J \subseteq \{1,\ldots,n\}$.  Indeed, a prime ideal~$P$ of $\NN^n$ is
the complement of a face $\NN^J$ of $\NN^n$, and specifying a prime
congruence on~$\NN^n$ amounts to choosing such a face along with a
lattice $L \subset \ZZ^J$.  To see why, first observe that
localization along~$P$ inverts the face, turning $\NN^n$ into $\ZZ^J
\times \NN^\oJ = G_P \times \NN^\oJ$.  Subsequently passing to the
quotient by a given prime congruence, the complement of the face maps
to nil, and the subgroup~$L$ is the stabilizer of any class under the
action of $\ZZ^J = G_P$ on the quotient.  We were led to associated
lattices (before the more general associated prime congruences) in
part by \cite[Theorem~8.1]{ES96}.  Although that theorem only covers
cellular cases, the upshot is that a collection of associated lattice
ideals contributes~associated~primes.
\end{remark}

\begin{remark}\label{r:ass}
When the domain~$K$ of a character $\rho: K \to \kk^*$ is a saturated
subgroup of~$G_P$, the ideal $\Irp$ is often an associated prime of a
binomial ideal~$I$ without being an associated mesoprime of~$I$.  The
reason is that the congruences induced by associated $P$-mesoprimes
are immediately visible in the congruence induced by~$I_P$, whereas
the associated primes of~$I$ usually induce coarser congruences
(larger congruence classes) than those visible.  The quintessential
example to consider is the lattice ideal~$I$ for an unsaturated
sublattice of~$\ZZ^n$: the lattice ideal for the saturation is an
associated prime of~$I$, but the unique associated mesoprime of~$I$ is
$I$ itself.
\end{remark}

\begin{prop}\label{p:one}
A binomial ideal $I \subseteq \kk[Q]$ is mesoprimary if and only if $I$
has exactly one associated mesoprime.
\end{prop}
\begin{proof}
If $I$ is mesoprimary then it is cellular by
Corollary~\ref{c:implies'} and the congruence $\til_I$ is mesoprimary
by Definition~\ref{d:cellular}.  If $w$ is any witness (essential or
not) for the unique associated prime congruence and $I' = (I:\ttt^w)$
is the annihilator of the image of~$\ttt^w$ in~$\kk[Q]/I$, then
multiplication by~$\ttt^w$ induces an isomorphism $I_P + \mm_P \to
I'_P + \mm_P$, so every associated mesoprime of~$I$ is equal to $I +
\mm_P$.

On the other hand, assume that $I$ has only one associated mesoprime,
and that its associated monoid prime is $P \subset Q$.  The congruence
$\til$ induced by~$I$ is mesoprimary by Lemma~\ref{l:associated} and
Theorem~\ref{t:oneassprim}.  Either $I$ contains a monomial, in which
case it is already maximal among ideals inducing its congruence by
Theorem~\ref{t:chain}, or else $I$ contains no monomials, in which
case the unique associated monoid prime ideal is $P = \nothing$ by
definition.  When $P = \nothing$, if $I$ is not maximal then $\til$
has a witness for some monoid prime ideal other than~$\nothing$ by
Remark~\ref{r:nothing}, as $\til$ has an associated monoid prime but
$\nothing$ is not one of them.  Thus uniqueness of the associated
mesoprime implies~maximality.
\end{proof}

\begin{remark}\label{r:semifree'}
Building on Remark~\ref{r:semifree}, Proposition~\ref{p:one} says that
the character of~$I_P$ is the same at every nonzero monomial as soon
as it is the same at every essential witness monomial, and that is
what it means to be a mesoprimary ideal.
\end{remark}

\begin{defn}\label{d:P-mesoprimary}
Given a monoid prime $P \subset Q$, a mesoprimary binomial ideal
in~$\kk[Q]$ is \emph{$P$-mesoprimary} if the associated prime ideal of
its induced congruence is~$P$.
\end{defn}

The principal use of the following definition, which builds on the
notion of order ideal from Definition~\ref{d:orderideal}, concerns the
case where the set $\ww$ consists of a single witness.  The more
general case arises during the construction of mesoprimary
decompositions with as few components as possible
(Corollary~\ref{c:mesodecomp'}).

\begin{defn}\label{d:cogenerated'}
Fix a binomial ideal $I \subseteq \kk[Q]$, a prime $P \subset Q$, and
a finite subset $\ww \subseteq Q$.  The \emph{monomial ideal $\MwpI
\subseteq \kk[Q]$ cogenerated by~$\ww$ along~$P$} is generated by the
monomials $\ttt^u \in \kk[Q]$ such that $u$ lies outside of the order
ideal $Q_{\preceq w}^P$ cogenerated by~$w$ at~$P$
(Definition~\ref{d:orderideal}) under the congruence~$\til_I$ for
all~$w \in \ww$.
\end{defn}

\begin{defn}\label{d:Wwp}
Fix a binomial ideal $I \subseteq \kk[Q]$ and a finite set $\ww
\subseteq Q$ such that the $P$-mesoprime $\Iwp$ of~$I$ at~$w$
is~$\Irp$ for all $w \in \ww$.  The \emph{$P$-mesoprimary component
of~$I$ cogenerated by~$\ww$} is the preimage $\WwpI$ in~$\kk[Q]$ of
the ideal $I_P + I_\rho + \MwpI \subseteq \kk[Q]_P$.
\end{defn}

\begin{remark}\label{r:Wwp}
Comparing to Definition~\ref{d:cogenerated}, adding $\MwpI$ in
Definition~\ref{d:Wwp} joins all pairs of elements in $Q \minus
\Qlqp$, while adding~$I_\rho$ joins the pairs $(a,b) \in Q$ satisfying
conditions~(i) and~(ii) in Definition~\ref{p:coprincipal}.
\end{remark}

\begin{defn}\label{d:cogenerator}
A \emph{cogenerator} of a mesoprimary binomial ideal $I \subseteq
\kk[Q]$, or of the quotient $\kk[Q]/I$, is a monoid element that is a
cogenerator of the induced congruence.  A~\emph{monomial cogenerator}
is a monomial in $\kk[Q]$ whose exponent is a cogenerator.
\end{defn}

The nomenclature in Definition~\ref{d:Wwp} is justified by the
following result, which arithmetizes the combination of
Theorem~\ref{t:P-coprincipal} and Lemma~\ref{l:peak}.

\begin{prop}\label{p:cogenerated}
If\/ $\ww\!$ consists of $I$-witnesses for~$P$, then the ideal $\WwpI$
in Definition~\ref{d:Wwp} is mesoprimary with associated
mesoprime~$\Irp$.  Moreover, if $I$ induces $\til$ on~$Q$, then
$\WwpI$ induces the common refinement of the coprincipal components
$\til_w^P$ cogenerated by the elements in~$\ww$ along~$P$.  Every
cogenerator of~$\WwpI$ lies in~$\ww$.
\end{prop}
\begin{proof}
The claim has little content if $P = \nothing$, as then $\Irp = I_\rho
= I_P$, so assume \mbox{$P \neq \nothing$}.  Since $\WwpI$ contains
monomials by definition, it suffices by Theorem~\ref{t:chain} to
verify that $\WwpI$ induces the common refinement~$\app$ of
coprincipal congruences in question, given that $\app$ is mesoprimary
by Theorem~\ref{t:P-coprincipal} and Proposition~\ref{p:mesorefine}.

By construction (specifically, Definition~\ref{d:cogenerated};
cf.~Remark~\ref{r:Wwp}), the mesoprimary congruence~$\app$ refines the
congruence~$\app'$ induced by~$\WwpI$: the monomial ideal $\MwpI$ sets
all elements outside of the order ideal equivalent to one another, and
the generators of~$I_\rho$ carry out the remaining required
identifications.  The harder direction is showing that no more
relations are introduced.

Since $\WwpI$ is obtained from an extension to the
localization~$\kk[Q]_P$ along~$P$, we may as well assume that $Q =
Q_P$, so $P$ is the maximal ideal of~$Q$.  The congruences induced by
$I$ and~$I_\rho$ each individually refine the congruence~$\app$ (not
to be confused with~$\app'$ here); for~$I$ this is by
Theorem~\ref{t:coprincipal}, and for~$I_\rho$ this is by
Corollary~\ref{c:one} (see also Remark~\ref{r:semifree}).  Therefore
both of $I$ and~$I_\rho$ are ideals graded by~$Q/\app$.  We deduce
that $\WwpI$ is graded by~$Q/\app$ as well, since $\MwpI$ is a
monomial ideal and hence is automatically graded by~$Q/\app$.
Consequently, each non-nil congruence class of~$\app'$ is contained in
some congruence class of~$\app$.

It remains to treat the nil class of~$\app'$.  Assuming $a \in Q$ with
$\ttt^a \not\in \MwpI$, it suffices to show $\ttt^a \not\in \WwpI$.
Choose $w \in \ww$ with $a$ in the order ideal $\Qlwp = \Qlwp(\til)$,
which can be done by definition of~$\MwpI$.  Next pick $u \in Q$ such
that the images of $u + a$ and~$w$ in $Q/\app$ are Green's equivalent
to one another; this is possible by definition of the order
ideal~$\Qlwp$.  Use a double bar to denote passage from $Q$
to~$Q/\app$, so $\ol{\ol q} \in Q/\app$ is the image of~$q$ for any $q
\in Q$.  The choice of the character~$\rho$ was made precisely so that
the graded piece $(I)_{\ol{\ol q}}$ of the ideal~$I$ contains the
graded piece $(I_\rho)_{\ol{\ol q}}$ whenever $\ol{\ol q}$ is Green's
equivalent to~$\ol{\ol w}$ in~$Q/\app$.  This means that $I_\rho$ adds
no new relations to~$I$ in degree~$\ol{\ol q}$.  Since $\MwpI$ adds no
new relations to~$I$ in degree~$\ol{\ol q}$ by definition,
$\WwpI_{\ol{\ol q}} = (I)_{\ol{\ol q}}$ for $q = u + a$.  The class of
$u + a$ is not nil in~$Q/\til$ because the character of~$I_P$ at $u +
a$ is~$\rho$.  Hence $\ttt^a \not\in \WwpI$.

The final claim of the Proposition follows from Lemma~\ref{l:peak}.
\end{proof}

\begin{defn}\label{d:coprincipal'}
A binomial ideal is \emph{coprincipal} if it is mesoprimary and its
induced congruence is coprincipal.  A~\emph{coprincipal component
$W_q^P(I)$ of~$I$ cogenerated by~$q$ at~$P$} is a $P$-mesoprimary
component $W_{\{q\}}^P(I)$ cogenerated by a single element~$q$.
\end{defn}

\begin{cor}\label{c:coprincipal}
If $I \subseteq \kk[Q]$ is a binomial ideal and $w$ is an $I$-witness
for~$P$, then the coprincipal component of~$I$ cogenerated by~$w$
at~$P$ is a coprincipal binomial ideal.
\end{cor}
\begin{proof}
Immediate from Proposition~\ref{p:cogenerated} and the definitions.
\end{proof}

\begin{remark}\label{r:mesorefine}
It would be superb if intersecting any pair of mesoprimary ideals with
the same associated mesoprime resulted in another mesoprimary ideal.
More precisely, a direct binomial ideal analogue of
Proposition~\ref{p:mesorefine} would be desirable.  Unfortunately, the
binomial analogue is false in general: in $\kk[x,y]$, the intersection
of the mesoprimary ideals $\<x - 2y\> + \<x,y\>^3$ and $\<x - y\> +
\<x,y\>^3$ is not mesoprimary when $\mathrm{char}(\kk) \neq 2$; it is
not even a binomial ideal.  Heuristically, if $I_1$ and~$I_2$ are
mesoprimary ideals in~$\kk[Q_P]$ with associated mesoprime~$\Irp$,
then in each of $I_1$ and~$I_2$ there are ``vertical'' binomials
from~$I_\rho$, whose coefficients are dictated by the
character~$\rho$, and ``horizontal'' binomials conglomerating the
vertical fibers, with more arbitrary coefficients.  (The vertical and
horizontal directions in Examples~\ref{e:msri} and~\ref{e:pure-diff}
are reversed for aesthetic reasons; the usage here makes sense in
Examples~\ref{e:assprim}, \ref{e:samewitness}, \ref{e:non-ass},
\ref{e:z+1}, and~\ref{e:poset}.)  When the horizontal coefficients
from $I_1$ and~$I_2$ conflict, the intersection need not be binomial.

That said, the analogue of Proposition~\ref{p:mesorefine} is true once
control is granted over binomiality, and that comes for free when
$I_1$ and~$I_2$ both arise from a single ideal via sets of witnesses
as in Proposition~\ref{p:cogenerated}.  In that sense, the binomial
analogue of Proposition~\ref{p:mesorefine} is ``true enough'' for the
relevant aspects of the theory of mesoprimary decomposition to
succeed, namely Corollary~\ref{c:mesodecomp'}.
\end{remark}

\begin{remark}\label{r:construct}
The existence of a mesoprimary ideal inducing a given congruence is
automatic by Remark~\ref{r:unital}.  However, the question becomes
more subtle when a given associated mesoprime other than the unital
one is desired.  Roughly speaking, we do not know how to construct
mesoprimary ideals with given associated mesoprimes de~novo, although
by Proposition~\ref{p:cogenerated} we do know how to construct
mesoprimary ideals given the foundation of a binomial ideal to start
from.  More precisely, fix a monoid prime $P \subset Q$, a
$P$-mesoprimary congruence $\app$ on~$Q$, and a character $\rho: K \to
\kk^*$ on the stabilizer~$K$ of some element that is not nil in the
localization of $Q/\app\,$ along~$P$.  It would be convenient to say
that there exists a mesoprimary ideal $J$ inducing~$\app$ with
associated mesoprime~$\Irp$, but it is not clear to us whether this
should be true.  What guarantees existence in the cases we care about,
namely Proposition~\ref{p:cogenerated}, is the $I$-witnessed nature
of~$\app$: each $I$-witness prefers a particular character over all
others---the one it sees by virtue of it being an $I$-witness---and
that is the only one required for the theory of mesoprimary
decomposition.

In a different light, the problem is one of automorphisms.  The
associated mesoprime of any unital $P$-mesoprimary ideal~$I$
is~$I_{1,P}$ for the trivial character.  Suppose, for simplicity, that
the ground field~$\kk$ is algebraically closed.  Then, for any
mesoprime~$\Irp$, there is an automorphism of~$\kk[Q]$ taking
$I_{1,P}$ to~$\Irp$; this amounts to the feasibility of extending the
character $\rho: K \to \kk^*$ to the entire group~$G_P$ of units
of~$Q_P$.  To transform $I$ into a mesoprimary ideal with associated
mesoprime~$\Irp$, however, the character must be extended
appropriately to all of~$Q_P$, not just to~$G_P$.  It is not clear to
us whether issues of horizontal coefficients
(cf.~Remark~\ref{r:mesorefine}) can intervene, particularly when the
inclusion of $G_P$ into~$Q_P$ fails to~split.
\end{remark}

\begin{remark}\label{r:verticalchoice}
Independent of the existence question, it is not clear how to describe
the class of mesoprimary ideals inducing a given congruence and with a
given associated mesoprime.  Certainly, a solution to the problem in
Remark~\ref{r:construct} need not be unique.  For instance in the
nilpotent situation, the one parameter family $\<x-\lambda y , x^2,
xy, y^2\>$ (for $\lambda \neq 0$) consists of mesoprimary ideals over
the associated mesoprime $\<x,y\>$, all inducing the same congruence.
\end{remark}

\section{Mesoprimary decomposition of binomial ideals}\label{s:mesodecomp}

This section makes precise the sense in which mesoprimary
decomposition of congruences lifts to a parallel combinatorial theory
for binomial ideals in monoid algebras.

\begin{defn}\label{d:mesodecomp'}
Fix a binomial ideal $I \subseteq \kk[Q]$ in a finitely generated
commutative monoid algebra over a field~$\kk$.
\begin{enumerate}
\item%
An expression of~$I$ as an intersection of finitely many mesoprimary
ideals is a \emph{mesoprimary decomposition} if, for each prime $P
\subset Q$ and $P$-mesoprimary intersectand~$J$, the $P$-mesoprimes
of~$I$ and~$J$ at every cogenerator of~$J$ coincide.
\item%
The decomposition is a \emph{combinatorial} mesoprimary decomposition
if every cogenerator of every component~$J$ in the decomposition is an
essential $I$-witness.
\end{enumerate}
\end{defn}

\begin{thm}\label{t:mesodecomp'}
Fix a finitely generated commutative monoid~$Q$ and a field~$\kk$.
Every binomial ideal in the algebra $\kk[Q]$ admits a combinatorial
mesoprimary decomposition.
\end{thm}
\begin{proof}
Examples include those in Theorem~\ref{t:coprincipal'} and
Corollary~\ref{c:mesodecomp'}, below, where the finiteness of the
intersection in Theorem~\ref{t:coprincipal'} is
Lemma~\ref{l:finite-ess}.
\end{proof}

The use of all essential witnesses and not merely key witnesses in the
next result stems from the element~$f$ in the proof, which can have
more than two terms.  See also Example~\ref{e:charWitness}, which
shows that non-key witnesses can be necessary for the intersection of
the corresponding coprincipal components to be a binomial ideal.  On
the other hand, the restriction to essential witnesses instead of all
witnesses ensures finiteness of the number of intersectands, according
to Lemma~\ref{l:finite-ess}.

\begin{thm}\label{t:coprincipal'}
Fix a finitely generated commutative monoid~$Q$ and a field~$\kk$.
Every binomial ideal in the monoid algebra $\kk[Q]$ is the
intersection of the coprincipal components cogenerated by its
essential witnesses.
\end{thm}
\begin{proof}
Pick an element $f$ outside of~$I$.  The goal is to show that $f$ lies
outside of the coprincipal component of~$I$ cogenerated by some
essential witness.  First assume that $f$ lies in the monomial
localization $I_{P'}$ along every nonmaximal prime~$P'$.  Thus $f$ is
annihilated, modulo~$I$, by some power of the maximal monomial ideal
$\mm_P \subseteq \kk[Q]$.  Replacing $f$ by a monomial multiple
of~$f$, assume that $f$ is annihilated, modulo~$I$, by the entire
maximal monomial ideal; that is, assume $\mm_P f \subseteq I$.  By
Definition~\ref{d:witness'}, some essential $I$-witness~$w$ for~$P$ is
the exponent on a monomial $\ttt^w$ with nonzero coefficient in~$f$.
Minimality of~$w$ ensures that all terms of~$f$ other than~$\ttt^w$
itself vanish modulo~$W_w^P(I)$, whence $f \not\in W_w^P(I)$.

The argument just completed proves, in particular, the case where $Q$
has only one prime ideal.  Now assume that $Q$ has more than one prime
ideal.  By the argument already given, assume the image of~$f$ under
monomial localization along some nonmaximal monoid prime~$P$ lies
outside of~$I_P$.  Induction on the number of prime ideals of~$Q$
implies that the localized image of~$f$ lies outside of some
$P$-coprincipal component of~$I_P$.  By Definition~\ref{d:Wwp}, a
$P$-coprincipal component of~$I_P$ is the localization along~$P$ of a
$P$-coprincipal component of~$I$.  Lemma~\ref{l:localize} implies that
$f$ lies outside of that $P$-coprincipal component before
localization, as desired.
\end{proof}

\begin{lemma}\label{l:localize}
If $I$ is a $P$-mesoprimary ideal, then localization along a monoid
prime is either injective or~$0$ on $\kk[Q]/I$, with injectivity
precisely when the prime contains~$P$.
\end{lemma}
\begin{proof}
By Definition~\ref{d:prim*}, any $P$-mesoprimary congruence on~$Q$ is
$P$-primary, whence the quotient~$\oQ$ either injects into its
localization along the given prime (if the prime contains~$P$) or else
$\oQ$ becomes trivial upon localization (if some element
of~$P$---which is nilpotent in~$\oQ$---is inverted).
Lemma~\ref{l:cong} implies that the result for congruences lifts to
binomial ideals.
\end{proof}

Using Theorem~\ref{t:coprincipal'} and
Proposition~\ref{p:cogenerated}, one can find a mesoprimary
decomposition that minimizes the number of components by intersecting
all coprincipal components for a given associated mesoprime.

\begin{cor}\label{c:mesodecomp'}
Fix a finitely generated commutative monoid~$Q$ and a field~$\kk$.
\mbox{Every} binomial ideal in the monoid algebra $\kk[Q]$ admits a
combinatorial mesoprimary decomposition with one component per
associated mesoprime.
\end{cor}

\begin{remark}\label{r:stronger}
The existence of any mesoprimary decomposition---let alone a
combinatorial one as in Theorem~\ref{t:mesodecomp'}---is much stronger
than mere existence of a decomposition as an intersection of
mesoprimary ideals, essentially because of the phenomenon in
Remark~\ref{r:ass}.  The strength is particularly visible when the
field~$\kk$ is algebraically closed of characteristic~$0$.  In that
case, every binomial primary decomposition of~$I$ expresses $I$ as an
intersection of mesoprimary ideals by Corollary~\ref{c:implies'}, but
a mesoprime must honor stringent combinatorial conditions to be an
associated mesoprime of~$I$, and a mesoprimary ideal for an associated
mesoprime must honor stringent combinatorial conditions to be an
intersectand in a mesoprimary decomposition of~$I$.  The difference
between ordinary and combinatorial mesoprimary decompositions is a
relatively slight distinction among potential cogenerator locations:
in the ordinary case, $I$ is merely required to possess the correct
characters at the cogenerators of the intersectands, whereas in the
combinatorial case only certain intrinsically defined elements
possessing the correct characters from~$I$ are allowed as cogenerators
of components.
\end{remark}

\section{Binomial localization}\label{s:binloc}

Upon localization of a binomial quotient $\kk[Q]/I$ at a binomial
prime, some monomials become units and others are annihilated.  The
units are easy: if the prime is~$\Isp$, then the monomials outside
of~$\mm_P$ become units.  The question of which monomials die is much
more subtle.  There are two potential reasons that a monomial gets
killed upon ordinary localization (Theorem~\ref{t:binloc}): a
combinatorial one and an arithmetic one.  Combinatorially, a monomial
dies if its class under $\til_I$ points into~$P$
(Definition~\ref{d:intoP}); arithmetically, a monomial dies if the
character of~$I_P$ at it is incommensurate with~$\rho$
(Definition~\ref{d:incommensurate}).  These annihilations result from
the inversion of two different types of binomials: in the
combinatorial case the inverted binomials have one monomial outside
of~$\mm_P$, and in the arithmetic case the inverted binomials lie
along the unit group~$G_P$ locally at~$P$.  The relevant monomials die
because locally each becomes a binomial unit multiple of a binomial
in~$I$; see the proof of Theorem~\ref{t:binloc}.

\begin{defn}\label{d:intoP}
Given a prime $P \subset Q$, and a congruence $\til$ on~$Q$, the
congruence class of $q \in Q$ \emph{points into~$P$} if $q + p \sim q$
in the localization~$Q_P$ for some $p \in P$.
\end{defn}

\begin{lemma}\label{l:intoP}
Given a prime $P \subset Q$ and a congruence $\til$ on~$Q$, the set of
elements in~$Q$ whose class points into~$P$ is an ideal of~$Q$.
\end{lemma}
\begin{proof}
If $q + p \sim q$ then $u + q + p \sim u + q$ by additivity of~$\til$.
\end{proof}

\begin{defn}\label{d:MintoP}
The \emph{$P$-infinite ideal}~$M_\infty^P(\til) \subseteq Q$ for a
prime $P \subset Q$ and congruence~$\til$ on~$Q$ is generated by the
elements of~$Q$ whose classes point into~$P$.  If $\til = \til_I$ is
induced by a binomial ideal $I \subseteq \kk[Q]$, then $\MipI
\subseteq \kk[Q]$ is the corresponding \emph{$P$-infinite monomial
ideal}.
\end{defn}

\begin{remark}\label{r:MipI}
The terminology involving infinity stems from
\cite[Lemma~2.10]{primDecomp}, which concerns binomial localization at
a monomial prime of an affine semigroup ring: when the ambient
monoid~$Q$ is an affine semigroup, a class that points into~$P$ is
infinite.  The focus on monomial primes in affine semigroup rings
arises there because the field is algebraically closed of
characteristic~$0$ and the ideals to be localized are $\Irp$-primary
(and hence contain~$I_\rho$), so the binomial localization procedure
can be carried out in the affine semigroup ring $\kk[Q]/I_\rho$.
Definitions~\ref{d:intoP} and~\ref{d:MintoP} lift the picture from $(I
+ I_\rho)/I_\rho \subseteq \kk[Q]/I_\rho$ to $I + I_\rho \subseteq
\kk[Q]$ itself; but see Remark~\ref{r:MrpI}.
\end{remark}

\begin{lemma}\label{l:R}
Let $R$ be a set of characters on subgroups of the unit group~$G_P$
of~$Q_P$.  Given a binomial ideal $I \subseteq \kk[Q]$, the set $\{q
\in Q \mid$ the character $\rqp$ of~$I_P$ at~$q$ is not a restriction
of every character from~$R\}$ is an ideal of~$Q$.
\end{lemma}
\begin{proof}
The character $\rpqP$ of~$I_P$ at~$p + q$ is an extension of~$\rqp$.
\end{proof}

\begin{defn}\label{d:incommensurate}
Given a binomial ideal $I \subseteq \kk[Q]$ and a mesoprime~$\Irp$,
the \emph{incommensurate ideal} of~$I$ at~$\rho$ is the ideal $\MrpI
\subseteq \kk[Q]$ spanned over~$\kk$ by all monomials $\ttt^q$ such
that the character of~$I_P$ at~$q$ is not a restriction of~$\rho$.
\end{defn}

\begin{remark}\label{r:MrpI}
The condition for a monomial to lie in the incommensurate ideal is
phrased arithmetically, but in reality many monomials in it are there
for combinatorial reasons: if the domain of the character of~$I_P$
at~$q$ fails to be contained in the (saturation of) the domain
of~$\rho$---that is, if the stabilizer of the class of~$q$
in~$Q/\til_I$ is too big---then $q$ has no hope of being commensurate
with~$\rho$.  This type of combinatorial obstruction to
commensurability also contributes infinite classes in
\mbox{\cite[Lemma~2.10]{primDecomp}}.
\end{remark}

\begin{defn}\label{d:binloc}
The \emph{binomial localization} of $I \subseteq \kk[Q]$ at a binomial
prime~$\Isp$ is the sum $I + \MipI + \MspI \subseteq \kk[Q]$ of~$I$
plus its $P$-infinite and incommensurate~ideals.
\end{defn}

The point of this section is to compare the previous definition with
ordinary (inhomogeneous) localization of a $\kk[Q]$-module at a
binomial prime~$\Isp$, obtained by inverting all elements of $\kk[Q]$
outside of~$\Isp$.

\begin{thm}\label{t:binloc}
Given a binomial ideal $I \subseteq \kk[Q]$ over an arbitrary
field\/~$\kk$, the kernel of the localization homomorphism from
$\kk[Q]$ to the ordinary localization of\/ $\kk[Q]/I$ at a binomial
prime $\Isp$ contains the binomial localization of~$I$ at~$\Isp$.
\end{thm}
\begin{proof}
First suppose that the class of $q \in Q$ points into~$P$.  Pick $p
\in P$ such that $q + p \sim q$.  This congruence means that there is
a binomial $\ttt^q - \lambda\ttt^{q+p} = \ttt^q(1 - \lambda\ttt^p)$
in~$I$.  But $1 - \lambda\ttt^p$ lies outside of~$\Isp$ because its
image modulo~$\mm_P$ is already~$1$.  Therefore $1 - \lambda\ttt^p$ is
a unit in the ordinary localization of $\kk[Q]/I$ at~$\Isp$,
so~$\ttt^q$ is~$0$~there.

Next suppose that $\ttt^q \in \MspI$.  By definition, there is a
binomial $1 - \lambda\ttt^g$ for some $g \in G_P$ such that $\lambda
\neq \sigma(g)$ and $\ttt^q(1 - \lambda\ttt^g) \in I_P$.  The element
$1 - \lambda\ttt^g$ lies outside of~$\Isp$ by definition.  Therefore
the argument in the previous paragraph works in this case, too.  We
conclude that the binomial localization of~$I$ is contained in the
kernel.
\end{proof}

\begin{remark}\label{r:binloc}
How is Theorem~\ref{t:binloc} to be applied?  While the binomial
localization $I'$ of~$I$ at~$\Isp$ might not coincide with the kernel
of ordinary localization at~$\Isp$, it is always the case, by
Theorem~\ref{t:binloc}, that $I$ and~$I'$ have the same ordinary
localization at~$\Isp$.  Therefore, for the purpose of detecting
$\Isp$-primary components, $I'$ is just as good as~$I$ was in the
first place.  But the combinatorics of~$I'$ might be much simplified,
thereby clarifying the role of $\Isp$ in the primary decomposition
of~$I$.  See the proof of Theorem~\ref{t:assPrim} for a quintessential
example.
\end{remark}

\section{Primary decomposition of binomial ideals}\label{s:irreducible}

Passing from mesoprimary and coprincipal ideals and decompositions to
primary ideals and decompositions requires a minimal amount of
knowledge concerning primary decomposition of mesoprimary ideals
themselves.  To speak about binomial primary decomposition of binomial
ideals in~$\kk[Q]$ we are forced to assume, in appropriate locations,
that $\kk$ is algebraically closed (Example~\ref{e:arithsat}); we
write $\kk = \ol\kk$ in that case.  Doing so guarantees that each
binomial ideal $I \subset \kk[Q]$ has binomial associated primes
\cite[Theorem~6.1]{ES96}.  However, most of this section works for an
arbitrary ground field, so we are explicit about our hypotheses in
this section.  One reason is that the characterization of binomial
prime ideals (Theorem~\ref{t:binomPrime}) does not rely on properties
of~$\kk$: every binomial prime can be expressed as a sum $\pp + \mm_P$
in which $P \subset Q$ is a monoid prime ideal and $\pp$ is a binomial
ideal (unique and prime modulo~$\mm_P$, but not necessarily
in~$\kk[Q]$) that contains no monomials.

\begin{prop}\label{p:filter}
Fix an arbitrary field\/~$\kk$.  If $I \subset \kk[Q]$ is mesoprimary
with associated mesoprime~$\Irp$, and the localized quotient monoid
$\oQ_P = Q_P/\til_I$ has unit group~$G$, then (i)~localizing along~$P$
induces an injection $\kk[Q]/I \into (\kk[Q]/I)_P$, and
(ii)~$(\kk[Q]/I)_P$ has finitely many nonzero $(\oQ_P/G)$-graded
pieces, all isomorphic to~$(\kk[Q]/\Irp)_P$.  Conditions (i) and~(ii)
characterize mesoprimary ideals~$I$ with associated mesoprime~$\Irp$.
\end{prop}
\begin{proof}
The monomials outside of $\mm_P$ are nonzerodivisors on the quotient
modulo any $P$-mesoprimary ideal by definition; hence the
injection~(i).  Claim~(ii) and the statement about characterizing
mesoprimary ideals follow from Proposition~\ref{p:one} (see also
Definition~\ref{d:qchar}, Remark~\ref{r:qchar}, and
Lemma~\ref{l:mesoprime}).
\end{proof}

\begin{cor}\label{c:assofmesoprimary}
Fix an arbitrary field\/~$\kk$.  If $I \subset \kk[Q]$ is mesoprimary,
then the associated primes of~$I$ are exactly the minimal primes of
its unique associated mesoprime.  In particular, $I$ is primary if it
is mesoprimary and its associated mesoprime is prime.
\end{cor}
\begin{proof}
The partial order on the monoid $\oQ_P/G$ afforded by
Lemma~\ref{l:poset} induces a filtration of~$(\kk[Q]/I)_P$ by
$\kk[Q]_P$-submodules whose associated graded module is free of finite
rank---in fact isomorphic to~$(\kk[Q]/I)_P$ itself---as a module over
$(\kk[Q]/\Irp)_P$.
\end{proof}

\begin{remark}\label{r:semifree"}
Corollary~\ref{c:assofmesoprimary} says that, although one expects to
derive information about associated primes of~$I$ from the characters
at its witnesses, when~$I$ is mesoprimary the appropriate characters
appear at the identity $1 \in \kk[Q]$.  This is another
manifesta\-tion of semifreeness (Remark~\ref{r:semifree}), detailed in
the present case in Proposition~\ref{p:filter}.
\end{remark}

Primary decomposition of mesoprimary ideals reduces to that of
mesoprimes.

\begin{prop}\label{p:mesoprimcomp}
Fix $\kk = \ol\kk$.  Any mesoprimary ideal $I \subset \kk[Q]$ with
associated mesoprime $I_{\rho, P}$ has unique minimal primary
decomposition $I = \bigcap_\sigma (I + I_\sigma)$, if \mbox{$\Irp =
\bigcap_\sigma \Isp$} is the unique minimal primary decomposition
of~$I_{\rho, P}$ from Proposition~\ref{p:assofmeso}.
\end{prop}
\begin{proof}
Adding the binomials~$I_\sigma$ to the mesoprimary ideal~$I$ coarsens
its congruence to another mesoprimary one, so each ideal $I +
I_\sigma$ is mesoprimary, and hence primary by
Corollary~\ref{c:assofmesoprimary}.  The intersection $J =
\bigcap_\sigma (I + I_\sigma)$ obviously contains~$I$, and we need
that $J \subseteq I$, or equivalently that $I_\rho = \bigcap_\sigma
I_\sigma$ maps to~$0$ in the quotient $\kk[Q]/I$.  This is a
consequence of Proposition~\ref{p:filter}, completing the proof.
\end{proof}

\begin{remark}\label{r:coprincipal}
If $I$ is coprincipal in Proposition~\ref{p:mesoprimcomp}, then every
primary component there is a coprincipal ideal.  Indeed, The partially
ordered monoid of Green's classes that is used to detect (or
construct) coprincipal ideals is the same for $I$ and for $I +
I_\sigma$.
\end{remark}

The remainder of this section outlines the main consequences of
mesoprimary decomposition for primary decomposition.

\begin{thm}\label{t:primDecomp}
Fix a binomial ideal $I \subseteq \ol\kk[Q]$ over an algebraically
closed field\/~$\ol\kk$.  Refining any mesoprimary decomposition
of~$I$ by canonical primary decomposition of its components yields a
binomial primary decomposition of~$I$.  In characteristic~$0$, each
primary component in this decomposition induces a primitive congruence
on~$Q$.
\end{thm}
\begin{proof}
Proposition~\ref{p:mesoprimcomp} implies binomiality of the primary
decomposition.  For the final claim, it suffices to prove that every
component $I + I_\sigma$ in Proposition~\ref{p:mesoprimcomp} induces a
primitive congruence in characteristic~$0$.  But since $\sigma$ is a
saturation of~$\rho$, the quotient of~$Q_P$ modulo the congruence
induced by $I + I_\sigma$ is exactly the quotient of $Q_P/\til_I$ by
the torsion subgroup of its unit~group.
\end{proof}

\begin{remark}\label{r:canonical}
No choices are necessary to construct the coprincipal decomposition in
Theorem~\ref{t:coprincipal'} or the combinatorial mesoprimary
decomposition in Theorem~\ref{c:mesodecomp'}, and hence no choices are
necessary to construct the primary decomposition in
Theorem~\ref{t:primDecomp}: these decompositions are all canonically
recovered from essentially combinatorial data---a set of witnesses and
monoid primes, plus the congruence induced by the binomial
ideal---just as in the monomial case.  Canonicality in the binomial
context, however, comes at the price of non-minimality.  Some
redundancy can be eliminated using Section~\ref{s:false}, but without
arbitrary, unmotivated (and often symmetry-breaking) choices,
redundancy can stubbornly persist.  The reason is that the redundancy
is already inherent in the combinatorics; that is, it happens at the
level of monoids, congruences, and witnesses, before coefficients
enter the picture.  Note that by ``canonical'' we mean in the sense of
``determined without extra data or requirements''.  In contrast, Ortiz
\cite{ortiz59} uses the adjective ``canonical'' to refer to primary
decompositions that minimize a certain index of nilpotency.
Regardless of the name, Ojeda \cite{ojeda11} proves that the
components in Ortiz's ``canonical'' decompositions are binomial when
the original ideal is binomial, but these decompositions generally
differ from the ones here, which rely solely on intrinsic data.
\end{remark}

\begin{remark}
In positive characteristic~$p$, primary binomial ideals need not be
mesoprimary.  This feature of mesoprimary decomposition reflects its
freedom from characteristic.  For instance, according to Hasse's
local-to-global principle the ideal $\<x^p-\nolinebreak 1,\linebreak
y(x-1), y^2\>$ has no business being primary: in all but one
characteristic it has two or more associated objects that
accidentially coincide in characteristic~$p$.
\end{remark}

When the base field~$\kk$ is not algebraically closed, the binomial
ideal~$I$ need not possess a binomial primary decomposition over~$\kk$
(see Example~\ref{e:arithsat}, for instance), but it does have one
over the algebraic closure~$\ol\kk$.  One of our original motivations
for seeking a theory of mesoprimary decomposition was to gather
primary components in such a way that Galois automorphisms
(of~$\ol\kk$ over~$\kk$) permute them.  In particular, if two primes
are Galois translates of one another, then we wanted their
corresponding primary components to look combinatorially the same.

\begin{thm}\label{t:galois}
If the ideal~$I$ in Theorem~\ref{t:primDecomp} is defined over a
subfield\/~$\kk$ of its algebraic closure~$\ol\kk$, then the primary
decomposition there is fixed by the Galois group
$\mathrm{Gal}(\ol\kk/\kk)$.  More precisely, if $\pi \in
\mathrm{Gal}(\ol\kk/\kk)$ is a Galois automorphism and $C$ is one of
the primary components of~$I$ from Theorem~\ref{t:primDecomp}, then
$\pi(C)$ is another one of them.
\end{thm}
\begin{proof}
The Galois group fixes every mesoprimary component of~$I$, and the
primary decomposition of a mesoprimary ideal
(Proposition~\ref{p:mesoprimcomp}) is canonical.
\end{proof}

Our final result on the primary-to-mesoprimary correspondence shows
that, for general binomial ideals, every associated prime is detected
by an associated mesoprime.  For cellular binomial ideals, the
relationship between associated mesoprimes and associated primes is
even more perfectly precise.  The cellular case of the following
result over an algebraically closed field is \cite[Theorem~8.1]{ES96}
and its converse; the latter was stated and used without proof after
\cite[Algorithm~9.5]{ES96}.  First, a matter of notation.

\begin{defn}\label{d:P-cellular}
Fix a cellular binomial ideal $I \subset \kk[Q]$.  If $P \subset Q$ is
the prime ideal of exponents on monomials that are nilpotent
modulo~$I$, then $I$ is \emph{$P$-cellular}.
\end{defn}

\begin{thm}\label{t:assPrim}
Fix a binomial ideal $I \subseteq \kk[Q]$ over an arbitrary
field\/~$\kk$.
\begin{enumerate}
\item%
Each associated prime of~$I$ is minimal over some associated mesoprime
of~$I$.
\item%
If $I$ is cellular, then the binomial converse holds: every binomial
prime that is minimal over an associated mesoprime of~$I$ is an
associated prime of~$I$.
\end{enumerate}
\end{thm}
\begin{proof}
For part~1, apply Corollary~\ref{c:assofmesoprimary} to the components
of~$I$ under any mesoprimary decomposition from
Theorem~\ref{t:mesodecomp'}.

For the cellular converse, suppose that $I$ is $P$-cellular, and that
a binomial prime $\Isp$ is minimal over some associated
mesoprime~$\Irp$ of~$I$.  The submodule of~$\kk[Q]/I$ generated by a
witness for~$\Irp$ is isomorphic to a quotient $\kk[Q]/I'$ for a
binomial ideal~$I'$ all of whose witness characters are extensions
of~$\rho$.  After subsequently binomially localizing at~$\Isp$, the
only surviving characters are restrictions of~$\sigma$, and hence sit
between $\sigma$ and~$\rho$.  In particular, this is true for the
character at any given monomial~$\ttt^q$ such that $q$ is a
cogenerator of the induced congruence.  Such a monomial generates a
mesoprime submodule with $\Isp$ among its associated primes by
Corollary~\ref{c:assofmesoprimary}.  Therefore $\Isp$ is associated
to~$I'$, and hence to $I$ by Theorem~\ref{t:binloc}; see
Remark~\ref{r:binloc}.
\end{proof}

\begin{example}\label{e:nonuniqueAss}
Given an associated prime of~$I$ as in Theorem~\ref{t:assPrim}.1, the
associated mesoprime guaranteed by the theorem need not be unique.
This phenomenon is illustrated by Example~\ref{e:pure-diff}.5.  The
binomial prime $\<x - 1, y\>$ for the trivial character on the
$x$-axis $\NN \times \{0\}$ is associated to~$I$ and has two possible
choices of associated mesoprime, namely $\<x - 1, y\>$ and $\<x^2 - 1,
y\>$.  Combinatorially, the row of dots at height~$1$ consists of two
classes, each being the nonnegative points in a coset of an
unsaturated lattice, while the row of dots at height~$2$ comprise just
one class, the nonnegative points in a coset of the saturation.  In
general, when the group of units $G_P$ acts, there could be a whole
$G_P$-orbit of classes corresponding to an unsaturated subgroup~$K$,
and a higher $G_P$-orbit with an associated subgroup anything between
$K$ and its saturation.
\end{example}

\begin{example}\label{e:cellular-not-mesoprimary}
Unmixed (cellular) binomial ideals need not be mesoprimary.  Consider
the cellular binomial ideal $\<x^2-1, y(x-1), y^2\> \subset \kk[x,y]$.
It is not mesoprimary, but because its associated primes are
$\<x-1,y\>$ and $\<x+1, y\>$, it is unmixed (even primary if
$\mathrm{char}(\kk)=2$).  Consequently, the unmixed decompositions of
\cite[Corollary~8.2]{ES96} and \cite[Algorithm~A4]{ojeda00} do not
decompose this ideal and thus do not lead to mesoprimary---let alone
coprincipal---decompositions, even in cellular cases.
\end{example}

\section{Character witnesses and false witnesses}\label{s:false}

The set of $I$-witnesses in the arithmetic setting of a binomial
ideal~$I$ in a monoid algebra can be redundant in a manner that
parallels the redundancy of witnesses in the combinatorial setting of
monoid congruences.  In the combinatorial setting, some of the
redundancy is naturally eliminated by restricting to key witnesses; in
the arithmetic setting here, character witnesses
(Definition~\ref{d:false'}) play an analogous role.  For cellular
binomial ideals this is Theorem~\ref{t:false}.  Lifting to the general
(i.e., non-cellular) case is possible but would take us too far afield
to be included here.

\begin{defn}\label{d:coverExt}
Fix a binomial ideal $I \subset \kk[Q]$, an element $q \in Q$, and a
monoid prime ideal $P \subset Q$.  A \emph{$P$-cover extension at~$q$}
is an extension of the character $\rqp: K_q \to \kk^*$ of~$I_P$ at~$q$
to the character $\rpqP: K_{p+q} \to \kk^*$ at a $P$-cover $p+q$
of~$q$ (Defs.~\ref{d:cover} and~\ref{d:qchar}).
\end{defn} 

There can be many---even infinitely many---choices of minimal
generating sets for~$P$ (Remark~\ref{r:cover}), but just as in
Lemma~\ref{l:cover}, there are not too many $P$-cover extensions.

\begin{lemma}
In the situation of Definition~\ref{d:coverExt}, the set of $P$-cover
extensions at~$q$ is finite, in the sense that only finitely many
stabilizers $K_{p+q}$ occur, and only finitely many characters defined
on each stabilizer occur among the characters~$\rpqP$.
\end{lemma}
\begin{proof}
Let $\oQ$ be the quotient of~$Q$ modulo the congruence determined
by~$I$.  If the images of $p$ and~$p'$ are Green's equivalent
in~$\oQ$, then the stabilizers $K_{p+q}$ and $K_{p'+q}$ coincide, as
do the extensions to $\rpqP$ and~$\rho_{p'+q}^{\hspace{.2ex}P}$.  Now
apply Remark~\ref{r:cover}.
\end{proof}

\begin{defn}\label{d:false'}
Fix a prime $P \subset Q$, a $P$-cellular binomial ideal $I \subset
\kk[Q]$, and $w \in Q$.
\begin{enumerate}
\item%
The \emph{testimony} of~$w$ at~$P$ is the set $T_P(w)$ of $P$-cover
extension characters.
\item%
The testimony $T_P(w)$ is \emph{suspicious} if the intersection of the
corresponding mesoprimes equals the $P$-mesoprime $\Iwp$
(Definition~\ref{d:P-mesoprime}); that is, if $\Iwp = \bigcap_{\rho
\in T_P(w)} \!\Irp$.
\item%
A \emph{false witness} is an $I$-witness $w$ for~$P$ that is not
maximal (under Green's preorder) among $I$-witnesses for~$P$ and whose
testimony at~$P$ is suspicious.
\item%
An $I$-witness that is not false is a \emph{character
witness}.
\end{enumerate}
\end{defn}

\begin{remark}\label{r:false}
For algebraically closed $\kk = \ol\kk$, Definition~\ref{d:false'}.4
becomes transparent, as follows.  Minimal primary decompositions of
mesoprimes $\Irp$ (Proposition~\ref{p:assofmeso}) are easy and
canonical in that case: every saturated finite extension of~$\rho$
appears exactly once.  A finite intersection of mesoprimes~$\Isp$,
each containing~$\Irp$, equals~$\Irp$ when, among all of the saturated
finite extensions of the characters~$\sigma$, every saturated finite
extension of~$\rho$ appears at least once.  A character witness
for~$P$ with associated mesoprime $\Irp$ is a witness in possession of
a new character (a saturated finite extension) not present in its
testimony.  By the same token, a witness is false if it has no new
characters to mention: the set of characters in its testimony is
suspiciously complete.
\end{remark}

The relation between the different types of witnesses from monoid land
(key witnesses) and binomial land (character witnesses) is not as
strong as one may hope.  For example, a key witness can be a false
witness (Example~\ref{e:false}), and a character witness might not be
a key witness (Example~\ref{e:charWitness}).  It is also possible for
a non-key witness to be a false witness (Example~\ref{e:non-key}).
All of these examples are cellular binomial ideals.

\begin{example}\label{e:false}
Consider the ideal $I' = \<x(z-1),y(z+1),z^2-1,x^2,y^2\>$ from
Example~\ref{e:z+1} and let $P$ be the monoid prime of~$\NN^3$ such
that $\mm_P = \<x,y\>$.  Then $0 \in \NN^3$ is a key $I'$-witness
for~$P$ that is a false $I'$-witness: the $P$-mesoprimes at the
$P$-covers of~$0$ are $\<z-1\>$ and $\<z+1\>$, whose characters form
the complete set of saturated finite extensions of the character
for~$\<z^2-1\>$.  The testimony is suspicious because $\<z-1\> \cap
\<z+1\> = \<z^2-1\>$.  In contrast, $0 \in \NN^3$ is a character
$I$-witness for~$P$, where the ideal $I =
\<x(z-1),y(z-1),z^2-1,x^2,xy,y^2\>$ induces the same congruence
as~$I'$.
\end{example}

\begin{example}\label{e:charWitness}
In Definition~\ref{d:false'}, the intersection of the mesoprimes is
the analogue of intersecting the kernels of the cover morphisms in
Definition~\ref{d:witness}.  The necessity of allowing all (non-key)
witnesses as potential character witnesses stems from the phenomenon
in Example~\ref{e:irreducible} (the common refinement of the
congruences induced by $\<x-1\>$ and $\<y-1\>$ is trivial whereas the
intersection of these ideals not) but is better illustrated by $I =
\<x^2 -xy, y^2 -\nolinebreak xy, x(z-\nolinebreak 1), y(w-1), x^3\>
\subset \kk[x,y,z,w]$, which throws an extra generator~$x^3$ into the
ideal from Example~\ref{e:assprim}.\ref{e:witt2nonKey}.  In contrast
with that example, the extra monomial causes~$I$ to be cellular: the
primary congruence it induces has associated monoid prime $P = \<e_x,
e_y\>$.  But the $P$-prime congruence at the character $I$-witness $0
\in \NN^4$ remains trivial, being the common refinement of the
congruences induced by $\<z-1\>$ and $\<w-1\>$.  This trivial
$P$-prime congruence at~$0$ indicates a total lack of binomials in the
$\oQ$-degree~$0$ part of the intersection $\<z-1,x^2,y\> \cap
\<w-1,x,y^2\>$, but this lack is accompanied by non-binomial elements.
An additional intersectand, namely the prime ideal $\<x,y\>$ itself,
is required to enforce~binomiality.

In terms of Definition~\ref{d:false'}, the testimony consists entirely
of saturated but infinite extensions of the character of~$I_P$ at $0
\in \NN^4$.  Therefore no saturated finite extensions occur, in the
sense of Remark~\ref{r:false}, making $0 \in \NN^4$ a rather strong
character $I$-witness, even though it is not a key witness for the
congruence induced by~$I$.
\end{example}

\begin{example}\label{e:non-key}
Non-key witnesses can be false witnesses.  In Example~\ref{e:non-ass}
the origin is a false witness because $\<a^2-1, b-1\> \cap \<a-1,
b^2-1\> \cap \<ab-1, a-b\> = \<a^2-1, b^2-1\>$ exhibits suspicious
testimony.
\end{example}

\begin{defn}\label{d:mesodecomp"}
Fix a cellular binomial ideal $I \subseteq \kk[Q]$ in a finitely
generated commutative monoid algebra over a field~$\kk$.  A
mesoprimary decomposition of~$I$ is \emph{characteristic} if every
cogenerator for every mesoprimary component is a character
$I$-witness.
\end{defn}

\begin{thm}\label{t:false}
Fix $I$, a cellular binomial ideal.  $I$ admits a characteristic
mesoprimary decomposition.  In fact, $I$ is the intersection of the
coprincipal ideals cogenerated by its character witnesses.  More
generally, if $I$ is expressed as an intersection of coprincipal
components of~$I$, then any component cogenerated by a false witness
is~redundant.
\end{thm}

In particular, the components for false witnesses can be thrown out
(with their testimony) from the coprincipal decomposition in
Theorem~\ref{t:coprincipal'} for a cellular binomial~ideal.

\begin{proof}
$P$-cellular ideals have only finitely many Green's classes of
witnesses for~$P$, because their induced congruences have only
finitely many Green's classes to begin with by Lemma~\ref{l:poset}.
Therefore the intersection over character witnesses is finite.

Express $I$ as an intersection of mesoprimary components of~$I$
cogenerated by single witnesses, one of which is $W = W_w^P(I)$
cogenerated by a false witness~$w$.  Given an element $f \not\in W$,
we need $f$ to lie outside of the intersection~$I'$ of the other
components.  It suffices to show that $f$ lies outside at least one of
the other components.  To that end, there is no harm in localizing
along~$P$, because by Lemma~\ref{l:localize} if $f$ lies outside of a
coprincipal component after localizing then it does so before
localizing.  Henceforth, therefore, assume $P$ is the maximal monoid
ideal.  Furthermore, if $f'$ is a monomial multiple of~$f$ that
remains outside of~$W$, then concluding that $f' \not\in I'$ is
enough.  Therefore, replacing $f$ by a monomial multiple of~$f$,
assume $f$ is annihilated, modulo~$W$, by the entire maximal monomial
ideal.  Write $f = f_{\preceq w} + f_{\not\preceq w}$, where
$f_{\preceq w}$ is the sum of the terms of~$f$ whose exponents lie in
$w + G$ for $G = Q \minus P$, the unit group of~$Q$.

The first goal is to show that $f \in I' \implies f_{\preceq w} \in
I'$.  Let $v$ be any $I$-witness and set $W' = W_v^P(I)$.  When $w
\not\prec v$ in Green's preorder, it is automatic that $f_{\preceq w}
\in W'$, for then all monomials with exponents in $w + G$ lie in~$W'$.
Therefore assume $w \prec v$ and $f \in W'$.  The relation $w \prec v$
implies that $w$ is not nil modulo the congruence $\til_{W'}$ induced
by~$W'$, and consequently no term of $f_{\preceq w}$ has an exponent
that is congruent under $\til_{W'}$ to the exponent on a term of
$f_{\not\preceq w}$.  Therefore $f\in W'\implies f_{\preceq w}\in W'$.

We have reduced to showing that $f \not\in W \implies f \not\in I'$
when $f = f_{\preceq w}$, so assume $f = f_{\preceq w} \not\in W$.
For each generator $p \in P$, let $\sigma_p \in T_P(w)$ be the
corresponding $P$-cover extension character.  The crucial observation
is that, since $f$ is a sum over~$w + G$,
$$%
  \ttt^p f \in I\ \iff\ f \in W + I_{\sigma_p}.
$$
This equivalence holds by tracing through all of the definitions; it
implies that $\ttt^p f \in I$ for all generators $p \in P$ precisely
when
\begin{align*}
   f \in \textstyle\bigcap_p (W + I_{\sigma_p})
&= W + \textstyle\bigcap_p I_{\sigma_p}\\
&= W + I_w^P\\
&= W,
\end{align*}
where the first displayed equality is a consequence of
Proposition~\ref{p:filter}.  Since $f \not\in W$, it follows that
there is some generator $p \in P$ such that $\ttt^p f \not\in I$.  But
$\ttt^p f \in W$ by construction, so $\ttt^p f$ lies outside of some
other coprincipal component of~$I$, and hence so does~$f$ itself, as
desired.
\end{proof}

Where did cellularity enter the proof of the preceding proposition?
Beyond finiteness of witnesses, the conclusion that no term of
$f_{\preceq w}$ has an exponent congruent under $\til_{W'}$ to the
exponent on a term of $f_{\not\preceq w}$ would be false if $W'$ were
allowed to be a coprincipal component for a monoid prime strictly
contained in~$P$; see Example~\ref{e:non-cellular}.

\begin{example}\label{e:non-cellular}
Let $I = \<x^2-x\dot x,x\dot x-\dot
x^2,x^3,x^2y,z^2-1,x^2(z-1),y(z+1),y(x-\nolinebreak\dot x)\> \subseteq
\kk[x,\dot x,y,z]$.  (The variables $x$ and~$\dot x$ correspond to $x$
and~$y$ in Example~\ref{e:mesoprimary}.)  Then $\dot x$ is a false
$I$-witness monomial for the monoid prime~$P$ corresponding to $\mm_P
= \<x,\dot x,y\>$:
the character at~$\dot x$ is $z^2-1$, while at $x\dot x$ it is $z-1$
and at $y\dot x$ it is $z+1$.  Omitting the coprincipal component
$\<x^2,x\dot x,\dot x^2,y,z^2-1\>$ of~$I$ cogenerated by~$\dot x$ from
the coprincipal decomposition of~$I$ in Theorem~\ref{t:coprincipal'}
leaves $\<x^3,x^2-x\dot x,x\dot x-\dot x^2,y,z-1\> \cap \<x^2,x-\dot
x,z+1\>$, which is not a binomial ideal.  The element $f = x - \dot x$
has a monomial $\dot x = \ttt^w$ whose exponent~$w$ is congruent to
the exponent of $x = \ttt^q$ under $\til_{W'}$ for $W' = \<x^2,x-\dot
x,z+1\>$ even though $q$ and~$w$ are incomparable.  $W'$ is
$P'$-mesoprimary for $\mm_{P'} = \<x,\dot x\> \subsetneq \mm_P$.
\end{example}

\begin{remark}\label{r:essfalse}
One reason Theorem~\ref{t:false} restricts to the cellular case is the
automatic finiteness for witnesses.  In contrast, in
Section~\ref{s:components} the notion of essential witness does the
job by Lemma~\ref{l:finite-ess}.  In general, even modulo Green's
equivalence the set of $I$-witnesses can be infinite.  For example,
infinitness causes Proposition~\ref{p:cogenerated} to fail when $I =
\<x^2y-y^2x\>$ if one uses all $I$-witnesses for $\mm_P = \<x,y\>$.
The sets of essential and character witnesses do not coincide, because
of the false key witnesses in Example~\ref{e:false}, but it is
possible that every character $I$-witness could be an essential
$I$-witness.
\end{remark}

\begin{question}\label{q:redundant}
Are there redundant character witnesses?  How about key witnesses?
\end{question}

\section{Open problems}\label{s:open}

Beyond Question~\ref{q:redundant}, the results of this paper raise
other problems implicitly in the remarks, and still others that
constitute future research directions beyond the scope of this paper.
We collect some of these problems here.

\subsection{Intersections of binomial ideals}

\begin{prob}\label{p:bIntersect}
Characterize when an intersection of binomial ideals is binomial.
\end{prob}

Problem~\ref{p:bIntersect} was originally posed by Eisenbud and
Sturmfels \cite[Problem~4.9]{ES96}, who answered it in the reduced
situation \cite[Theorem~4.1]{ES96}.  In our language, that theorem
contains information about the associated prime ideals of the
congruence induced by a radical binomial ideal.  It is possible that
the general case could reduce to the radical case, by considering what
the congruence or the $P$-prime characters induced by the intersection
could possibly look like at each monoid element.  This type of
consideration underlies the definition of character witness
(Definition~\ref{d:false'}), where non-binomiality at specific monoid
elements would otherwise occur, without specifically throwing in
additional binomials, because of incompatibility of congruences or
characters arising from covers.

As a stepping stone to a full answer to Problem~\ref{p:bIntersect},
one might consider \cite[Problem~6.6]{ES96}: does intersecting the
minimal primary components of a binomial ideal result in another
binomial ideal?

\subsection{Choices of vertical coefficients}

Remarks~\ref{r:construct} and~\ref{r:verticalchoice} raise the
following.

\begin{prob}\label{p:verticalCoeff}
Characterize the mesoprimary ideals that induce a fixed mesoprimary
congruence with a fixed associated mesoprime.  In particular, decide
when the set of such mesoprimary ideals is nonempty.
\end{prob}

\subsection{Primary binomial ideals in positive characteristic}

Lack of knowledge conconerning the combinatorics of primary binomial
ideals in positive characteristic is an obstacle in our
investigations.  In particular we do not know exactly which primary
binomial ideals are mesoprimary.

\begin{prob}\label{p:p>0}
Characterize primary binomial ideals with nontrivial mesoprimary
decompositions.
\end{prob}

\subsection{Posets of mesoprimes}

\begin{prob}\label{p:assmeso}\vspace{-.36ex}
Characterize the posets of associated prime congruences of primary
congruences.
\end{prob}\vspace{-.36ex}

The problem could have been stated for arbitrary congruences, but then
every finite poset would be possible, because every finite poset
occurs as the set of associated primes of a monomial ideal (this is a
good exercise, but it follows from \cite{multIdeals}).
Problem~\ref{p:assmeso} is equivalent to characterizing posets of
associated mesoprimes of unital cellular binomial ideals.  Such posets
always possess a unique minimal element, represented by the identity
element of the finite partially ordered monoid of Green's classes in
Lemma~\ref{l:poset}.  When devising examples for the present paper, we
often used a technique to ``place'' associated mesoprimes at desired
locations, illustrated as follows.

\begin{example}\label{e:poset}
Let $\Delta \subsetneq \Gamma$ be simplicial complexes on
$\{1,\ldots,n\}$ and consider the polynomial ring in $2n$ variables $S
= \kk[x_1,\dots,x_n,y_1,\dots,y_n]$. For any $A \in \Gamma \minus
\Delta$ write $x_A := \prod_{i\in A}x_i$.  The poset of associated
mesoprimes of the cellular binomial ideal
$$%
    I_{\Gamma\minus \Delta} = \sum_{A \in \Gamma\minus\Delta}
  I_A + \<x_i^2 \mid i = 1,\dots, n\> \subset S
  \quad\text{ for }\quad I_A = \< x_A(y_i-1) \mid i\in A\>
$$
is isomorphic to $(\Gamma \minus \Delta) \cup \{\emptyset\}$.
\end{example}

\begin{remark}\label{r:filler}
The construction in the previous example is fairly general, and one
might hope that complete generality is possible.
In practice this problem will be about understanding what happens to
the partial order on~$\NN^n$ when passing to a quotient and under the
order-preserving map assigning to a witness its associated prime
congruence.
\end{remark}

\begin{remark}\label{r:poset}
Definition~\ref{d:mesoass} requires associated prime congruences to
appear at key witnesses.  Allowing arbitrary witnesses yields an
a~priori different notion of associated prime congruence: although the
$P$-prime congruence at an arbitrary witness for~$P$ agrees with the
$P$-prime congruence at some key witness, the key witness might be for
a monoid prime smaller than~$P$.  This phenomenon does not occur for
primary congruences, however, as they have only one associated monoid
prime.  Thus Problem~\ref{p:assmeso} would have the same answer if
Definition~\ref{d:mesoass} had allowed arbitrary witnesses.

Nonetheless, this line of thinking indicates that care must be taken
in lifting Problem~\ref{p:assmeso} to the arithmetic setting, where
Definition~\ref{d:witness'} requires associated mesoprimes to appear
at arbitrary witnesses, not at a subset of all witnesses.  For
instance, a $P$-mesoprime can be associated to an ideal even though it
only appears at a false witness; this occurs in both
Example~\ref{e:false} and Example~\ref{e:non-key}.  This idiosyncracy
in the definition of associated mesoprime motivates a new definition.
\end{remark}

\begin{defn}\label{d:truly}
An associated mesoprime of a binomial ideal~$I$ is \emph{truly
associated} if it is the $P$-mesoprime of~$I$ at a character
$I$-witness for~$P$.
\end{defn}

\begin{prob}\label{p:poset}
Characterize the posets of associated mesoprimes of cellular binomial
ideals.  Do the same for posets of truly associated mesoprimes.
\end{prob}

\begin{remark}
The family of posets referred to in (either version of)
Problem~\ref{p:poset} contains the family of posets in
Problem~\ref{p:assmeso} by Remark~\ref{r:poset} applied to the case of
unital binomial ideals.
\end{remark}

\subsection{Mesoprimary decomposition of modules}

Grillet \cite{grilletActions} shows how subdirect decompositions of
semigroups induce subdirect decompositions of sets acted on by
semigroups; see Remark~\ref{r:grilletActs}.  In a similar vein,
mesoprimary decomposition ought to extend to finitely generated monoid
actions.

\begin{prob}
Generalize mesoprimary decomposition of congruences to $Q$-modules.
\end{prob}

The generalization ought to parallel the manner in which ordinary
primary decomposition of ideals in rings extends to primary
decomposition of modules over rings.  In the arithmetic setting of
mesoprimary decomposition, however, even the first step of the
extension requires thought.

\begin{question}
What is a binomial module over a commutative monoid algebra?
\end{question}

A good theory of such modules should yield the desired generalization.

\begin{prob}
Extend mesoprimary decomposition to binomial $\kk[Q]$-modules.
\end{prob}

\subsection{Homological invariants of binomial rings}

The combinatorics of the free commutative monoid~$\NN^n$ gives rise to
formulas and constructions for all sorts of homological invariants
involving monomial ideals---Betti numbers, Bass numbers, free
resolutions, local cohomology, and so on---due to the $\NN^n$-grading;
see \cite{cca}.  Gradings by more general affine semigroups yield
formulas and constructions for local cohomology over affine semigroup
rings (with maximal support \cite{Ish87} as well as with more
arbitrary monomial support \cite{bassNumbers,injAlg}), and Betti
numbers for toric ideals \cite[Theorem~I.7.9]{Sta}, etc.  Having
identified the combinatorics controlling decompositions of binomial
ideals, the way is open to generalize monomial homological~algebra.

\begin{question}
Do there exist combinatorial (that is, monoid-theoretic) formulas for
local cohomology, Tor, and Ext involving binomial quotients of
polynomial rings?
\end{question}

\begin{remark}
In contrast, it is unclear to us whether combinatorial formulas for
local cohomology with binomial support should exist, partly because of
ill-behaved characteristic dependence; see
\cite[Example~21.31]{twentyfour}.
\end{remark}

As soon as there is some control over Betti tables, Boij--S\"oderberg
theory \cite{floystad} enters the picture.  There one decomposes the
Betti table $\beta(M)$ of a module $M$ over a polynomial ring~$S$ as a
rational linear combination of certain \emph{pure tables} $\pi_d$:
\begin{equation*}
\beta (M) = \sum a_d \pi_d.
\end{equation*}

\begin{question}
What combinatorics, if any, explains the coefficients $a_d$ of
$S/\Irp$ as an $S$-module when $\Irp$ is a mesoprime?
\end{question}

Even the special case of Boij--S\"oderberg theory for toric ideals is
currently open.

\subsection{Test sets in integer programming}

Let $A\in\ZZ^{d\times n}$ be an integer matrix.  An \emph{integer
program} is an optimization problem that seeks, for a given cost
vector $\omega\in\RR^n$, to maximize $\omega \cdot u$ over the integer
points in the polyhedron $\mathcal{F}_b = \{u \in \NN^n \mid Au = b\}$
for $b \in \NN A := A(\NN^n) \subseteq \ZZ^d$.  A solution to this
problem is the computation of a \emph{test set}: a set $\mathcal{B}$
of differences between points in~$\mathcal{F}_b$ such that for any
candidate solution~$u$ to the optimization problem, its optimality can
be tested by comparing it to $u + v$ for $v \in \mathcal{B}$.
Computing a Gr\"obner basis of the toric ideal
$$%
  I_A = \<\xx^u - \xx^v \mid u,v \in \NN^n \text{ and } Au = Av\>
$$
provides a simultaneous test set for all right-hand sides~$b$, but
this procedure may be computationally prohibitive.  The hope behind
the following problem is that for many~$b$ a test set is significantly
simpler than a Gr\"obner basis.
\begin{prob}\label{p:testset}
Fix a finite set $\mathcal{B} \subset \ker_\ZZ A$.
\begin{enumerate}
\item%
Characterize the multidegrees $b \in \NN A$ for which $\mathcal{B}$ is
a test set.
\item%
Quantify the failure of~$\mathcal{B}$ to be a test set in large
fibers~$\mathcal{F}_b$.
\end{enumerate}
\end{prob}
\noindent
Intuition for the second problem comes from the geometry of
mesoprimary components, or better yet, coprincipal components: their
thicknesses in various directions should provide bounds on how close
an integer point in~$\mathcal{F}_b$ can get to optimality
using~$\mathcal{B}$.  Indeed, starting at some $u \in \ZZ^n$ and
successively progressing to the (local) optimum achieved by moving
along vectors in~$\mathcal{B}$ is equivalent to normal form reduction
of $\xx^u$ using binomials in the ideal $I_\mathcal{B} = \<\xx^u -
\xx^v \mid u - v \in \mathcal{B}\>$.  Classes for the congruence
induced by~$I_\mathcal{B}$ can be thought of, roughly speaking, as
polyhedra of the form~$\mathcal{F}_b$ with bits (the ``skerries'' from
\cite[Section~1.1]{primDecomp}) eaten away from the boundary;
mesoprimary decomposition controls the missing boundary bits.

Diaconis, Eisenbud, and Sturmfels suggested---though not in the
presence of a cost vector---to systematically study lattice walks with
step set~$\mathcal{B}$ using primary decomposition of
$I_\mathcal{B}$~\cite{diaconisLattice}.  Given the unsuitability of
primary decomposition for combinatorial purposes, the method should be
updated to work with mesoprimary decompositions.  This is especially
true in the presence of unsaturated lattices among the minimal primes
of $I_\mathcal{B}$, in which case the combinatorial flavor of the
problem becomes clouded in the arithmetic (rather than combinatorics)
of binomial primary decomposition.

A first step toward Problem~\ref{p:testset} was developed
in~\cite{krs}.  There the authors study only the connectivity of
$\mathcal{F}_b$ as a function of the position of $b$ in the cone
$\QQ_+ A$.  Additionally all ideals there are radical, and
consequently the subtleties of mesoprimary decomposition play no role.


\end{document}